\documentclass[11pt,a4paper,english,reqno]{amsart}
\usepackage[margin=2.5cm,
]{geometry}

\usepackage{babel}
\usepackage[T1]{fontenc}
\usepackage{uniinput} %
\usepackage{csquotes}
\usepackage{comment}
\usepackage{enumerate}
\usepackage{enumitem}
\usepackage{amsmath,amssymb,amsthm}
\usepackage{thmtools} %
\usepackage{esint} %
\usepackage[only,llbracket,rrbracket]{stmaryrd}
\usepackage{nicefrac}
\usepackage[spaced=-1]{diffcoeff} %
\usepackage[most]{tcolorbox}
\usepackage{footnote}
\BeforeBeginEnvironment{tcolorbox}{\savenotes}
\AfterEndEnvironment{tcolorbox}{\spewnotes}

\usepackage{orcidlink}

\usepackage{tikz}
\usetikzlibrary{decorations.pathmorphing,decorations.pathreplacing,tikzmark,calc,angles,quotes}
\usepackage{tabto}
\usepackage{xparse}
\usepackage{ifthen}
\usepackage{mathtools}
\usepackage{centernot} %
\usepackage{float}
\usepackage[disable]{todonotes}

\NewDocumentCommand{\todu}{s O{5} m}
{%
	\IfBooleanTF{#1}%
	{\todo[color=orange!\fpeval{ #2 * 10 }]{#3}}%
	{\todo[inline,color=orange!\fpeval{ #2 * 10 }]{#3}}%

}
\NewDocumentCommand{\note}{s m}
{%
	\IfBooleanTF{#1}%
	{\todo[color=green!10]{#2}}%
	{\todo[inline,color=green!10]{#2}}%

}

\usepackage{environ}

\newlist{pfparts}{description}{1}
\setlist[pfparts,1]{%
	font=\normalfont\textsf,
	itemindent=2pt,
	wide,
	itemsep=0pt,topsep=2pt,
	labelsep=0.75ex
}

\NewDocumentCommand{\Lp}{m O{\RR^n}}{
	L^{#1}(#2)
}

\newcommand{\indicator}[1]{\mathbf{1}_{#1}}
\newcommand{\hd}[1]{\mathcal{H}^{#1}}
\newcommand{\grapha}{\tilde{\mathcal{A}}}
\newcommand{\G}{\mathcal{G}}
\newcommand{\VG}{\mathcal{VG}}

\numberwithin{equation}{section}

\newcommand{\NN}[0]{\mathbb{N}}
\newcommand{\ZZ}[0]{\mathbb{Z}}
\newcommand{\RR}[0]{\mathbb{R}}
\newcommand{\CC}[0]{\mathbb{C}}
\renewcommand{\epsilon}{\varepsilon}
\renewcommand{\phi}{\varphi}
\renewcommand{\subset}{\subseteq}
\renewcommand{\supset}{\supseteq}

\renewcommand{\angle}{\measuredangle}

\DeclareMathOperator{\osc}{osc}

\DeclareMathOperator{\supp}{supp}

\DeclareMathOperator{\dist}{dist}
\DeclareMathOperator{\excess}{excess}

\DeclarePairedDelimiter{\paren}{(}{)}
\DeclarePairedDelimiter{\brac}{[}{]} %
\DeclarePairedDelimiter{\cbrac}{\{}{\}} %
\DeclarePairedDelimiter{\abs}{\lvert}{\rvert}
\DeclarePairedDelimiter{\norm}{\lVert}{\rVert}
\DeclarePairedDelimiterX{\inp}[2]{\langle}{\rangle}{#1, #2} %

\newcommand{\weaklyto}{\rightharpoonup}

\renewcommand{\restriction}{\big|}
\renewcommand\Re{\operatorname{Re}}%
\renewcommand\Im{\operatorname{Im}}%

\setlist[enumerate,1]{label=(\roman*)}

\usepackage[
	backend=biber,
	style=alphabetic, %
	maxalphanames=3,
	minalphanames=3,
	giveninits=true,
]{biblatex}
\DeclareSourcemap{
	\maps[datatype=bibtex]{
		\map{
			\step[fieldsource=doi,final]
			\step[fieldset=isbn,null]
			\step[fieldset=url,null]
			\step[fieldset=urldate,null]
			\step[fieldset=issn,null]
			\step[fieldset=eprint,null]
		}
	}
}%

\usepackage{hyperref}
\hypersetup{hidelinks}
\usepackage[nameinlink]{cleveref}

\crefname{equation}{}{}

\makeatletter
\DeclareRobustCommand{\crefnosort}[1]{%
	\begingroup\@cref@sortfalse\cref{#1}\endgroup
}
\makeatother

\makeatletter
\newcommand{\proofstep}[1]{%
	\par%
	\addvspace{\medskipamount}%
	\textit{#1\@addpunct{.}}\enspace\ignorespaces
}
\makeatother

\declaretheoremstyle[
	bodyfont=\normalfont,
	headfont=\normalfont\bfseries,
	notefont=\normalfont, notebraces={(}{)},
]{myrem}
\declaretheoremstyle[
	headfont=\normalfont\bfseries,
	notefont=\mdseries\bfseries, notebraces={(}{)},
]{mydefi}
\declaretheorem[name=Definition, style=mydefi]{definition}
\numberwithin{definition}{section}

\declaretheorem[name=Theorem]{theorem}

\declaretheorem[name=Theorem, numbered=no]{theorem*}
\declaretheorem[name=Corollary,sibling=definition]{corollary}
\declaretheorem[name=Corollary, numbered=no]{cor*}
\declaretheorem[name=Proposition,sibling=definition]{proposition}
\declaretheorem[name=Proposition, numbered=no]{prop*}
\declaretheorem[name=Lemma,sibling=definition]{lemma}
\declaretheorem[name=Lemma, numbered=no]{lemma*}
\declaretheorem[name=Remark, sibling=definition, style=myrem]{rem}
\declaretheorem[name=Remark, style=myrem, numbered=no]{rem*}

\declaretheorem[name=Example, style=myrem, numbered=no]{ex*}

\tcolorboxenvironment{lemma}{
	colback=blue!5,
	borderline west={2pt}{0pt}{blue!80!white},
	breakable,
	skin=enhanced,
	frame hidden,
	boxrule=0pt,
}
\tcolorboxenvironment{theorem}{
	colback=green!10,
	borderline west={2pt}{0pt}{green!80!black},
	skin=enhanced,
	frame hidden,
	boxrule=0pt,}
\tcolorboxenvironment{corollary}{
	colback=green!5,
	borderline west={2pt}{0pt}{green!80!white},
	skin=enhanced,
	frame hidden,
	boxrule=0pt,}
\tcolorboxenvironment{proposition}{
	colback=green!5,
	borderline west={2pt}{0pt}{green!80!white},
	skin=enhanced,
	frame hidden,
	boxrule=0pt,}
\tcolorboxenvironment{definition}{
	colback=yellow!5,
	borderline west={2pt}{0pt}{yellow!80!white},
	skin=enhanced,
	frame hidden,
	boxrule=0pt,
	}

\title{The Usual Square Function on Weakly Flat Sets}
\author[Benjamin Jaye \and Tobias Wang]{Benjamin Jaye\,\orcidlink{0000-0001-6326-7392} \and Tobias Wang\,\orcidlink{0009-0007-8208-308X}}
\address{
  School of Mathematics,
  Georgia Institute of Technology,
  Atlanta,
  GA 30332,
  USA
}
\email{\href{mailto:bjaye3@gatech.edu}{bjaye3@gatech.edu}}
\email{\href{mailto:tobias.wang@gatech.edu}{tobias.wang@gatech.edu}}
\date{June 23, 2026}

\hypersetup{
    pdftitle={The Usual Square Function on Weakly Flat Sets},
    pdfauthor={Benjamin Jaye, Tobias Wang},
    pdfsubject={Square Functions, Rectifiability, Usual square function estimate},
}

\begin{document}

\begin{abstract}
  We study the usual square function estimate associated with the Cauchy single-layer kernel in the plane, without assuming Ahlfors-David regularity.
  We prove that a finite Radon measure with positive and finite upper density is rectifiable if it satisfies the usual square function estimate and a weak flatness condition.
  We also prove that, under the same finiteness and density hypotheses, the weak flatness condition follows when the support is contained in a locally two-sided NTA curve.
  As a corollary, rectifiability follows when the support is contained in a quasicircle.
\end{abstract}

\maketitle

\setcounter{tocdepth}{1}
\tableofcontents

\section{Introduction}

We identify $\RR^2$ with $\CC$.
Let $\mu$ be a finite Radon measure in the plane and set
\begin{align}
  E:=\supp \mu .
\end{align}
We say that $μ$ is \emph{$1$-rectifiable} if $μ$ is absolutely continuous with respect to $1$-dimensional Hausdorff measure (we write $μ\le \hd 1$), and $\supp μ$ is $1$-rectifiable, i.e. there are Lipschitz maps $f_i\colon ℝ \to ℝ$, such that
\begin{align}
  \hd1\paren[\Big]{\supp μ\setminus \bigcup_{i=1}^{∞} f_i(ℝ)} = 0.
\end{align}

In their groundbreaking study of uniform rectifiability (a quantitative version of rectifiability), see for instance \cite{davidAnalysisUniformlyRectifiable1993,davidSingularIntegralsRectifiable1991}, \citeauthor{davidAnalysisUniformlyRectifiable1993} introduced the usual square function estimate which is the following boundedness condition for the Cauchy single-layer kernel.
\begin{definition}[Usual square function estimate]\label{def:usfe}
  We say that $\mu$ satisfies the \emph{usual square function estimate} (USFE) if there is a constant $C_{\mathrm{USFE}}$ such that, for every $f\in L^2(\mu)$,
  \begin{align}\label{eq:intro-usfe}
    \int_{\CC\setminus E}
    \abs[\Big]{\int_E \frac{f(\xi)}{(z-\xi)^2}\dl{\mu(\xi)}}^2
    \dist(z,E)\dl{m_2(z)}
    \le C_{\mathrm{USFE}} \norm{f}_{L^2(\mu)}^2 .
  \end{align}
\end{definition}

Indeed, one of the principal results in \cite{davidAnalysisUniformlyRectifiable1993} states that, for an Ahlfors--David-regular (\emph{AD-regular}) measure $μ$, i.e. a measure satisfying $μ(B(x,r))\approx r$ for all $r>0$, USFE holds if and only if $\mu$ is uniformly rectifiable \cite[Theorem I.2.41]{davidAnalysisUniformlyRectifiable1993}.
This characterization has proved useful in free boundary problems and partial differential equations \cite{hofmannUniformRectifiabilityHarmonic2014b,hofmannUniformRectifiabilityHarmonic2014c}.
The purpose of this paper is to initiate the study of USFE without the presence of AD-regularity.

For a Radon measure $\mu$, write
\begin{align}\label{eq:intro-density-notation}
  \delta_\mu(x,r)=\delta_\mu(B(x,r)):=\frac{\mu(B(x,r))}{r},
  \quad\text{and}\quad
  \delta_\mu^*(x):=\limsup_{r\to0}\delta_\mu(x,r),
\end{align}
for the \emph{density} at $x$ and scale $r$ and the \emph{upper density} at $x$, respectively.
One might hope that, if $\mu$ has finite and positive upper density, then USFE forces rectifiability.
In other words, one might conjecture that the measure-theoretic density assumption
\begin{align}\label{eq:intro-positive-upper-density}
  0<\delta_\mu^*(x)<\infty
  \qquad\text{for }\mu\text{-a.e. }x
\end{align}
is enough to recover a non-AD-regular analogue of the David--Semmes theorem.  This is false.  In \cref{thm:finite-support-counterexample}, we construct a finite Radon measure satisfying \cref{eq:intro-positive-upper-density} and USFE, with support of finite $\hd1$ measure, which is not rectifiable.
Moreover, the set satisfies that the $\alpha$-numbers introduced by Tolsa (cf. \cite{tolsaUniformRectifiabilityCalderon2009}) tend to zero, which would be a natural measure theoretic weakening of weak flatness.
Thus, an additional assumption on the support of the measure is required for a rectifiability theorem involving the support-sensitive square function in \cref{eq:intro-usfe}.

We now introduce the support condition used in the first theorem.
For a closed set $F\subset\CC$ we define Jones' $β$-number
\begin{align}\label{eq:intro-beta-number}
  \beta_F(x,r)
  :=
  \inf_L \sup_{y\in F\cap B(x,3r)}\frac{\dist(y,L)}{r},
\end{align}
where the infimum is taken over all lines $L\subset\CC$ \cite{jonesRectifiableSetsTraveling1990}.

\begin{definition}[Weak flatness]\label{def:weak-flatness}
  We say that $\mu$ is \emph{weakly flat} if, with $E=\supp\mu$,
  \begin{align}\label{eq:intro-weak-flatness}
    \lim_{r\to0}\beta_E(x,r)\delta_\mu(x,r)=0
    \qquad\text{for }\mu\text{-a.e. }x\in E.
  \end{align}
\end{definition}
Note that weak flatness does not imply rectifiability even in the AD-regular case.
In that setting, the weak geometric lemma, a condition stronger than ours, does not imply rectifiability as exemplified by a variation of the von Koch-curve \cite[Section 20]{davidSingularIntegralsRectifiable1991}.
Weak flatness is not a standard necessary condition for rectifiability in the non-AD-regular setting: low-mass pieces of the support may destroy support flatness while carrying little measure.
Its role here is to rule out the support-dust pathology in \cref{thm:finite-support-counterexample}.

The first main result is the following non-AD-regular analogue of the David--Semmes implication for the usual square function.
\begin{theorem}[Weakly flat case]\label{thm:main-theorem}
  Let $\mu$ be a finite Radon measure on $\RR^2$ satisfying
  \begin{enumerate}
    \item \label{item:finite-positive-upper-density} $0<\delta_\mu^*(x)<\infty$ for $\mu$-almost every $x$, and
    \item \label{item:weak-flatness} $\mu$ is weakly flat.
  \end{enumerate}
  If $\mu$ satisfies USFE, then $\mu$ is $1$-rectifiable.
\end{theorem}

This result is also naturally compared with non-AD-regular rectifiability criteria based on Jones' square function.
\Citeauthor{pajotConditionsQuantitativesRectifiabilite1997}, \citeauthor{badgerTwoSufficientConditions2015}, and \citeauthor{azzamCharacterizationNrectifiabilityTerms2015} showed that rectifiability can be characterized or forced under suitable density hypotheses by the finiteness of square functions built from $\beta$-numbers \cite{pajotConditionsQuantitativesRectifiabilite1997,badgerTwoSufficientConditions2015,tolsaCharacterizationNrectifiabilityTerms2015,azzamCharacterizationNrectifiabilityTerms2015}.  The present paper concerns a different square function, namely the usual square function in \cref{eq:intro-usfe}.

The two-dimensional definition of USFE generalizes to higher dimensions (maintaining codimension $1$) via the fundamental solution of the Laplacian (cf. \cite[Definition I.2.38]{davidAnalysisUniformlyRectifiable1993}).
We also mention related higher-codimension square-function characterizations based on regularized distances, due to \citeauthor{davidSquareFunctionsNontangential2021} \cite{davidSquareFunctionsNontangential2021}.
Our focus here, however, remains on the classical codimension-one Cauchy/single-layer square function.

The second result gives a geometric situation in which weak flatness follows from USFE.  We use a local version of two-sided NTA geometry.  NTA domains were introduced by \citeauthor{jerisonBoundaryBehaviorHarmonic1982} \cite{jerisonBoundaryBehaviorHarmonic1982}; local two-phase NTA blow-up arguments of this type are standard in the harmonic-measure literature, for example in the work of \citeauthor{kenigBoundaryStructureSize2009} \cite{kenigBoundaryStructureSize2009}.

\begin{definition}[Locally two-sided NTA curve]\label{def:local-two-sided-nta}
  Let $\Gamma\subset\CC$ be closed and let $x\in\Gamma$.  We say that $\Gamma$ is \emph{locally two-sided NTA at $x$} if there are constants $c_x>0$, $C_x<\infty$, a radius $r_x>0$, and two disjoint open connected sets $\Omega_x^+,\Omega_x^-\subset B(x,r_x)\setminus\Gamma$ such that
  \begin{align}
    B(x,r_x)\setminus\Gamma=(\Omega_x^+\cup\Omega_x^-),
    \qquad
    \partial\Omega_x^+\cap B(x,r_x)=\partial\Omega_x^-\cap B(x,r_x)=\Gamma\cap B(x,r_x),
  \end{align}
  and both $\Omega_x^+$ and $\Omega_x^-$ satisfy the local corkscrew and local Harnack-chain conditions in $B(x,r_x)$ with constants $c_x,C_x$.
  Concretely, let $Ω_x$ denote either $Ω_x^+$ or $Ω_x^-$.
  The local corkscrew condition for $Ω_x$ means that whenever \(y\in\Gamma\), $t>0$ and
  \(B(y,t)\subset B(x,r_x)\), then $Ω_x$ contains a ball $B(a,c_x t)\subset\Omega_x\cap B(y,t)$.
  By the local Harnack-chain condition we mean\footnote{Any localized version of the many equivalent Harnack-chain conditions would work.} that for every \(\eta>0\) there are \(N_\eta,\rho_\eta>0\), such that whenever for some $b\in\CC$, $t>0$,
  \[
    B(b,C_x t)\subset B(x,r_x), \quad
    p,q\in \Omega_x\cap B(b,t),
    \quad\text{and}\quad
    \dist(p,\Gamma),\dist(q,\Gamma)\ge \eta t,
  \]
  then \(p\) and \(q\) can be joined by a Harnack chain $B_j = B(a_j,s_j)$, $j=1,...,N$ with
  \(B_j \subset \Omega_x\cap B(b,C_x t)\), $s_j \approx \dist(B_j, Γ)$, \(N\le N_\eta\), consecutive balls intersecting, $p\in B_1$ and $q\in B_N$.
\end{definition}

\begin{theorem}[Locally two-sided NTA case]\label{thm:local-nta-rectifiable}
  Let $\mu$ be a finite Radon measure on $\RR^2$ satisfying
  \begin{enumerate}
    \item $0<\delta_\mu^*(x)<\infty$ for $\mu$-almost every $x$, and
    \item $\supp\mu\subset\Gamma$, where $\Gamma$ is a closed set and locally two-sided NTA at $\mu$-almost every point of $\supp\mu$.
  \end{enumerate}
  If $\mu$ satisfies USFE, then $\mu$ is $1$-rectifiable.
\end{theorem}

\begin{corollary}[Quasicircle case]\label{cor:quasicircle-rectifiable}
  Let $\mu$ be a finite Radon measure on $\RR^2$ satisfying $0<\delta_\mu^*(x)<\infty$ for $\mu$-almost every $x$.  Assume that $\supp\mu$ is contained in a quasicircle $\Gamma$.  If $\mu$ satisfies USFE, then $\mu$ is $1$-rectifiable.
\end{corollary}

Indeed, by the planar theorem of Jones, quasicircles are exactly the Jordan curves whose two complementary components are NTA domains, with quantitative constants (cf. \cite[Theorem C]{jonesQuasiconformalMappingsExtendability1981} and the discussion therein).
Thus a global quasicircle is locally two-sided NTA at every point, and the corollary follows from \cref{thm:local-nta-rectifiable}.
We also stress that support containment in a Jordan curve alone is not sufficient; see \cref{rem:jordan-curve-version} after \cref{thm:finite-support-counterexample}.

\subsection*{Outline of the proof.}
The proof of \cref{thm:main-theorem} is reduced to a quantitative lemma at a fixed scale.  After decomposing the measure according to upper density, weak flatness, and local square-function size, it is enough to prove that a normalized piece of the measure with small weak flatness and small square-function energy has a fixed positive proportion lying on a Lipschitz graph.

The proof of this quantitative lemma follows the David--Semmes--L\'eger--Tolsa stopping-time and Fourier energy scheme and later square-function variants \cite{legerMengerCurvatureRectifiability1999,tolsaPrincipalValuesRiesz2008,jayeProofCarlesons$varepsilon^2$conjecture2021,jayeHuovinenTransformRectifiability2022}.
The good balls determine a Lipschitz graph, and the remaining part of the measure is divided into two exceptional pieces.  The first exceptional piece is small by density considerations.  The second is controlled by comparing the square function for $\mu$ with the corresponding square function on the constructed graph.  The usual square function on Lipschitz graphs has delicate dependence on the Lipschitz constant, so the comparison is carried out through a conical square function and through estimates for the operator applied to the constant function $1$.

The proof of \cref{thm:local-nta-rectifiable} reduces the locally two-sided NTA case to \cref{thm:main-theorem}.  If weak flatness fails, then on a density-controlled piece one finds bad scales at which both flatness and density are bounded from below.  Since $\mu$ is finite, the square function estimate may be applied to indicator functions of finite-measure sets, and it implies that the normalized square-function energy tends to zero at almost every such point.  Blowing up along the bad scales at a locally two-sided NTA point produces a non-zero limiting measure $\nu$ with at most linear growth, supported on the limiting boundary of a two-sided corkscrew configuration, and satisfying a \emph{Cauchy flatness} condition:
\begin{align}
  \int \frac{1}{(z-\xi)^2}\dl{\nu(\xi)}=0,
  \qquad \forall z\in\CC\setminus\supp\nu.
\end{align}

Cauchy-flat measures were introduced by David and Semmes in this context.
As part of their work in the AD-regular setting, they showed that every connected component of the complement of the support of an AD-regular Cauchy-flat measure is convex \cite[Proposition III.2.38]{davidAnalysisUniformlyRectifiable1993}.
Here we use more elementary means to understand Cauchy-flat measures without AD-regularity, but in the more specialized local geometric setup needed for the blow-up argument.
In particular, the appendix describes such Cauchy-flat measures when the complement of the support has finitely many corkscrew components: these complementary components are polygonal domains and the measure is a finite positive linear combination of arclength measure on their sides.
In the two-sided local blow-up, this forces the limiting boundary to be a line, contradicting the lower bound on the $\beta$-number.
Thus, weak flatness holds and \cref{thm:main-theorem} applies.

\subsection*{Outline of the paper.}
The paper is organized as follows.
\Cref{sec:main-lemma} proves the reduction of \cref{thm:main-theorem} to the quantitative main lemma.
\Cref{sec:graph-construction} constructs the approximating Lipschitz graph.
\Cref{sec:usf-lipschitz-graphs} proves square function estimates on Lipschitz graphs.
\Cref{sec:comparison} contains the localization and comparison arguments needed to control the second exceptional set.
Finally, we summarize the estimates for the two exceptional sets in \cref{sec:F1} and provide the proof of the main lemma.
\Cref{sec:blow-up} proves \cref{thm:local-nta-rectifiable} by a blow-up argument.
\Cref{sec:appendix} contains the auxiliary graph estimates used earlier in the paper,
\cref{app:proof-cauchy-flat-finite-components} proves the description of Cauchy-flat measures with finitely many complementary components,
and \cref{app:finite-support-dust-example} provides an example with support of finite length showing that a support-level condition cannot simply be removed from \cref{thm:main-theorem}.

\subsection*{Acknowledgements.}
We thank John Hoffman for a discussion at an early stage of this project.  Research supported in part by NSF grants DMS-2453251 and DMS-2049477.
This research was undertaken in part while B.J. was a Simons' Fellow.
T.W. was supported in part by the Department of Education's Graduate Assistance in Areas of National Need Award (P200A240169).

AI tools were used for proofreading.

\newpage
\section{Main Lemma}\label{sec:main-lemma}
In this \namecref{sec:main-lemma} we present the main lemma, an outline for its proof and the proof of \cref{thm:main-theorem}.

\begin{lemma}\label{thm:main-lemma}
  Let $μ$ be a Radon measure on $\RR^2$, $0\in\supp μ$, $B_0 :=B(0,1)$ and $F\subset 10B_0$ a compact set satisfying
  \begin{enumerate}
    \item $δ_{\mu}(B_0)=1$ and $μ(10B_0\setminus F) \le η$,
    \item  $μ(B(x,r))\le (1+δ_1)r$ for all $x\in F$ and $r\in(0,100)$,
    \item \label{item:weakly-flat-uniform} $β_{E}(x,r)δ_μ(x,r) \le λ_1$ for all $x\in F$ and $r\in(0,100)$,
    \item \label{item:usf-small} $\displaystyle \int_{z\in B(0,100)\setminus E}\abs[\Big]{ \int_{\CC} \frac{1}{(z-ξ)^2} \dl{μ(ξ)} }^2 \dist(z,E) \dl m_2(z) \le λ_2$,
    \item \label{item:main-lemma-usfe} the localized USFE bound
          \begin{equation}\label{eq:main-lemma-local-usfe}
            \int_{8B_0\setminus E}
            \abs[\Big]{\int_{10B_0\setminus F}\frac{1}{(z-ξ)^2}\dl{μ(ξ)}}^2
            \dist(z,E)\dl{m_2(z)}
            \le M μ(10B_0\setminus F),
          \end{equation}
    \item \label{item:main-lemma-tail} for some finite constant $C_{\operatorname{tail}}$,
          \begin{equation}
            \sup_{\zeta\in 8B_0}
            \int_{\CC\setminus 10B_0} \frac{1}{|\zeta-ξ|^2}\dl{μ(ξ)}
            \le C_{\operatorname{tail}} .
          \end{equation}
  \end{enumerate}
  where $E=\supp μ$.
  There exists an absolute constant $c_0>0$ such that if $δ_1,λ_1,λ_2$ are chosen small enough, and $η$ is small enough in terms of $M$ and $C_{\operatorname{tail}}$, then there exists a Lipschitz graph $γ$ such that $μ(B_0\cap F\cap γ)\ge c_0 μ(B_0)$.
\end{lemma}

Before we come to the proof of \cref{thm:main-theorem}, we record the following \namecref{thm:localized-square-function-differentiation}.

\begin{lemma}[Localized square-function differentiation]\label{thm:localized-square-function-differentiation}
    Let $μ$ satisfy USFE, set $E=\supp μ$, and let $G,H\subset E$ be Borel sets with $μ(G)<\infty$.
    Suppose that $H$ has local linear growth: there are constants $C_H,r_H>0$ such that
    \begin{align}\label{eq:local-linear-growth}
        μ(B(x,r))\le C_H r,\qquad \forall x\in H,
        \quad 0<r<r_H .
    \end{align}
    There exists a $μ$-nullset $O$ such that for every $x\in H\setminus O$ and every $n\in\NN$,
    \begin{align}\label{eq:localized-cutoff-energy-zero}
        \lim_{r\to0}\frac{1}{r}
        \int_{z\in B(x,nr)\setminus E }
        \abs[\Big]{\int_G \frac{1}{(z-ξ)^2}\dl{μ(ξ)}}^2 \dist(z,E)\dl{m_2(z)}=0 .
    \end{align}
\end{lemma}
\begin{proof}
    Since $μ(G)<\infty$, $\indicator{G}\in L^2(μ)$.  Hence USFE gives a finite measure
    \begin{align}
        σ_G(A):=\int_{A\setminus E}
        \abs[\Big]{\int_G \frac{1}{(z-ξ)^2}\dl{μ(ξ)}}^2 \dist(z,E)\dl{m_2(z)} .
    \end{align}
    Fix $n\in\NN$ and $λ>0$, and let
    \begin{align}
        A_{λ,n}:=\cbrac[\Big]{x\in H:
            \limsup_{r\to0}\frac{1}{r}σ_G(B(x,nr)\setminus E)\ge λ } .
    \end{align}
    Let $0<r^*<r_H/n$.  For each $x\in A_{λ,n}$ choose $r_x<r^*$ such that
    \begin{align}
        σ_G(B(x,nr_x)\setminus E)\ge λ r_x .
    \end{align}
    By the Besicovitch covering theorem, there is a countable subcollection of the balls $B(x,nr_x)$, say $B(x_a,nr_a)$, which covers $A_{λ,n}$ with bounded overlap.  Since $x_a\in H$ and $nr_a<r_H$,
    \begin{align}
        μ(B(x_a,nr_a))\le C_H n r_a .
    \end{align}
    Therefore
    \begin{align}
        μ(A_{λ,n})
         & \lesssim \sum_a μ(B(x_a,nr_a))
        \lesssim_{C_H,n} \sum_a r_a                                                          \\
         & \lesssim_{C_H,n,λ} \sum_a σ_G(B(x_a,nr_a)\setminus E) \\
         & \lesssim_{C_H,n,λ} σ_G(\cbrac{z\in\CC\setminus E:\dist(z,E)\le nr^*}) .
    \end{align}
    Letting $r^*\downarrow0$ gives $μ(A_{λ,n})=0$, because $σ_G$ is finite and the sets $\cbrac{z\in\CC\setminus E:\dist(z,E)\le nr^*}$ decrease to the empty set.
    Taking the union over $λ\in\mathbb Q_+$ and $n\in\NN$ proves the lemma.
\end{proof}

\begin{proof}[Proof of \cref{thm:main-theorem} using \cref{thm:main-lemma}]
  Let $E:= \supp μ$.
  Absolute continuity of $μ$ with respect to $\hd 1$ follows from the finite upper density $μ$-almost everywhere.
  Take an arbitrary Borel subset $\tilde E\subset E$ with $μ(\tilde E)>0$.
  Our goal is to show that there is a Lipschitz graph $γ$ such that $μ(\tilde E \cap γ)>0$.
  This implies rectifiability of $\supp μ$.

  Let $C_{\mathrm{USFE}}$ be the constant in the USFE for $μ$.  Since $μ$ is finite, the constant function $1$ belongs to $L^2(μ)$.  Thus USFE gives a finite square-function measure
  \begin{align}
    σ(A):=\int_{A\setminus E}
    \abs[\Big]{\int_E \frac{1}{(z-ξ)^2}\dl{μ(ξ)}}^2
    \dist(z,E)\dl{m_2(z)} .
  \end{align}
  The finiteness of $σ$ is used below to prove local square-function smallness.  The finiteness of $μ$ itself is used in the tail estimate.  No global upper linear growth is assumed.

  We now use a uniformization procedure to partition $\tilde E$ so that we have the desired control as stated in \cref{thm:main-lemma}.
  Given $δ>0$, for each $i\in \ZZ$, define
  \begin{align}
    E_i = \{x\in \tilde E : (1+δ)^{i}\le δ_μ^*(x) < (1+δ)^{i+1} \} .
  \end{align}
  Assumption \ref{item:finite-positive-upper-density} of \cref{thm:main-theorem} ensures $μ(\tilde E\setminus \cup_i E_i) =0$.
  For $j\ge 1$, denote
  \begin{align}
    E_{i,j} = \{x\in E_i : δ_{\mu}(x,r) \le (1+δ)^{i+2} \text{ if } 0<r<1/j \} .
  \end{align}
  Next, we set
  \begin{align}
    E_{i,j,k} := \{ x\in E_{i,j} : β_{E}(x,r) δ_{\mu}(x,r) < (1+δ)^{i-1} λ_1 \text{ if } 0<r\le 1/k \}.
  \end{align}
  Assumption \ref{item:weak-flatness} ensures that $μ(\tilde E\setminus \cup_{i,j,k} E_{i,j,k}) =0$.
  Finally, we denote by $E_{i,j,k,ℓ}$ the set
  \begin{align}
    \cbrac[\Big]{x\in E_{i,j,k} :
      \frac{1}{r} σ(B(x,100r))
      \le (1+δ)^{2i-2} λ_2
      \text{ if } 0<r\le 1/ℓ }.
  \end{align}
  Since $μ$ is finite, and $E_{i,j}$ has local linear growth in the sense of \cref{eq:local-linear-growth},
  we can apply \cref{thm:localized-square-function-differentiation} to $G := E$, $H:=E_{i,j,k}$ which shows that
  \begin{align}
    \lim_{r\to0}\frac{1}{r}
    σ(B(x,100r))=0, \quad \forall x\in E_{i,j,k}\setminus O,
  \end{align}
  where $O$ is a $μ$-nullset.
  Thus, $μ(\tilde E\backslash \bigcup_{i,j,k,\ell} E_{i,j,k,\ell})=0$.

  We \enquote{disjointify} the sets $(E_{i,j,k,ℓ})$ by choosing pairwise disjoint sets $\tilde E_{i,j,k,ℓ} \subset E_{i,j,k,ℓ}$ while maintaining
  $μ\paren[\Big]{ \tilde E \setminus \cup_{i,j,k,ℓ} \tilde E_{i,j,k,ℓ} } = 0$.
  Now fix $i,j,k,ℓ$ with $μ(\tilde E_{i,j,k,ℓ}) > 0$.
  We will apply \cref{thm:main-lemma} with a finite tail constant $C_{\operatorname{tail}}$ depending only on the fixed density level; $η$ is chosen small enough in terms of this constant and the scaled USFE constant.
  For each density point $x_0$ of $\tilde E_{i,j,k,ℓ}$, we choose $r_0<\min \{1/(100j),1/(100k), 1/ℓ \}$, small enough also for the tail estimate verified at the end of the proof, such that
  \begin{align}\label{eq:density-choice}
    μ(B(x_0,10r_0)\setminus \tilde E_{i,j,k,ℓ})
    < \frac{η}{20(1+δ)^3} μ(B(x_0,10r_0))
    \quad\text{and}\quad
    (1+δ)^{i-1} \le δ_{\mu}(x_0,r_0).
  \end{align}
  Choose a compact set $K\subset \tilde E_{i,j,k,ℓ}\cap B(x_0,10r_0)$ so that
  \begin{align}
    μ(B(x_0,10r_0)\setminus K) < \frac{η}{10(1+δ)^3} μ(B(x_0,10r_0));
  \end{align}
  this is possible by the density choice and the inner regularity of $μ$, after decreasing $η$ by an absolute factor.
  We wish to apply the main lemma to the translated and rescaled measure
  \begin{align}
    ν := \frac{(T_{x_0,r_0})_{\#} μ}{μ(B(x_0,r_0))},
    \quad T_{x_0,r_0}(y):=\frac{y-x_0}{r_0},
  \end{align}
  with $F:= T_{x_0,r_0}(K)$.
  This shows that there exists a Lipschitz graph $γ$ such that
  \begin{align}
    μ\paren[\Big]{\tilde E\cap (x_0+r_0 γ)}
     & \ge μ\paren[\Big]{B(x_0,r_0) \cap K\cap (x_0+r_0 γ)} \\
     & \ge c_0 μ(B(x_0,r_0)) > 0.
  \end{align}

  Let us verify that the assumptions are satisfied; here $δ>0$ is chosen so small that $(1+δ)^3\le 1+δ_1$.
  We denote $E' := T_{x_0,r_0}(E) = \supp ν$ and $μ_0 := μ(B(x_0,r_0))$.
  For any $x\in F$, we have $x_0+r_0x\in K\subset E_{i,j}$, and the lower bound in \cref{eq:density-choice} gives
  \begin{align}\label{eq:doubling-scale}
    μ(B(x_0+r_0x,rr_0)) \le (1+δ)^{i+2} (rr_0)
    \le r (1+δ)^3 μ_0, \quad\forall r\in(0,100).
  \end{align}
  Assumption (i) follows since $δ_{\nu}(B_0)=ν(B(0,1))=1$ and
  \begin{align}
    ν(B(0,10)\setminus F)
    = \frac{μ(B(x_0,10r_0)\setminus T_{x_0,r_0}^{-1}(F))}{μ_0}
    < \frac{η}{10(1+δ)^3} \frac{μ(B(x_0,10r_0))}{μ_0}
    \le η.
  \end{align}
  For $r\in(0,100)$ and $x\in F$, we have
  \begin{align}
    ν(B(x,r)) = \frac{μ(B(x_0+r_0x,rr_0))}{μ_0} \le (1+δ)^3 r
  \end{align}
  and
  \begin{align}
    β_{E'}(x,r) δ_{\nu}(x,r)
     & = β_{E}(x_0+r_0x,rr_0) \frac{μ(B(x_0+r_0x,rr_0))}{μ_0 r}             \\
     & \le β_{E}(x_0+r_0x,rr_0) δ_{\mu}(x_0+r_0x,rr_0) \frac{r_0}{μ_0}      \\
     & \le β_{E}(x_0+r_0x,rr_0) δ_{\mu}(x_0+r_0x,rr_0) (1+δ)^{1-i} \le λ_1,
  \end{align}
  since $x_0+r_0x\in K\subset E_{i,j,k}$.
  This verifies assumptions (ii) and (iii).
  For assumption (iv), we have
  \begin{align}
    \int_{B(0,100)\setminus E'} & \abs[\Big]{\int_{\CC} \frac{1}{(z-ξ)^2} \dl{ν(ξ)} }^2 \dist(z,E') \dl m_2(z)                                                   \\
                                & = \frac{r_0}{μ_0^2}\int_{B(x_0,100r_0)\setminus E} \abs[\Big]{\int_{\CC} \frac{1}{(y-ξ)^2} \dl{μ(ξ)} }^2 \dist(y,E) \dl m_2(y) \\
                                & \le (1+δ)^{2-2i} (1+δ)^{2i-2} λ_2
    \le λ_2,
  \end{align}
  since $x_0\in E_{i,j,k,ℓ}$ and we substituted $y=x_0+r_0 z$.
  The localized USFE bound \ref{item:main-lemma-usfe} follows by applying USFE to
  \begin{align}
    f=\indicator{B(x_0,10r_0)\setminus K}.
  \end{align}
  After the change of variables $y=x_0+r_0z$, we obtain
  \begin{align}
     & \int_{8B_0\setminus E'}
    \abs[\Big]{\int_{10B_0\setminus F}\frac{1}{(z-ξ)^2}\dl{ν(ξ)}}^2
    \dist(z,E')\dl{m_2(z)}                                                        \\
     & \qquad = \frac{r_0}{μ_0^2}
    \int_{B(x_0,8r_0)\setminus E}
    \abs[\Big]{\int_{B(x_0,10r_0)\setminus K}\frac{1}{(y-ξ)^2}\dl{μ(ξ)}}^2
    \dist(y,E)\dl{m_2(y)}                                                         \\
     & \qquad \le C_{\mathrm{USFE}}\frac{r_0}{μ_0^2}\, μ(B(x_0,10r_0)\setminus K)
    = C_{\mathrm{USFE}}\frac{r_0}{μ_0}\,ν(10B_0\setminus F)                       \\
     & \qquad \le C_{\mathrm{USFE}}(1+δ)^{1-i}ν(10B_0\setminus F).
  \end{align}
  Thus \ref{item:main-lemma-usfe} holds with a finite constant depending only on the fixed density level and the original USFE constant.

  It remains to verify the tail condition \ref{item:main-lemma-tail}.  Since $μ(\CC)<∞$, for $\zeta\in8B_0$,
  \begin{align}\label{eq:scaled-tail-condition}
    \int_{\CC\setminus 10B_0}\frac{1}{|\zeta-z|^2}\dl{ν(z)}
    = \frac{r_0^2}{μ_0}\int_{|y-x_0|>10r_0}
    \frac{1}{|x_0+r_0\zeta-y|^2}\dl{μ(y)}.
  \end{align}
  On the part where $10r_0<|y-x_0|<1/j$, the estimate $|x_0+r_0\zeta-y|\gtrsim |x_0-y|$, the upper-density bound defining $E_{i,j}$, and the lower bound $μ_0\ge (1+δ)^{i-1}r_0$ give, by dyadic annuli,
  \begin{align}
    \frac{r_0^2}{μ_0}
    \int_{10r_0<|y-x_0|<1/j}
    \frac{1}{|x_0+r_0\zeta-y|^2}\dl{μ(y)}
    \lesssim_{i,δ} 1 .
  \end{align}
  On the remaining part, $|y-x_0|\ge 1/j$, and for $r_0<1/(100j)$ we have $|x_0+r_0\zeta-y|\gtrsim 1/j$.  Hence
  \begin{align}
    \frac{r_0^2}{μ_0}
    \int_{|y-x_0|\ge 1/j}
    \frac{1}{|x_0+r_0\zeta-y|^2}\dl{μ(y)}
    \lesssim_{i,δ} j^2 μ(\CC) r_0 .
  \end{align}
  Decreasing $r_0$ if necessary, the last term is bounded by $1$.  Thus $ν$ satisfies \ref{item:main-lemma-tail} with a finite constant $C_{\operatorname{tail}}$ depending only on the fixed density level.  This is the tail constant used in the choice of $η$ above.
\end{proof}

\subsection*{Outline of proof of Main Lemma}
Using the first three properties of \cref{thm:main-lemma}, we can partition the set following David, Semmes, Léger and Tolsa, and build a Lipschitz graph $γ$ related to $\supp μ$ (\cref{sec:graph-construction}).
For the scheme to be meaningful, the graph needs to cover a decent amount of $\supp μ$.
We need to prove that the exceptional sets $F_1$ (points with low density) and $F_2$ (points with big angle) have small measure.
The proof of the latter relies on the smallness of the original square function (\cref{item:usf-small} of \cref{thm:main-lemma}) and the behaviour of the usual square function on Lipschitz graphs (\cref{sec:usf-lipschitz-graphs}).
We show that the norm of the latter can be essentially bounded below by $\norm{A'}$, which itself is an upper bound for the measure of the exceptional set.
Our strategy for the norm bound is to expand the kernel using its Taylor expansion and show that the term corresponding to $m=1$ is lower bounded by $\norm{A'}$.
Subsequently, we show that the terms for $m\ge 2$ each yield bounded operators with good dependence on the Lipschitz constant.
Finally, we need to connect the operator on $γ$ to the true operator.
We approximate the true operator with respect to $μ$ by the one with respect to $\hd1\restriction_{\gamma}$ in \cref{sec:comparison} using the following steps.
\begin{enumerate}
  \item Localization: Only pieces around the origin and close to the graph are important.
  \item We can pass from $\dl y$ to $h\hd1\restriction_{γ}$, where $h$ is the density of the projection of $μ$ onto $γ$.
  \item The inner integrals with respect to $μ\restriction_F$ and $h\hd1\restriction_{γ}$ are close.
  \item The measure $μ\restriction_F$ behaves like $μ$.
  \item For the outer integral, we can pass from $s\dl s \dl x$ to $\dist(ζ,E) \dl m_2(ζ)$.
\end{enumerate}
In order to hide the error term appearing in the \cref{sec:db-to-hH}, we actually need to investigate the conical square function instead of the original (\enquote{vertical}) one.
We introduce the relevant background in \cref{sec:usfe-cones}.

The final proof of \cref{thm:main-lemma} is in \cref{sec:F2}.

\section{Construction of Lipschitz Graph}\label{sec:graph-construction}
In the construction of the Lipschitz graph, we will use
\begin{equation}
  \log ε \ll \log τ \ll \log α \ll \log δ \ll -1.
\end{equation}
After the tail constant in \cref{item:main-lemma-tail} of \cref{thm:main-lemma} is fixed, $ε$ is also chosen small enough in terms of this constant.

The constants below are chosen so that every radius used in the construction is less than the radius $100$ appearing in \cref{thm:main-lemma}.
Let $x\in F\subset 10B_0$ (where $F$ is the set from \cref{thm:main-lemma}) and $r\in(0,50)$.
We say that a ball $B=B(x,r)$ is \emph{good} and write $B\in\G$ if
\begin{enumerate}
  \item $δ_μ(B)\ge δ$, and
  \item $\angle (L_B,L_0)\le α$ for some best approximating line $L_B$.
\end{enumerate}
We say that $B$ is \emph{very good}, and write $B\in\VG$, if $B(x,s)\in\G$ for all $s\in[r,50]$.

Note that by \cref{item:weakly-flat-uniform} of \cref{thm:main-lemma}, every good ball satisfies
\begin{align}
  β(x,r) \le \frac{λ_1}{δ}.
\end{align}

For $x\in F$, we set
\begin{align}
  h(x) := \inf\{r\in(0,50] : B(x,r)\in\VG  \},
\end{align}
with the convention that the infimum is $50$ if the displayed set is empty.
Finally, we partition the set relevant to the main lemma, namely $F$, by defining
\begin{align}
  Z &:= \{ x\in F : h(x) = 0 \}, \\
  F_1 &:= \{ x\in F\setminus Z : δ_μ(B(x,h(x))) \le δ \} \text{ and} \\
  F_2 &:= F \setminus (Z\cup F_1).
\end{align}

We now introduce the functions $d$ and $D$ defined by
\begin{align}
  d(x) := \inf_{B(z,r)\in\VG} |x-z| + r, \quad x\in\CC,
\end{align}
and
\begin{align}
  D(p) := \inf_{x\in π^{-1}(p)} d(x) = \inf_{B(z,r)\in\VG} |p-π(z)| + r,\quad p\in\RR.
\end{align}
For $x\in \CC$, we further set
\begin{align}
  ℓ(x)=\frac{1}{10}D(π(x)).
\end{align}
We also define
\begin{align}
  Z_0 := \{ x\in\CC : d(x) = 0 \},
\end{align}
and since $d(x)\le h(x)$ for $x\in F$, we have $Z⊂ Z_0$.
In the following we let $I_0 := (-1,1)\subset \RR$.

These definitions allow us to build a Lipschitz graph as described in \cite[Sections 7.4-7.6]{tolsaAnalyticCapacityCauchy2014}.
\begin{proposition}\label{thm:graph-properties}
  There exists a Lipschitz function $A\colon\RR\to\RR$ satisfying $\norm{A}_{\text{Lip}}\lesssim α$ and $\supp(A)\subset 3I_0$.
  Furthermore, denoting $\grapha(x) := (x,A(x))$ and $γ:= \{\grapha(x) : x\in\RR \}$,
  the following properties hold:
  \begin{enumerate}
    \item $ |A''(x)| \lesssim  \frac{ε}{D(x)}$ for all $x\in\RR$, \label{item:second-derivative-estimate}
    \item $\dist(p,L_0) \lesssim ε$ for $p\in γ$ and \label{item:graph-close-to-L0}
    \item \label{item:x-in-F-close-to-graph}if $x\in F\cap π^{-1}(8I_0)$, then
          \begin{align}
            \abs{\grapha(π(x)) - x} \lesssim ε D(π(x)).
          \end{align}
    \item $Z_0\subset γ$.
  \end{enumerate}
\end{proposition}
\Cref{item:x-in-F-close-to-graph} is similar to \cite[Lemma 7.30]{tolsaAnalyticCapacityCauchy2014} where we additionally use that $d(x) \lesssim D(π(x))$ for $x\in F\cap π^{-1}(8I_0)$ (see \cref{thm:dx-Dpix}).

\section{Usual Square Function on Lipschitz Graphs}\label{sec:usf-lipschitz-graphs}
The goal of this section is to prove that $\norm{A'}_2$ can essentially be bounded above by the usual square function with respect to $\hd 1\restriction_{γ}$ applied to $1$.
The key estimate is \cref{thm:linear-term-A-prime} which yields the $L^2$-estimate in \cref{thm:l2-lower-bound}.
In order to carry out the comparison estimates in \cref{sec:comparison}, we need to introduce a version of the usual square function using cones which is bounded on $L^p$ for $p ∈ (1, 2)$ in \cref{sec:usfe-cones}.
Finally, we derive an $L^p$-estimate in \cref{thm:lower-bound-conical}.

\subsection{Preliminaries}\label{sec:usfe-lip-preliminaries}
We will make use of established theory for square functions and use \cite{hofmann$L^p$squareFunctionEstimates2017} as a reference.
In the following, we summarize and specialize the results therein.

Consider $θ\colon(\RR^2\setminus \RR)\times \RR \to \CC$ with $C_{\theta}>0$ which satisfies
\begin{align}
  |θ(ξ,y)|           & \le \frac{C_{\theta}}{|ξ-y|^2}, \label{eq:kernel-estimate-size}             \\
  |θ(ξ,y) - θ(ξ,y')| & \le C_{\theta} \frac{|y-y'|}{|ξ-y|^{3}}, \label{eq:kernel-estimate-hoelder}
\end{align}
for all $ξ\in\RR^2\setminus\RR$, $y,y'\in\RR$ such that $|y-y'|\le \frac{1}{2}|ξ-y|$.
Then define the integral operator $Θ$ for functions $f\in L^p(\RR)$, $p\in [1,∞]$, by
\begin{align}
  (Θf)(ξ) := \int_{\RR} θ(ξ,y) f(y) \dl y,
\end{align}
for every $ξ \in \RR^2\setminus \RR$.

We further define \emph{cones} in $\CC$ based at $z\in\RR$ to be
\begin{align}
  Γ(z) := \{ ξ\in\CC : |ξ-z| < (1+κ) |\Im ξ| \},
\end{align}
where $κ>0$ is a small constant which will not be important to us.
We therefore fix a small $κ$ from now on and will suppress dependence on $κ$.
These definitions enable us to define the conical and vertical square function corresponding to $θ$ as
\begin{align}\label{eq:general-square-function}
  (S^{∨}f)(z) := \paren[\Bigg]{ \int_{Γ(z)} |(Θf)(ξ)|^2  \dl{ξ}}^{1/2}, \quad
  (S^{|}f)(x) := \paren[\Bigg]{ \int_{\RR\setminus\{0\}} |(Θf)(x+is)|^2 |s| \dl s}^{1/2},
\end{align}
respectively.
We will discuss conical square functions in more detail in \cref{sec:usfe-cones} and first focus on how the vertical square function relates to usual square function estimates for $\hd 1\restriction_{γ}$.
We let
\begin{align}
  θ(x+is, y) := \frac{1}{\paren[\big]{x-y+i(A(x)-A(y)+s)}^2},
\end{align}
where $x,y\in\RR$ and $s\in\RR\setminus\{0\}$,
and will often expand this to
\begin{equation}
  \begin{aligned}\label{eq:kernel-expansion}
    θ(x+is,y)
    = \sum_{m=0}^{\infty} (m+1) (-i)^m \frac{(A(x)-A(y))^m}{(x-y+is)^{m+2}}
    =: \sum_{m=0}^{\infty} ω_m θ_{m}(x+is,y),
  \end{aligned}
\end{equation}
where $ω_m := (m+1) (-i)^m$.
We denote the corresponding integral operators as $Θ$ and $Θ_m$.
Let $T^|$, $T^{∨}$ and $T_m^|$, $T_m^{∨}$ denote the vertical and conical square functions corresponding to $θ$ and $θ_m$, respectively.
We omit the dependence on the Lipschitz graph, as we will only consider the graph $γ$ constructed in \cref{thm:graph-properties}.
For convenience, we also define
\begin{align}
  θ_*\colon (\CC\times\CC) \setminus \{ (ξ,ξ):ξ\in\CC\}\to \CC, \quad (x+is, ζ) \mapsto \frac{1}{\paren[\big]{x+iA(x)+is-ζ}^2},
\end{align}
for usage outside the above definitions of square-function operators.
Note that $θ_*(x+is, \grapha(y))=θ(x+is,y)$.
With our notation set up, an application of the coarea formula
\begin{align}
  \int_{\RR^2} g(ζ) |\nabla u(ζ)| \dl{ζ}
  = \int_{\RR} \int_{u^{-1}(t)} g(ζ) \dl{\hd1(ζ)} \dl t,
\end{align}
with
\begin{align}
  u(x+is) = s - A(x) \quad\text{and}\quad g(ζ) = \abs[\big]{ \int_{\gamma} \frac{f(ξ)}{(ζ-ξ)^2} \dl {\hd1(ξ)}}^2 \frac{|\dist(ζ,γ)|}{\sqrt{1+A'(\Re ζ)^2}},
\end{align}
together with the area formula,
shows that $\hd 1\restriction_{γ}$ satisfies USFE if and only if $T^|$ is bounded on $L^2(\RR)$.
Note that $\norm{1+(A')^2}_{∞}\approx 1$ and $\dist(\cdot+iA(\cdot)+is,γ)\approx |s|$.

\subsection{Vertical Square Function}
In the following, we record a few simple observations about the vertical square function.
\begin{lemma}\label{thm:zero-term-mean-zero}
  For all $x+is\in\CC\setminus\RR$, we have
  \begin{align}
    (Θ_01)(x+is) = \int_{\RR} θ_{0}(x+is,y) \dl y = 0.
  \end{align}
\end{lemma}

\begin{proposition}\label{thm:each-term-bounded}
  There exists $R>0$ such that
  \begin{equation}
    \norm{T^|_m f}_{L^2(\RR)} \le R^{m+1} L^m \norm f_{L^2(\RR)},
  \end{equation}
  for every $f\in L^2(\RR)$, where $L$ is the Lipschitz constant of $A$.
\end{proposition}
\begin{proof}[Sketch of proof.]
  The kernels $θ_m$ satisfy the kernel estimates in \cref{eq:kernel-estimate-size,eq:kernel-estimate-hoelder} with $C_{θ_m} \lesssim (m+1)L^m$ which can be shown using the mean-value theorem.
  In view of \cref{thm:zero-term-mean-zero}, the case $m=0$ follows from an application of the $T(1)$-theorem (\cite[Theorem 3.2]{hofmann$L^p$squareFunctionEstimates2017}).
  Note that with the formulation therein, the operator-norm satisfies $\norm{T^|_0}_{2\to 2}\lesssim \max\{ C_{θ_0}, \sqrt{\mathcal C_{θ_0}} \}$, where $\mathcal C_{θ_0}$ is the constant of the Carleson measure condition \cite[Eq. (3.5)]{hofmann$L^p$squareFunctionEstimates2017}, i.e.
  \begin{align}
    \sup_I \frac{1}{ℓ(I)} \int_{-ℓ(I)}^{ℓ(I)}\int_{I}  |Θ_0 1(x+is)|^2 \dl x |s|\dl s \le \mathcal C_{θ_0},
  \end{align}
  where the $\sup$ ranges over all (dyadic) intervals.

  We now induct on $m$.
  Let $m\ge 1$.
  The induction step relies on the fact that
  \begin{align}
    (m+1) Θ_m(1) = m Θ_{m-1}(A').
  \end{align}
  This is true since
  \begin{equation}
    \diff*{\frac{(A(x)-A(y))^{m}}{(x-y+is)^{m+1}}}{y} = m \frac{(A(x)-A(y))^{m-1}}{(x-y+is)^{m+1}} (-A'(y)) + (m+1)\frac{(A(x)-A(y))^{m}}{(x-y+is)^{m+2}}.
  \end{equation}
  Integrating this term equals $0$ and yields the claim.

  We now verify the Carleson measure condition for $θ_m$ and show that
  \begin{align}
    |Θ_m(1)(x+is)|^2 \dl x |s|\dl s = \paren[\Big]{\frac{m}{m+1}}^2 |Θ_{m-1}(A')(x+is)|^2 \dl x |s| \dl s
  \end{align}
  is a Carleson measure.
  Let $I\subset \RR$ be an interval with length $ℓ(I)$.
  By the induction hypothesis, we have
  \begin{align}
    \int_{-ℓ(I)}^{ℓ(I)} \int_{I} |Θ_{m-1}(\indicator{3I}A')(x+is)|^2 \dl x |s| \dl s \lesssim (RL)^{2m} ℓ(I).
  \end{align}
  For the non-local part we get
  \begin{align}
    \int_{-ℓ(I)}^{ℓ(I)} \int_I |Θ_{m-1,s}(\indicator{\RR\setminus 3I} A')(x)|^2 \dl x |s|\dl s
     & = \int_{-ℓ(I)}^{ℓ(I)} \int_I \abs[\Big]{\int_{\RR\setminus 3I} \frac{(A(x)-A(y))^{m-1}}{(x-y+is)^{m+1}} A'(y) \dl y }^2 \dl x |s|\dl s \\
     & \le \int_{-ℓ(I)}^{ℓ(I)} \int_I \paren[\Big]{\int_{\RR\setminus 3I} \frac{L^{m-1}}{|x-y|^2} \cdot L \dl y }^2 \dl x |s|\dl s            \\
     & \lesssim L^{2m} ℓ(I).
  \end{align}
  Therefore, an application of the $T(1)$-theorem implies that $\norm{T^|_m}_{2\to 2} \lesssim \max\{ C_{θ_m}, \sqrt{\mathcal C_{θ_m}} \} \lesssim (R^m+1) L^m$.
  Carefully tracking the implicit constants yields the existence of $R>0$ as claimed in the statement.
\end{proof}

Summing up \cref{thm:each-term-bounded} yields the following.
\begin{proposition}\label{thm:usfe-bdd-lipschitz-graph}
  For Lipschitz graphs with sufficiently small Lipschitz constant, the vertical square function $T^|$ is bounded on $L^2(\RR)$, i.e.
  \begin{equation}
    \paren[\Big]{ \int_{\RR} \int_{\RR} \abs[\Big]{ \int_{\RR} θ(x+is,y) f(y) \dl y}^2 \dl x |s|\dl s}^{1/2} \\
    \lesssim \norm f_2, \quad\forall f\in L^2(\RR).
  \end{equation}
\end{proposition}
\Cref{thm:usfe-bdd-lipschitz-graph} in fact also holds for Lipschitz graphs with large constant (cf. \cite[Lemma III.2.3]{davidAnalysisUniformlyRectifiable1993}).
We shall not need this result.

\subsection{Fourier Calculation for Lower Bound}
We prove that the vertical square function corresponding to $\Im (ω_1θ_1)$ captures the norm of $A'$ using an energy argument.
\begin{lemma}\label{thm:linear-term-A-prime}
  We have
  \begin{equation}
    \int_{\RR} \int_{\RR} \abs[\Big]{\Im \int_{\RR} (-2i) \frac{A(x)-A(y)}{(x-y+is)^3} \dl y }^2 \dl x |s|\dl s
    \approx \norm{A'}_2^2.
  \end{equation}
\end{lemma}
\begin{proof}
  The left-hand side can be written as
  \begin{align}
    \int_{\RR} \int_{\RR} \paren[\Big]{\int_{\RR} & \frac{\paren[\big]{A(x)-A(y)} \paren[\big]{(x-y)^3 - 3(x-y)s^2}}{|x-y+is|^6} \dl y }^2 \dl x |s|\dl s                                 \\
                                                  & = \int_{\RR} \int_{\RR} \paren[\Big]{\int_{\RR} \paren[\big]{A(x)-A(x-g)} \frac{ g^3 - 3gs^2}{|g+is|^6} \dl g}^2 \dl x |s|\dl s,      \\
                                                  & = \int_{\RR} |\hat A(ξ)|^2 \int_{\RR} \abs[\Big]{\int_{\RR} (1-e^{-igξ}) \frac{g^3 - 3gs^2}{|g+is|^6}   \dl g}^2 |s|\dl s \dl{ξ}      \\
                                                  & = \int_{\RR} |\hat A(ξ)|^2 |ξ|^2 \int_{\RR} \abs[\Big]{\int_{\RR} (1-e^{-iu}) \frac{u^3 - 3ut^2}{|u+it|^6} \dl u}^2 |t| \dl t \dl{ξ},
  \end{align}
  where we substituted $g=x-y$, used Plancherel and Fubini, and finally substituted $u=gξ$ and $t = sξ$.

  The conclusion follows once we show that the constant
  \begin{align}
    C_0:=\int_{0}^{∞} \abs[\Big]{\int_{\RR} (1-e^{-iu}) \frac{u^3 - 3ut^2}{|u+it|^6} \dl u}^2 t \dl t
  \end{align}
  is finite and positive.  Positivity follows because the multiplier is not identically zero.  For finiteness, split $C_0=I_1+I_2$, where the $t$-integration is over $(0,1)$ and $(1,∞)$, respectively.
  For $I_2$, taking absolute values yields
  \begin{align}
    I_2
    &\lesssim \int_{1}^{∞} \abs[\Big]{\int_{\RR} \frac{1}{|u+it|^3} \dl u}^2 t \dl t
    \lesssim \int_{1}^{∞} \frac{1}{t^3} \dl t < ∞.
  \end{align}
  For $I_1$, we note that
  \begin{equation}
    - \diff*{\frac{u^3}{(u^2+t^2)^2}}{u}
    = \frac{-3u^2}{(u^2+t^2)^2} + \frac{4u^4}{(u^2+t^2)^3}
    = \frac{u^4-3u^2t^2}{|u+it|^6}.
  \end{equation}
  Thus, for $u\neq0$,
  \begin{equation}
    \frac{u^3-3ut^2}{|u+it|^6}
    = -\frac{1}{u}\diff*{\paren[\Big]{\frac{u^3}{|u+it|^4}}}{u}.
  \end{equation}
  Since $\int_{\RR}\diff*{\paren[\Big]{u^3/|u+it|^4}}{u}\dl u=0$, this allows us to introduce the constant $i=\lim_{u\to0}(1-e^{-iu})/u$, i.e.
  \begin{align}
    \abs[\Big]{ \int_{\RR} (1-e^{-iu}) \frac{\paren{u^3 - 3ut^2 }}{|u+it|^6} \dl u }
     & = \abs[\Big]{ \int_{\RR} \paren[\Big]{ i - \frac{1-e^{-iu}}{u}} \diff*{\paren[\Big]{\frac{u^3}{|u+it|^4}}}{u} \dl u } \\
     & \lesssim \int_{0}^{∞} \frac{\min(2, |u|)}{u^2+t^2} \dl u
    \lesssim 1+|\log t|,
  \end{align}
  where we used the mean-value theorem for the real and imaginary part, respectively.
  We therefore get the desired bound
  \begin{align}
    I_1 \lesssim \int_{0}^{1} (1+|\log t|)^2 t \dl t < \infty.
  \end{align}

\end{proof}

\begin{lemma} \label{thm:l2-tail}
  For $m_0\ge 1$, we have
  \begin{equation}
    \int_{\RR} \int_{\RR} \abs[\Big]{ \int_{\RR} \sum_{m=m_0}^{\infty} ω_m θ_{m}(x+is,y) \dl y }^2 |s| \dl x\dl s \lesssim α^{2m_0}.
  \end{equation}
\end{lemma}

\begin{proof}
  The essence of the proof is a localization argument.
  Let us denote the vertical square-function operator corresponding to the kernel $\sum_{m=m_0}^{\infty} ω_m θ_{m}$ by $T^|_{\text{tail}}$.
  For any $f\in L^2(\RR)$, we have
  \begin{align}
    \norm{T^|_{\text{tail}}f}_2 \le \sum_{m=m_0}^{∞} |ω_m| \norm{T^|_m f}_2 \lesssim α^{m_0} \norm{f}_2,
  \end{align}
  by summing up \cref{thm:each-term-bounded}.
  We have
  \begin{equation}
    \norm{T^|_{\text{tail}} 1}_2 \le \norm{T^|_{\text{tail}} \indicator{5I_0}}_2 + \norm{T^|_{\text{tail}} \indicator{\RR\setminus 5I_0}}_2
    \lesssim α^{m_0} + \sum_{m=m_0}^{∞} |ω_m| \norm{T^|_m \indicator{\RR\setminus 5I_0}}_2.
  \end{equation}
  For $m\ge m_0$, the last term evaluates to
  \begin{align}
    \norm{T_m^| \indicator{\RR\setminus 5I_0}}_2^2
     & = \int_{\RR} \int_{\RR} \abs[\Big]{\int_{\RR\setminus 5I_0} \frac{(A(x)-A(y))^m}{(x-y+is)^{m+2}} \dl y }^2 |s| \dl x\dl s \\
     & \le \int_{\RR} \int_{3I_0} \abs[\Big]{ ε^m  \int_{\RR\setminus 5I_0} \frac{1}{|x-y+is|^{m+2}} \dl y }^2 |s| \dl x\dl s    \\
     & \lesssim \int_{\RR} \int_{3I_0} \abs[\Big]{ε^m  \frac{1}{(1+|s|)^{m+1}} }^2 |s| \dl x\dl s
    \lesssim ε^{2m}.
  \end{align}

  In total, we get
  \begin{align}
    \norm{T_{\text{tail}}^| 1}_2
    \lesssim α^{m_0} + \sum_{m=m_0}^{\infty} |ω_m| ε^m
    \lesssim α^{m_0} + ε^{m_0}
    \lesssim α^{m_0}.
  \end{align}
\end{proof}

\begin{lemma}\label{thm:l2-lower-bound}
  There exists constants $c,C>0$ and $α_0>0$ such that if $\norm{A'}_{∞}\le α_0$, then we have the lower bound
  \begin{equation}
    \norm{T^| 1}_2 \ge c\norm{A'}_2 - C \norm{A'}_{∞}^2 .
  \end{equation}
\end{lemma}

\begin{proof}
  Combining \cref{thm:zero-term-mean-zero,thm:linear-term-A-prime,thm:l2-tail}, we conclude
  \begin{align}
    \norm{T^| 1}_2
     & \ge \paren[\Big]{ \int_{\RR} \int_{\RR} \abs[\big]{ ω_1 (Θ_1 1)(x+is) }^2 |s| \dl s \dl x }^{1/2}
    - \paren[\Big]{ \int_{\RR} \int_{\RR} \abs[\big]{ \sum_{m=2}^{∞} ω_m (Θ_m 1)(x+is) }^2 |s| \dl s \dl x }^{1/2} \\
     & \ge c\norm{A'}_2 - \norm{\sum_{m=2}^{∞} ω_m T_m^| 1 }_2
    \ge c\norm{A'}_2 - C\norm{A'}_{∞}^2.
  \end{align}
\end{proof}

\subsection{Conical Square Function}\label{sec:usfe-cones}
An estimate like \cref{thm:l2-lower-bound} would be sufficient for the classical execution of this scheme.
However, in our case, comparisons using the vertical square function are too delicate, and we need to switch to conical square functions.
We will need the following two results relating conical and vertical square functions.

\begin{proposition}\label{thm:conical-vertical-sf}
  For any $u\colon\CC\setminus\RR\to [0,∞]$, we have
  \begin{align}
    \int_{\RR} \int_{Γ(z)} u(ξ) \dl{ξ}  \dl z \approx \int_{\RR}\int_{\RR} u(x+is) |s| \dl s  \dl x.
  \end{align}

  In particular, the square-function operators $S^{∨}$ and $S^|$ in \cref{eq:general-square-function} satisfy
  \begin{align}
    \norm{S^{∨}f}_2 \approx \norm{S^{|}f}_{2}, \quad\forall f\in L^2(\RR).
  \end{align}
\end{proposition}
\begin{proof}
  Using Fubini, we get
  \begin{align}
    \int_{\RR} \int_{Γ(z)} u(ξ) \dl{ξ} \dl z
    = \int_{\RR} \int_{\RR} \int_{\RR} \indicator{Γ(z)}(x+is) u(x+is)\dl z \dl s \dl x
    \approx \int_{\RR} \int_{\RR} u(x+is) |s| \dl s \dl x.
  \end{align}
  The second part follows by considering $u=|Θf|^2$, where $Θ$ is the integral operator corresponding to $S$.
\end{proof}

The following extrapolation result can be found in \cite{hofmann$L^p$squareFunctionEstimates2017}.
\begin{proposition}[{\cite[Theorem 6.18]{hofmann$L^p$squareFunctionEstimates2017}}]\label{thm:conical-sf-extrapolation}
  Let $S^{∨}$ and $S^|$ be as in \cref{eq:general-square-function}.
  If $S^{∨}$ is bounded in $L^2(\RR)$, then it is bounded in $L^p(\RR)$ for $p\in(1,∞)$.
  More specifically, $\norm{S^{∨}}_{p\to p}\lesssim \max\{ \norm{S^{∨}}_{2\to 2}, C_{\theta} \}$, where the implicit constant only depends on $p$.
\end{proposition}

\begin{proposition}\label{thm:conical-sf-lp-bdd}
  The conical square function $T^{∨}$ is bounded in $\Lp{p}[\RR]$ for $p\in (1,∞)$.
\end{proposition}
\begin{proof}
  For $p=2$, \cref{thm:conical-vertical-sf,thm:usfe-bdd-lipschitz-graph} yield
  \begin{align}
    \norm{T^{∨}f}_2 \approx \norm{T^{|}f}_2 \lesssim \norm{f}_2,
  \end{align}
  for every $f\in L^2(\RR)$,
  and thus the claim follows with \cref{thm:conical-sf-extrapolation}.
\end{proof}

With the boundedness of the conical square function in hand, we now aim for a variant of \cref{thm:l2-lower-bound} using the conical square function.
\begin{lemma}\label{thm:lower-bound-conical}
  There exists constants $c,C>0$ and $α_0>0$ such that if $\norm{A'}_{∞}\le α_0$, then we have the lower bound
  \begin{align}
    \norm{T^{∨} 1}_{3/2}^{1/2} \ge c\frac{\norm{A'}_2}{α^{1/2}} - C \frac{\norm{A'}_{∞}^2}{α^{1/2}}.
  \end{align}
\end{lemma}
\begin{proof}
  Combining \cref{thm:l2-lower-bound,thm:conical-vertical-sf}, we get
  \begin{align}
    c\norm{A'}_2 - C \norm{A'}_{∞}^2
    \le \norm {T^{|}1}_2
    \approx \norm {T^{∨}1}_2.
  \end{align}
  Now for every $p$ and $q$ such that $p<2<q$, there exists $θ\in (0,1)$ such that
  \begin{align}
    \norm {T^{∨}1}_2
    \le \norm{T^{∨} 1}_p^{θ} \norm{T^{∨} 1}_q^{1-θ}.
  \end{align}
  Finally, we claim that $\norm{T^{∨} 1}_q \lesssim α$:
  We basically follow the proof of \cref{thm:l2-tail}.
  In view of \cref{thm:zero-term-mean-zero}, let us denote the conical square-function operator corresponding to the kernel $\sum_{m=1}^{\infty} ω_m θ_{m}$ by $T^{∨}_{\text{tail}}$, so that $T^{∨} 1 = T^{∨}_{\text{tail}} 1$.
  By \cref{thm:conical-sf-extrapolation,thm:conical-vertical-sf,thm:l2-tail},
  \begin{align}
    \norm{T^{∨}_{\text{tail}}}_{q\to q}
    \lesssim \max\{ \norm{T^{∨}_{\text{tail}}}_{2\to 2}, α \}
    \approx \max\{ \norm{T^{|}_{\text{tail}}}_{2\to 2}, α \}
    \lesssim α.
  \end{align}
  This implies $\norm{T^{∨}_{\text{tail}} \indicator{5I_0}}_q \lesssim α$.

  Fix $m\ge 1$ and $z\in\RR$,
  the inner integrals are
  \begin{align}
    \paren[\big]{(T_m^{∨}\indicator{\RR\setminus 5I_0})(z)}^2
     & = \int_{Γ(z)} \abs[\Big]{ \int_{\RR\setminus 5I_0} \frac{(A(\Re ξ)-A(y))^m}{(ξ-y)^{m+2}} \dl y }^2 \dl{ξ}                            \\
     & \le \int_{Γ(z) \cap \{\Re ξ \in 3I_0 \}} \abs[\Big]{ ε^m  \int_{\RR\setminus 5I_0} \frac{1}{|\Re ξ-y+i\Im ξ|^{m+2}} \dl y }^2 \dl{ξ} \\
     & \lesssim ε^{2m} \int_{Γ(z) \cap \{\Re ξ \in 3I_0 \}}\abs[\Big]{\frac{1}{(1+|\Im ξ|)^{m+1}} }^2 \dl{ξ}                               \\
    &\lesssim ε^{2m} \frac{1}{(1+\dist(z,3I_0))^{2m} }.
  \end{align}
  This leads to
  \begin{align}
    \norm{T^{∨}_{m} \indicator{\RR\setminus 5I_0}}_q^q
    \lesssim ε^{qm} \int_{\RR} \frac{1}{(1+\dist(z,3I_0))^{qm} } \dl z
    \lesssim ε^{qm}.
  \end{align}
  Using triangle-inequality, we finish the proof of the claim, since
  \begin{align}
    \norm{T^{∨}_{\text{tail}} \indicator{\RR\setminus 5I_0}}_q
    \le \sum_{m=1}^{∞}|ω_m|\norm{T_m^{∨}\indicator{\RR\setminus 5I_0}}_q
     & \lesssim \sum_{m=1}^{∞} |ω_m| ε^m \lesssim ε \lesssim α.
  \end{align}

  With the choice of $p=\frac{3}{2}$, $q=3$ and therefore $θ=\frac{1}{2}$, we get
  \begin{align}
    \norm{T^{∨} 1}_{3/2}^{1/2} \ge c\frac{\norm{A'}_2}{α^{1/2}} - C \frac{\norm{A'}_{∞}^2}{α^{1/2}}.
  \end{align}
\end{proof}

\section{Comparison Estimates}\label{sec:comparison}
In this section we prove the comparison estimates needed to relate the estimates on the usual square function (with respect to the original measure $μ$) to the usual square function on the approximating Lipschitz graph.

\subsection{Localization}
We show that the outer integrals over the shaded areas in \cref{fig:localization} are small, and it therefore suffices to compare the different integrals while the outer integrals only cover the white area.
We start by proving that the integral over the grey area is small.
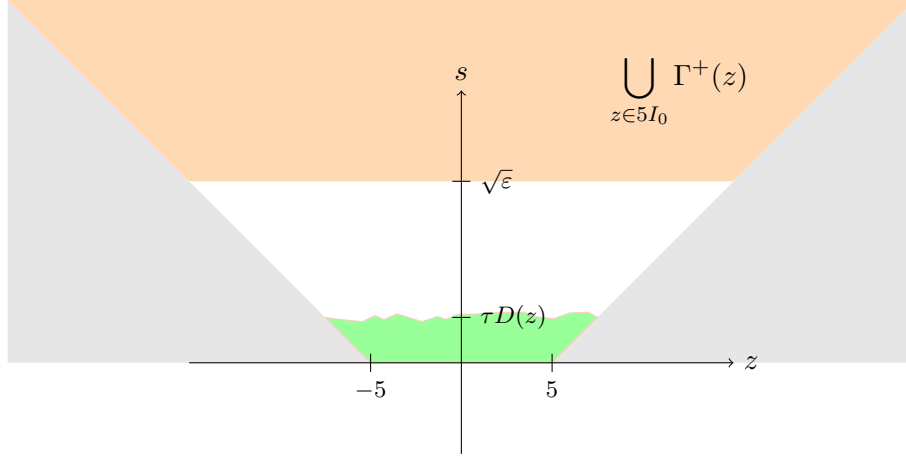
\begin{figure}
  \centering
  \begin{tikzpicture}[scale=1.2]
    \clip (-5,-1) rectangle (5,5);

    \fill[gray!20] (-5,0) -- (-1,0) -- (-5,4);

    \fill[gray!20] (5,0) -- (1,0) -- (5,4);

    \fill[orange!30] (-5,4) -- (-3,2) -- (3,2) -- (5,4);
    \node at (2.4,3) {$\displaystyle \bigcup_{z\in 5I_0} Γ^+(z)$};

    \path[fill=green!40,draw=red!20] (-1.5,0.5) -- (-1,0) -- (1,0) -- (1.5,0.5) decorate[decoration={random steps, segment length=2mm}]{--  (-1.5,0.5)};

    \draw (-1,0.1) -- (-1,-0.1) node[below] {\footnotesize $-5$};
    \draw (1,0.1) -- (1,-0.1) node[below] {\footnotesize $5$};

    \draw (-0.1,0.5) -- (0.1,0.5) node[right] {\footnotesize $τD(z)$};
    \draw (-0.1,2) -- (0.1,2) node[right] {\footnotesize $\sqrt{ε}$};

    \draw[->] (-3,0) -- (3,0) node[right] {$z$};
    \draw[->] (0,-3) -- (0,3) node[above] {$s$};
\end{tikzpicture}
  \caption{Integration regions on the upper half space.}\label{fig:localization}
\end{figure}

\begin{lemma}[Localization]\label{thm:outer-localisation-lp}
  We have
  \begin{equation}
    \norm{(T^{∨}1)\indicator{\RR\setminus 5I_0}}_p =
    \brac[\Big]{
      \int_{\RR\setminus 5I_0}
      \paren[\Big]{ \int_{Γ(z)} \abs[\big]{\int_{\RR} θ(ξ,y) \dl y}^2  \dl{ξ}}^{p/2} \dl z
    }^{1/p}
    \lesssim ε.
  \end{equation}
\end{lemma}

\begin{proof}
  We have
  \begin{align}\label{eq:loc-expansion}
    \abs[\big]{ \int_{\RR} \frac{1}{(ξ + iA(\Re ξ)-\grapha(y))^2} \dl y}
     & \le \sum_{m=0}^{∞} |ω_m| \abs[\big]{ \int_{\RR} \frac{(A(\Re ξ)-A(y))^m}{(ξ-y)^{m+2}} \dl y} \\
     & \le \sum_{m=1}^{∞} |ω_m| \int_{\RR} \frac{|A(\Re ξ)-A(y)|^m}{|ξ-y|^{m+2}} \dl y
  \end{align}

  Let $m\ge 1$.
  Case distinction yields
  \begin{align}
    \int_{\RR} \frac{|A(\Re ξ)-A(y)|^m}{|ξ-y|^{m+2}} \dl y
     & \lesssim ε^m \indicator{4I_0}(\Re ξ) \int_{\RR} \frac{1}{|ξ-y|^{m+2}} \dl y
    + ε^m \indicator{\RR\setminus 4I_0}(\Re ξ) \int_{3I_0} \frac{1}{|ξ-y|^{m+2}} \dl y
    \\
     & \lesssim ε^m \frac{ \indicator{4I_0}(\Re ξ) }{|\Im ξ|^{m+1}}
    + ε^m \frac{ \indicator{\RR\setminus 4I_0}(\Re ξ) }{|1+|\Re ξ|+i\Im ξ|^{m+2}}  \\
     & \lesssim ε^m  \frac{ \indicator{4I_0}(\Re ξ) }{|\Im ξ|^{2}}
    + ε^m \frac{ \indicator{\RR\setminus 4I_0}(\Re ξ) }{|1+|\Re ξ|+i\Im ξ|^{3}}.
  \end{align}
  Where we have used $|\Re ξ-y|\gtrsim 1+|\Re ξ|$ in the second integral, and that $|\Im ξ| \gtrsim 1$ if $|\Re ξ|<4$ and $ξ\in Γ(z)$ for $|z|\ge 5$.

  Let $z\in\RR\setminus 5I_0$.
  Then, we plug the previous estimate in the cone integral and get
  \begin{align}
    ε^{2m} \int_{Γ(z)} & \abs[\big]{
    \frac{ \indicator{4I_0}(\Re ξ) }{|\Im ξ|^{2}}
    + \frac{ \indicator{\RR\setminus 4I_0}(\Re ξ) }{|1+|\Re ξ|+i\Im ξ|^{3}}
    }^2  \dl{ξ}                                                                                   \\
                       & \lesssim ε^{2m} \int_{4I_0} \int_{|x-z|}^{∞} \frac{1}{|s|^4} \dl s \dl x
    + ε^{2m} \int_{\RR\setminus 4I_0} \int_{|x-z|}^{∞}  \frac{1}{|1+|x|+is|^6} \dl s \dl x        \\
                       & \lesssim ε^{2m} \int_{4I_0} \frac{1}{|x-z|^3} \dl x
    + ε^{2m} \int_{Γ(z)} \frac{1}{|1+ξ|^6} \dl{ξ}                                                 \\
                       & \lesssim ε^{2m} \frac{1}{|4-|z||^2}
    + ε^{2m} \frac{1}{|1+z|^4},
  \end{align}
  where we used
  \begin{align}
    |1+|x|+is| \approx 1+|x| + |s| \ge |1+x| + |s| \approx |(1+i0) + x+is|,
  \end{align}
  and $|x-z|\ge |4-|z||$.

  Putting these estimates together, and noting \cref{thm:zero-term-mean-zero}, we get
  \begin{align}
    \norm{(T^{∨}1)\indicator{\RR\setminus 5I_0}}_p
    \lesssim \sum_{m=1}^{∞} |ω_m| ε^m \brac[\Big]{ \int_{\RR\setminus 5I_0} \paren[\big]{ \frac{1}{|4-|z||^2} + \frac{1}{|1+z|^4} }^{p/2} \dl z}^{1/p}
    \lesssim ε
  \end{align}
\end{proof}

Our comparison estimates will work well over the \emph{truncated cones}
\begin{align}
  \tilde{\Gamma}(z) := Γ(z)\cap \paren[\big]{\RR\times (-\sqrt{ε}, -τD(z))\cup (τD(z),\sqrt{ε})}.
\end{align}
We will show that the integrals over the remaining parts of the cones are small and denote
\begin{align}
  \tilde{\Gamma}(z)^- := Γ(z)\cap (\RR\times (-τD(z),τD(z))) \quad\text{and}\quad \tilde{\Gamma}(z)^+ := Γ(z)\cap \paren[\big]{\RR\times (-∞,-\sqrt{ε})\cup (\sqrt{ε},∞)}
\end{align}

Next, we prove that the integral over the green area in \cref{fig:localization} is small.

\begin{lemma}[Localization]\label{thm:upper-cone-localisation-lp}
  We have
  \begin{equation}\label{eq:localization-plus-bound}
    \brac[\Big]{
      \int_{5I_0}
      \paren[\Big]{ \int_{\tilde{\Gamma}(z)^+} \abs[\big]{\int_{\RR} θ(ξ,y) \dl y}^2  \dl{ξ}}^{p/2} \dl z
    }^{1/p}
    \lesssim \sqrt ε.
  \end{equation}
\end{lemma}
\begin{proof}
  In view of \cref{eq:kernel-expansion,thm:zero-term-mean-zero}, fix $m\ge 1$.
  The naive bound works
  \begin{align}
    \int_{5I_0}
    \paren[\Big]{ \int_{\tilde{\Gamma}(z)^+} \abs[\big]{ & \int_{\RR} \frac{|A(\Re ξ)-A(y)|^m}{|ξ-y|^{m+2}} \dl y}^2  \dl{ξ}}^{p/2} \dl z \\
                                                         & \lesssim ε^{mp} \int_{5I_0}
    \paren[\Big]{ \int_{\tilde{\Gamma}(z)^+} \abs[\big]{\int_{\RR} \frac{1}{|ξ-y|^{m+2}} \dl y}^2  \dl{ξ}}^{p/2} \dl z                    \\
                                                         & \lesssim ε^{mp} \int_{5I_0}
    \paren[\Big]{ \int_{\tilde{\Gamma}(z)^+} \frac{1}{|\Im ξ|^{2m+2}}  \dl{ξ}}^{p/2} \dl z                                                \\
                                                         & \lesssim ε^{mp} \int_{5I_0}
    \paren[\Big]{ \int_{\RR} \int_{\max\{|x-z|,\sqrt{ε}\}}^{∞} \frac{1}{s^{2m+2}}  \dl s \dl x}^{p/2} \dl z                               \\
                                                         & \lesssim ε^{mp} \int_{5I_0} \paren[\Big]{ \frac{1}{ε^m}}^{p/2} \dl z
    \lesssim ε^{mp/2}.
  \end{align}

  In total, we conclude that the left-hand side of \cref{eq:localization-plus-bound} is bounded by
  \begin{align}
    \sum_{m=1}^{∞} |ω_m| ε^{m/2} \lesssim \sqrt ε.
  \end{align}
\end{proof}

Finally, we prove that the integral over the orange area in \cref{fig:localization} is small.
\begin{lemma}[Localization]\label{thm:lower-cone-localisation-lp}
  We have
  \begin{equation}
    \brac[\Big]{
      \int_{5I_0}
      \paren[\Big]{ \int_{\tilde{\Gamma}(z)^-} \abs[\big]{\int_{\RR} θ(ξ,y) \dl y}^2  \dl{ξ}}^{p/2} \dl z
    }^{1/p}
    \lesssim \sqrt τα.
  \end{equation}
\end{lemma}
\begin{proof}
  In view of \cref{eq:kernel-expansion,thm:zero-term-mean-zero}, fix $m\ge 1$.
  We rewrite
  \begin{align}\label{eq:localization-gamma-minus-proof}
    (Θ_m 1)(ξ)
    = \int_{\RR} \frac{(\Re ξ-y)^m}{(ξ-y)^{m+2}} \brac[\Big]{ \paren[\Big]{\frac{A(\Re ξ)-A(y)}{\Re ξ-y}}^m - A'(\Re ξ)^m } \dl y,
  \end{align}
  where we introduced the constant term $A'(\Re ξ)^m$ by using Cauchy's integral formula which shows that
  \begin{equation}
    \int_{\RR} \frac{(\Re ξ-y)^m}{(ξ-y)^{m+2}} \dl y = 0 \quad\forall m\in\NN, \, ξ\in\CC\setminus\RR.
  \end{equation}
  Let $ρ>0$.
  We split the domain of integration in \cref{eq:localization-gamma-minus-proof} and first estimate the contribution from $|\Re ξ-y|>ρ$:
  \begin{align}
     & \int_{|\Re ξ - y|>ρ}
    \frac{|\Re ξ-y|^m}{|ξ-y|^{m+2}}
    \abs[\Big]{ \paren[\Big]{\frac{A(\Re ξ)-A(y)}{\Re ξ-y}}^m - A'(\Re ξ)^m } \dl y \\
     & \qquad\lesssim \int_{|\Re ξ - y|>ρ} \frac{1}{|ξ-y|^2} 2α^m \dl y
    \lesssim \frac{α^m}{ρ}.
  \end{align}
  Plugging this into the outer integrals with $ρ=\sqrt τ D(z)$ yields
  \begin{align}
    \int_{5I_0} \paren[\Big]{ \int_{\tilde{\Gamma}(z)^-} \frac{α^{2m}}{ρ^2} \dl{ξ} }^{p/2} \dl z
    \approx \int_{5I_0} \paren[\Big]{ (τD(z))^2 \frac{α^{2m}}{τD(z)^2} }^{p/2} \dl z
    \approx (\sqrt τα^m)^p,
  \end{align}
  where we used that $|\tilde{\Gamma}(z)^-| \approx (τD(z))^2$.

  For the other part, let $z\in 5I_0$, $ξ\in \tilde{\Gamma}(z)^-$ and $|\Re ξ - y|< ρ$.
  We use \cref{item:second-derivative-estimate} of \cref{thm:graph-properties} and calculate
  \begin{equation}
    \abs[\Big]{ \frac{A(\Re ξ)-A(y)}{\Re ξ-y} - A'(\Re ξ) } \lesssim |A''(ζ)| |\Re ξ - y| \lesssim \frac{ε}{D(z)} |\Re ξ -y|
  \end{equation}
  where $ζ\in [\Re ξ,y]$, and we used that $D(ζ)\approx D(z)$.
  To see this, notice that $y\in B(\Re ξ, \sqrt τD(z))$ and $\Re ξ\in B(z,τD(z))$ %
  so that $ζ\in B(z,2\sqrt τD(z))$.
  Now since $D$ is $1$-Lipschitz, we deduce:
  \begin{align}
    D(z) \le |D(z)-D(ζ)| + D(ζ) \le 2 \sqrt τ D(z) + D(ζ).
  \end{align}
  We also have $|a^m-b^m| \le m|a-b|\max\{|a|,|b|\}^{m-1}$ and therefore
  \begin{equation}
    \abs[\Bigg]{ \paren[\Big]{\frac{A(\Re ξ)-A(y)}{\Re ξ-y}}^m - A'(\Re ξ)^m }
    \lesssim \frac{εm}{D(z)} α^{m-1} |\Re ξ- y|
    \lesssim \frac{mα^m}{D(z)} |\Re ξ- y|.
  \end{equation}
  Therefore, the other part of \cref{eq:localization-gamma-minus-proof} can be estimated as
  \begin{align}
     & \int_{|\Re ξ -y|<ρ}
    \frac{|\Re ξ-y|^m}{|ξ-y|^{m+2}}
    \abs[\Big]{ \paren[\Big]{\frac{A(\Re ξ)-A(y)}{\Re ξ-y}}^m - A'(\Re ξ)^m } \dl y \\
     & \lesssim \frac{mα^m}{D(z)}  \int_0^{\sqrt τ D(z)} \frac{1}{r+|\Im ξ|} \dl r  \\
     & \lesssim \frac{mα^m}{D(z)} \log\paren[\big]{1 + \frac{D(z)}{|\Im ξ|}}        \\
     & \lesssim \frac{mα^m}{D(z)} \sqrt{1 + \frac{D(z)}{|\Im ξ|}}.
  \end{align}
  Plugging this into the outer integrals yields

  \begin{align}
    (mα^m)^p \int_{5I_0} \paren[\Big]{ \int_{\tilde{\Gamma}(z)^-} \frac{1}{D(z)^2} \paren[\Big]{1+\frac{D(z)}{|\Im ξ|}} \dl{ξ} }^{p/2} \dl z
     & \lesssim (m\sqrt τα^m)^p .
  \end{align}

  Putting the pieces together, we get
  \begin{align}
    \brac[\Big]{
      \int_{5I_0}
      \paren[\Big]{ \int_{\tilde{\Gamma}(z)^-} \abs[\big]{\int_{\RR} θ(ξ,y) \dl y}^2  \dl{ξ}}^{p/2} \dl z
    }^{1/p}
    \lesssim \sqrt τ \sum_{m=1}^{∞} |ω_m|(α^m + mα^m) \lesssim \sqrt τ α.
  \end{align}
\end{proof}

\subsection{Comparing \texorpdfstring{$\dl y$ to $h\hd1\restriction_{\gamma}$}{db to hH|γ}}\label{sec:db-to-hH}

Let $η\colon\RR\to\RR$ be a smooth, radial and non-increasing function such that $\indicator{[-1/4,1/4]} \le η \le \indicator{[-1,1]}$ and $\norm{η}_1=1$.
For $ρ>0$, we denote
\begin{align}
  η_{\rho} = \frac{1}{ρ} η\paren[\Big]{\frac{\cdot}{ρ}}.
\end{align}
We define $σ:= π_{\#}(μ\restriction_F)$ to be the pushforward of $μ\restriction_F$ under the orthogonal projection onto the real axis and let $g\colon\RR\to\RR$ be defined by
\begin{align}
  g(x) = \paren[\big]{ η_{\sqrt ε D(x)} \ast σ }(x).
\end{align}

\begin{lemma}[{\cite[Lemma 10.5]{tolsaPrincipalValuesRiesz2008}}]\label{thm:sigma-g-properties}
  We have
  \begin{align}
    σ(B(x,r)) \lesssim r
  \end{align}
  for $x\in\RR$ and $r\ge \sqrt{ε}D(x)$.
\end{lemma}

\begin{lemma}\label{thm:g-small}
  If the auxiliary parameters have been chosen small enough with respect to $α$, then we have
  \begin{align}
    \norm{\indicator{8I_0} (g-1)}_p \lesssim α^{2/p}.
  \end{align}
\end{lemma}
\begin{proof}
  In \cite[Lemma 10.5]{tolsaPrincipalValuesRiesz2008}, Tolsa shows that
  \begin{align}
    0 \le g \le 1+Cα^2 \quad\text{and}\quad \norm{\indicator{8I_0} (g-1)}_1 \lesssim α^2.
  \end{align}
  From this it follows that
  \begin{align}
    \int_{8I_0} |1-g(t)|^p \dl t \le (1+\norm g_{∞})^{p-1}  \int_{8I_0} |1-g(t)| \dl t \lesssim α^2.
  \end{align}
\end{proof}

Furthermore, we define $h\colon γ\to \RR$,
\begin{align}
  h(x) = \frac{g(π(x))}{J\grapha(π(x))},
\end{align}
where $J\grapha(t)=\sqrt{1+|A'(t)|^2}$ denotes the Jacobian of the parametrization $t\mapsto\grapha(t)$.

We only compare when the domains of the inner integrals are $7I_0$ and $\pi^{-1}(7I_0)$, respectively.
We make up for this in \cref{thm:dt-Integral-outsideRegion-small,thm:dmu-innerIntegral-outsideRegion-small} by showing that the other parts are small anyway.

\begin{lemma}\label{thm:dy-hdH-lp}
  For the localized integrals, we have
  \begin{align}
    \int_{5I_0} \paren[\Big]{ \int_{\tilde{\Gamma}(a)}
      \abs[\Big]{ \int_{7I_0} θ(ξ,y) \dl y - \int_{\pi^{-1}(7I_0)} θ_*(ξ,z) h(z) \dl{\hd1\restriction_{γ}(z)} }^2
      \dl{ξ} } ^{p/2} \dl a
    \lesssim α^2.
  \end{align}
\end{lemma}

\begin{proof}
  The inner integral can be identified as $Θ((1-g)\indicator{7I_0})$.
  The claim now follows from \cref{thm:conical-sf-lp-bdd,thm:g-small}.
\end{proof}

\subsection{Comparing \texorpdfstring{$h\hd1\restriction_{\gamma}$ to $\dl{μ\restriction_F}$}{hH|γ to μ|F}}
In the following, we will often reduce the problem to the vertical $L^2$ square function.
\begin{lemma}\label{thm:conical-to-vertical}
  For $f\colon\CC\to [0,∞)$, we have
  \begin{align}
    \int_{5I_0} \paren[\Big]{ \int_{\tilde{\Gamma}(z)} f(ξ)  \dl{ξ} }^{p/2} \dl z
    \lesssim \paren[\Big]{ \int_{6I_0} \int_{\{|s|\in (\frac{τ}{2}D(x),\sqrt ε)\}} f(x+is)  |s| \dl s \dl x }^{p/2}.
  \end{align}
\end{lemma}
\begin{proof}
  We use Hölder to move the $p/2$ outside,
  \begin{align}
    \int_{5I_0} \paren[\Big]{ \int_{\tilde{\Gamma}(z)} f(ξ)  \dl{ξ} }^{p/2} \dl z
    \lesssim \paren[\Big]{ \int_{5I_0} \int_{\RR^2} f(x+is) \indicator{\tilde{\Gamma}(z)}(x+is)  \dl(x,s) \dl z }^{p/2}
  \end{align}
  Now to use Fubini, we need to analyze what happens if $x+is\in\tilde{\Gamma}(z)$.
  In this case, $|s|<\sqrt ε$ and $x\in 6I_0$.
  For a lower bound on $|s|$, we note $|x+is-z|<(1+κ)|s|$ and $|s|>τD(z)$.
  This implies
  \begin{align*}
    |s| > τD(z) \ge τ(D(x) - |x-z|) \ge τ (D(x) - (1+κ)|s|)
  \end{align*}
  and thus
  \begin{align*}
    |s| > \frac{1}{1+(1+κ)τ} τ D(x) \ge \frac{1}{2}τD(x).
  \end{align*}
  This implies that
  \begin{align}
    \int_{5I_0} \indicator{\{x+is\in\tilde{\Gamma}(z)\}} \dl z
    \le \indicator{\{|s|\in (τD(x)/2,\sqrt ε)\}} \indicator{6I_0}(x) \int_{5I_0} \indicator{\{x+is\in Γ(z)\}} \dl z
    \lesssim \indicator{\{|s|\in (τD(x)/2,\sqrt ε)\}} \indicator{6I_0}(x) |s|
  \end{align}
  and concludes the proof.
\end{proof}

\begin{lemma}\label{thm:x-y-A-s-comparability}
  For any $x,y,s\in \RR$, we have
  \begin{equation}
    |x-y+i(A(x) - A(y) + s)| \approx \sqrt{|x-y|^2+s^2}
  \end{equation}
\end{lemma}
\begin{proof}
  Let $L$ denote the Lipschitz constant of $A$.
  Switching between the $1$- and $2$-norm on $\RR^2$ together with
  \begin{align}
    (1-L) |x-y| + |s|
    \le |x-y| + |s| - |A(x)-A(y)|
     & \le |x-y| + |A(x) - A(y) + s| \\
     & \le (1+L) |x-y| + |s|
  \end{align}
  yields the claim.
\end{proof}

For our next comparison estimate, we import the following lemma from \cite{jayeHuovinenTransformRectifiability2022}.
\begin{lemma}[{\cite[Lemma 11.5]{jayeHuovinenTransformRectifiability2022}}]\label{thm:muF-hdH-f-estimate}
  There exists a constant $C>0$ such that for any (measurable) function $f\colon\CC\to\CC$:
  \begin{align}
    \abs[\Big]{ \int_{\pi^{-1}(7I_0)} f \, \dl(\mu\restriction_F - h\hd1\restriction_{\gamma}) }
    \lesssim & \int_{7I_0} \brac{\osc_{B(\grapha(t),C\sqrt ε ℓ(t))} f} \dl{σ(t)} \\
             & + \int_{(7+C\sqrt ε)I_0 \setminus 7I_0} |f\circ \grapha| \dl{m_1}
    + \abs[\Big]{ \int_{8I_0} (f\circ \grapha) b \dl{m_1}},
  \end{align}
  where $b\colon\RR\to\RR$ satisfies $\supp b\subset 8I_0$ and $\norm b_{\infty}\lesssim \sqrt ε$.
  Here, $\osc_{B} f = \sup_{ξ,ζ\in B} |f(ξ) - f(ζ)|$.
\end{lemma}
The setup is slightly different in \cite{jayeHuovinenTransformRectifiability2022} (they use modified $α$-numbers), but the proof of the lemma above only uses \cref{item:x-in-F-close-to-graph} of \cref{thm:graph-properties}, \cref{thm:sigma-g-properties}, \cref{thm:g-small} and that $D$ is $1$-Lipschitz.

\begin{lemma}\label{thm:muF-hdH-estimate}
  For the localized integrals, we have
  \begin{equation}
    \int_{5I_0} \paren[\Big]{ \int_{\tilde{\Gamma}(a)} \abs[\Big]{ \int_{\pi^{-1}(7I_0)} θ_*(ξ,z) \dl (μ\restriction_F(z) - h(z)\hd1\restriction_{\gamma}(z))  }^2  \dl{ξ}}^{p/2}  \dl a  \lesssim ε^{p/4}.
  \end{equation}
\end{lemma}
\begin{proof}
  We use \cref{thm:conical-to-vertical} and need to bound
  \begin{align}
    \int_{6I_0} \int_{\{|s|\in (\frac{τ}{2}D(x),\sqrt ε)\}} \abs[\Big]{ \int_{\pi^{-1}(7I_0)} θ_*(x+is,z) \dl (μ\restriction_F(z) - h(z)\hd1\restriction_{\gamma}(z))  }^2  |s| \dl s \dl x.
  \end{align}

  We apply \cref{thm:muF-hdH-f-estimate} to the function $θ_*(x+is,\cdot)$, and treat the three terms separately.

  \proofstep{First term.}
  Let $x\in 6I_0$ and $|s|\ge \frac{τ}{2}D(x)$; equivalently, $|s|\ge 5τℓ(x)$.
  We claim that
  \begin{align}
    \osc_{B(\grapha(ξ),\sqrt{ε} ℓ(ξ))} θ_*(x+is,\cdot)
     & \lesssim \sqrt{ε} \frac{ℓ(ξ)}{(|x-ξ| + |s| + τℓ(ξ))^3}.
  \end{align}
  \begin{proof}
    For any $u\in B(\grapha(ξ),\sqrt{ε} ℓ(ξ))$, we let $u'=u-\grapha(ξ)$, and have
    \begin{align}
      |x-ξ| + |s| + τℓ(ξ) \lesssim |\grapha(x)-u+is| \lesssim |x-ξ| + |s| + \sqrt ε ℓ(ξ).
    \end{align}
    For the lower bound, we calculate
    \begin{align}
      |\grapha(x)-u+is|
       & \ge |x-ξ + is| - |A(x) - A(ξ)| - |u'|                                    \\
       & \gtrsim |x-ξ| + |s| - α|x-ξ| - \sqrt{ε} ℓ(ξ)                             \\
       & \gtrsim |x-ξ| + \frac{|s|}{2} + \frac{τ}{4}(ℓ(ξ) - |x-ξ|) - \sqrt ε ℓ(ξ) \\
       & \gtrsim |x-ξ| + |s| + τℓ(ξ).
    \end{align}
    For the upper bound, we have
    \begin{align}
      |\grapha(x) - u + is|
      \le |x-ξ| + |A(x)-A(ξ)| + \sqrt{ε}ℓ(ξ) + |s|
      \lesssim |x-ξ| + \sqrt{ε}ℓ(ξ) + |s|.
    \end{align}

    To calculate the oscillation, let $y,z\in B(\grapha(ξ),\sqrt{ε} ℓ(ξ))$ and write $y' = y - \grapha(ξ)$, $z'=z-\grapha(ξ)$.
    We have
    \begin{align}
      |θ_*( & x+is,y) - θ_*(x+is,z)|
      = \abs[\Big]{ \frac{1}{(\grapha(x)-y+is)^2} - \frac{1}{(\grapha(x)-z+is)^2 }}                                                                                                                   \\
            & \le |\grapha(x) - \grapha(ξ) + is|^2 \abs[\Big]{ \frac{1}{|\grapha(x) - y+ is|^4} - \frac{1}{|\grapha(x) - z+ is|^4} }                                                                  \\
            & \quad\quad + 2 \sup_{u'\in B(0,\sqrt{ε} ℓ(ξ))} \frac{|2u'(\grapha(x) - \grapha(ξ) + is)|}{|\grapha(x) - \grapha(ξ) - u' + is|^4} + \frac{|u'|^2}{|\grapha(x) - \grapha(ξ) - u' + is|^4} \\
            & \lesssim \frac{|y-z|}{(|x-ξ| + |s| + τℓ(ξ))^3} + \frac{\sqrt ε ℓ(ξ)}{(|x-ξ| + |s| + τℓ(ξ))^3} + \frac{ε ℓ(ξ)^2}{(|x-ξ| + |s| + τℓ(ξ))^4}                                                \\
            & \lesssim \frac{\sqrt ε ℓ(ξ)}{(|x-ξ| + |s| + τℓ(ξ))^3}.
    \end{align}
  \end{proof}
  To get the desired bound for the first term, it suffices to see
  \begin{equation}
    \int_{6I_0} \int_{0}^{\sqrt ε} \abs[\Big]{ \int_{7I_0} \frac{ε^{1/4}ℓ(y)}{(τℓ(y)+s+|x-y|)^3} \dl{σ(y)} }^2 s \dl s  \dl x \lesssim 1.
  \end{equation}
  Using duality, fix $f\in\Lp2[(0,\sqrt ε)\times 7I_0;s\dl s \dl x]$.
  We need to see
  \begin{equation}\label{eq:duality-inequality}
    \int_{6I_0} \int_{0}^{\sqrt ε} \int_{7I_0} \frac{ε^{1/4}ℓ(y)}{(τℓ(y)+s+|x-y|)^3} f(s,x) \dl{σ(y)} \, s \dl s  \dl x \lesssim \norm f_2.
  \end{equation}
  We investigate the integral in $x$:
  \begin{align}
     & \int_{6I_0} \frac{ℓ(y)}{(τℓ(y)+s+|x-y|)^3} f(s,x) \dl x                                         \\
     & \lesssim \int_{0}^{∞} \frac{ℓ(y)}{(τℓ(y)+s+r)^4} \int_{|x-y|<r} |f(s,x)| \dl x \dl r            \\
     & \lesssim \int_{0}^{ℓ(y)} \frac{ℓ(y)^2}{(τℓ(y)+s+r)^4} \fint_{|x-y|<ℓ(y)} |f(s,x)| \dl x \dl r +
    \int_{ℓ(y)}^{∞} \frac{ℓ(y)r}{(τℓ(y)+s+r)^4} \fint_{|x-y|<r} |f(s,x)| \dl x \dl r                   \\
     & \lesssim \int_{0}^{∞} \frac{ℓ(y)}{(τℓ(y)+s+r)^3} (M_x f)(s,y) \dl r
    \lesssim \frac{ℓ(y)}{(τℓ(y)+s)^2} (M_x f)(s,y).
  \end{align}
  Here, $M_x$ is the maximal function with respect to the second variable and radii bounded below.
  Concretely, $(M_x f) (s,y) := (M f(s,\cdot))(y)$, where
  \begin{equation}
    M\colon\Lp2[6I_0;\dl x]\to \Lp2[7I_0;\dl{σ}],\quad
    (Mf)(y) = \sup_{r>ℓ(y)} \frac{1}{m(B(y,r))} \int_{B(y,r)\cap6I_0} |f(x)| \dl x
  \end{equation}
  is bounded by interpolation:
  The crucial \enquote{doubling}-property is $σ(3B_j)\lesssim m(B_j)$ for balls (intervals) sufficiently large (cf. \cref{thm:sigma-g-properties}).
  By homogeneity, it suffices to show $σ(E)\lesssim \norm{f}_{\Lp1[6I_0;\dl x]}$ for $E:=\{|Mf| > 1\}$ and $f\in \Lp1[6I_0;\dl x]$.
  As in the Euclidean case (cf. Vitali covering), choose disjoint intervals $B_j=B(y_j,r_j)$ with $r_j>ℓ(y_j)$ such that $\int_{B_j} |f(x)| \dl x \ge m_1(B_j)$ and $E\subset \cup_j 3B_j$.
  This gives the desired bound
  \begin{equation}
    σ(E) \le \sum_j σ(3B_j) \lesssim \sum_j r_j \lesssim \sum_j m_1(B_j) \le \int_{\cup_j B_j} |f(x)| \dl x.
  \end{equation}

  Coming back to \cref{eq:duality-inequality}, we see that its left-hand side is bounded by
  \begin{align}
    \int_{0}^{\sqrt ε} \int_{7I_0} & \frac{ε^{1/4} ℓ(y)}{(τℓ(y)+s)^2} (M_x f)(s,y) \dl{σ(y)} \, s \dl s                                                                                                                                                  \\
                                   & \le \paren[\Big]{\int_{0}^{\sqrt ε} \int_{7I_0} (M_x f)(s,y)^2 \dl{σ(y)} \, s \dl s}^{1/2} \paren[\Big]{\int_{0}^{\sqrt ε} \int_{7I_0} \paren[\Big]{\frac{ε^{1/4}ℓ(y)}{(τℓ(y)+s)^2}}^2 \dl{σ(y)} \, s \dl s}^{1/2}.
  \end{align}
  \Cref{eq:duality-inequality} now follows from the boundedness of $M$ and
  \begin{align}
    \int_{0}^{\sqrt ε} \int_{7I_0} \paren[\Big]{\frac{ε^{1/4}ℓ(y)}{(τℓ(y)+s)^2}}^2 \dl{σ(y)} \, s \dl s
     & \le \int_{7I_0} \int_{0}^{\sqrt ε}  \frac{\sqrt εℓ(y)^2}{(τℓ(y)+s)^3} \dl s \dl{σ(y)} \\
     & \lesssim \int_{7I_0} \frac{\sqrt{ε}}{τ^2} \dl{σ(y)}
    \lesssim σ(7I_0)
    \lesssim 1.
  \end{align}

  \proofstep{Second term.}
  The second term can be estimated as
  \begin{align}
    \int_{6I_0} \int_{\{|s|\in (\frac{τ}{2}D(x),\sqrt ε)\}} & \abs[\Big]{ \int_{(7+c\sqrt ε)I_0\setminus 7I_0} \frac{1}{|\grapha(x)-\grapha(y)+is|^2} \dl y }^2 |s| \dl s \dl x
    \le 2 \int_{6I_0} \int_{\frac{τ}{2}D(x)}^{\sqrt ε} s \dl s \dl x
    \le 2ε,
  \end{align}
  since the separation of $x$ from $y$ yields $|\grapha(x)-\grapha(y)+is| \ge 1$

  \proofstep{Third term.}
  The third term can be bounded by
  \begin{align}
    \int_{6I_0} \int_{\{|s|\in (\frac{τ}{2}D(x),\sqrt ε)\}} \abs[\Big]{ \int_{8I_0} \frac{b(y)}{(\grapha(x)-\grapha(y)+is)^2} \dl y  }^2  |s| \dl s \dl x \lesssim \norm{b}_2^2 \lesssim ε
  \end{align}
  by \cref{thm:usfe-bdd-lipschitz-graph} and the properties of $b$ in \cref{thm:muF-hdH-f-estimate}.
\end{proof}

\begin{lemma}\label{thm:dt-Integral-outsideRegion-small}
  We have
  \begin{align}
    \int_{5I_0} \paren[\Big]{ \int_{\tilde{\Gamma}(z)} \abs[\big]{ \int_{\RR\setminus 7I_0} θ(ξ,y) \dl y}^2 \dl{ξ} } ^{p/2} \dl z
    \lesssim ε^{p/2}.
  \end{align}
\end{lemma}
\begin{proof}
  We use \cref{thm:conical-to-vertical} and are led to bound
  \begin{align}
    \int_{6I_0} \int_{\{|s|\in (\frac{τ}{2}D(x),\sqrt ε)\}} \abs[\Big]{ \int_{\RR\setminus 7I_0} θ(x+is,y) \dl y }^2  |s| \dl s \dl x . \label{eq:inner-truncated-integral-db}
  \end{align}
  Using the integration bounds, \cref{thm:x-y-A-s-comparability} implies
  \begin{align}
    |Θ \indicator{\RR\setminus 7I_0}|
    \lesssim \int_{\RR\setminus 7I_0} \frac{1}{(|x-y|+|s|)^2} \dl y
    \lesssim 1.
  \end{align}
  Plugging this into \cref{eq:inner-truncated-integral-db} yields the claim.
\end{proof}

\begin{lemma}\label{thm:dmu-innerIntegral-outsideRegion-small}
  We have
  \begin{equation}
    \int_{5I_0} \paren[\Big]{ \int_{\tilde{\Gamma}(a)} \abs[\Big]{ \int_{\CC\setminus \pi^{-1}(7I_0)} θ_*(ξ,z) \dl{μ\restriction_F(z)}  }^2  \dl{ξ}}^{p/2}  \dl a  \lesssim ε^{p/2}.
  \end{equation}
\end{lemma}
\begin{proof}
  We use \cref{thm:conical-to-vertical} and are led to estimate
  \begin{equation}\label{eq:inner-truncated-integral-dmu}
    \int_{6I_0} \int_{\{|s|\in (\frac{τ}{2}D(x),\sqrt ε)\}} \abs[\Big]{ \int_{\CC\setminus \pi^{-1}(7I_0)} θ_*(x+is,z) \dl{μ\restriction_F(z)}}^2 |s| \dl s \dl x.
  \end{equation}
  From the separation of $\Re z$ from $x$, we get $|\grapha(x)+is-z| \ge 1$, and thus
  \begin{align}
    \abs[\Big]{ \int_{\CC\setminus \pi^{-1}(7I_0)} θ_*(x+is,z) \dl{μ\restriction_F(z)} }
    \lesssim \sum_{k=0}^{\infty} \int_{|z-\grapha(x)-is|\approx 2^k} \frac{1}{|\grapha(x)+is-z|^2}\dl{μ\restriction_F(z)}
    \lesssim 1.
  \end{align}
  Plugging this into \cref{eq:inner-truncated-integral-dmu} yields the claim.
\end{proof}

\subsection{Comparing \texorpdfstring{$\dl{μ\restriction_F}$ to $μ$}{μ|F to μ}}
We continue to compare the inner integrals, this time trying to reach the original measure $μ$.
In order to manage the mixture with the outer integrals with respect to $s\dl s\dl x$, we need the following lemma relating $s$ and the distance to the graph.
We outsource the proof of it to \cref{sec:appendix} since it requires technical construction details which we omitted in \cref{sec:graph-construction}.
\begin{lemma}\label{thm:compare-t-dist}
  For $x\in 6I_0$ and $|s|\in(\frac{τ}{2}D(x),\sqrt{ε})$, we have $|s|\lesssim \dist(\grapha(x)+is,E)$.
\end{lemma}
Note that the hypothesis in \cref{thm:compare-t-dist} covers the slightly enlarged interval used in \cref{thm:conical-to-vertical}.

\begin{lemma}\label{thm:muF-mu-small}
  For the localized integrals, we have
  \begin{equation}
    \int_{5I_0} \paren[\Big]{ \int_{\tilde{\Gamma}(a)} \abs[\Big]{ \int_{\CC} θ_*(ξ,z) \dl (μ\restriction_F(z) - μ(z))  }^2  \dl{ξ}}^{p/2}  \dl a  \lesssim ε^{p/2}.
  \end{equation}
\end{lemma}
\begin{proof}
  We use \cref{thm:conical-to-vertical} and are led to bound
  \begin{align}\label{eq:muF-mu-small-intermediate}
    \int_{6I_0} \int_{\{|s|\in (\frac{τ}{2}D(x),\sqrt ε)\}} \abs[\Big]{ \int_{\CC} \frac{\indicator{E\setminus 10B_0} + \indicator{10B_0\setminus F}}{(\grapha(x)-z+is)^2} \dl {μ(z)}  }^2  |s| \dl s  \dl x.
  \end{align}
  The contribution of $10B_0\setminus F$ is controlled by the localized USFE bound.  Indeed, for the points $ζ=\grapha(x)+is$ occurring above, we have $ζ\in8B_0$ and, by \cref{thm:compare-t-dist}, $|s|\lesssim\dist(ζ,E)$.  Hence the change of variables from $(x,s)$ to $ζ$ and \ref{item:main-lemma-usfe} of \cref{thm:main-lemma} give
  \begin{align}
     & \int_{6I_0} \int_{\{|s|\in (\frac{τ}{2}D(x),\sqrt ε)\}}
    \abs[\Big]{ \int_{10B_0\setminus F} \frac{1}{(\grapha(x)-z+is)^2} \dl {μ(z)}  }^2 |s|\dl s\dl x \\
     & \qquad\lesssim
    \int_{8B_0\setminus E}
    \abs[\Big]{\int_{10B_0\setminus F}\frac{1}{(ζ-z)^2}\dl{μ(z)}}^2
    \dist(ζ,E)\dl{m_2(ζ)}                                                                           \\
     & \qquad\le M μ(10B_0\setminus F)
    \le Mη \lesssim ε.
  \end{align}
  For the other part, $x\in 6I_0$ and $|s| < \sqrt ε$ implies that $\grapha(x)+is\in 8B_0$ and thus $|\grapha(x)-z+is|\ge 1$ for $z\notin 10B_0$.
  Therefore, by the tail assumption \ref{item:main-lemma-tail} of \cref{thm:main-lemma},
  \begin{align}
    \abs[\Big]{ \int_{\CC\setminus 10B_0} \frac{1}{(\grapha(x)-z+is)^2} \dl{μ(z)} }
     & \le \int_{\CC\setminus 10B_0} \frac{1}{|\grapha(x)-z+is|^2} \dl{μ(z)}
    \le C_{\operatorname{tail}}.
  \end{align}
  Plugging this back into \cref{eq:muF-mu-small-intermediate} gives a contribution bounded by
  \begin{align}
    C_{\operatorname{tail}}^2
    \int_{6I_0}\int_0^{\sqrt ε} |s|\dl s\dl x
    \lesssim C_{\operatorname{tail}}^2 ε.
  \end{align}
  The parameter $ε$ is chosen sufficiently small in terms of $C_{\operatorname{tail}}$, so this is acceptable.
\end{proof}

\subsection{Comparison of outer integrals}
In the last comparison step, we relate the outer integrals.
\begin{lemma}\label{thm:outer-integrals-comparison}
  For the outer integrals, we have
  \begin{align}
    \int_{5I_0} \paren[\Big]{ \int_{\tilde{\Gamma}(a)} & \abs[\Big]{ \int_{\CC} θ_*(ξ,z) \dl{μ(z)} }^2  \dl{ξ} }^{p/2}  \dl a                                                                  \\
                                                       & \lesssim \paren[\Big]{ \int_{100B_0\setminus E} \abs[\Big]{ \int_{\CC} \frac{1}{(ζ-z)^2} \dl{μ(z)} }^2  \dist(ζ,E) \dl m_2(ζ) }^{p/2}
  \end{align}
\end{lemma}
\begin{proof}
  We combine \cref{thm:conical-to-vertical} and \cref{thm:compare-t-dist} to get the upper bound
  \begin{align}
    \int_{6I_0} \int_{\{|s|\in (\frac{τ}{2}D(x),\sqrt ε)\}} & \abs[\Big]{ \int_{\CC} \frac{1}{(\grapha(x)+is-z)^2} \dl{μ(z)} }^2  |s| \dl s  \dl x                                                                                     \\
                                                            & \lesssim \int_{6I_0} \int_{\{|s|\in (\frac{τ}{2}D(x),\sqrt ε)\}} \abs[\Big]{ \int_{\CC} \frac{1}{(\grapha(x)+is-z)^2} \dl{μ(z)} }^2  \dist(\grapha(x)+is,E) \dl s  \dl x \\
                                                            & \le \int_{100B_0\setminus E} \abs[\Big]{ \int_{\CC} \frac{1}{(ζ-z)^2} \dl{μ(z)} }^2  \dist(ζ,E) \dl m_2(ζ).
  \end{align}
\end{proof}

\section{Size of \texorpdfstring{$F_1$}{F1} and \texorpdfstring{$F_2$}{F2} and the conclusion of the proof of the Main Lemma}\label{sec:F1}\label{sec:F2}
We adopt two results from \cite{tolsaAnalyticCapacityCauchy2014} regarding the sizes of $F_1$ and $F_2$, respectively.
\begin{lemma}[{\cite[Section 7.7]{tolsaAnalyticCapacityCauchy2014}}]\label{thm:f1-small}
  We have
  \begin{align}
    μ(F_1) \lesssim δ.
  \end{align}
  Here, $δ$ is the lower bound for the good ball density (cf. \cref{sec:graph-construction}).
\end{lemma}
The idea of the proof is to use that any point in $F_1$ is the centre of a ball with density less than $δ$ and these points are all close to the Lipschitz graph we constructed.
The details can be found in \cite[Section 7.7]{tolsaAnalyticCapacityCauchy2014}.
\begin{lemma}[{\cite[Lemma 7.35]{tolsaAnalyticCapacityCauchy2014}}]\label{thm:f2-measure-estimate-A-alpha}
  If $ε$ is sufficiently small, then
  \begin{align}
    μ(F_2) \lesssim \frac{\norm{A'}_2^2}{α^2}.
  \end{align}
\end{lemma}

\begin{lemma}\label{thm:f2-small}
  We have
  \begin{align}
    μ(F_2) \lesssim α^{1/3} + \frac{\sqrt{λ_2}}{α}.
  \end{align}
\end{lemma}
\begin{proof}
  We put together all localization and comparison estimates, i.e.
  \crefnosort{thm:outer-localisation-lp,thm:upper-cone-localisation-lp,thm:lower-cone-localisation-lp,thm:dt-Integral-outsideRegion-small,thm:dy-hdH-lp,thm:muF-hdH-estimate,thm:dmu-innerIntegral-outsideRegion-small,thm:muF-mu-small,thm:outer-integrals-comparison}, and finish with \cref{item:usf-small} of \cref{thm:main-lemma} to get
  \begin{align}
    \norm{T^{∨} 1}_p
    &\lesssim ε + \sqrt{ε} + \sqrt{τ} α +
    \brac[\Bigg]{ \int_{5I_0} \paren[\Big]{ \int_{\tilde{\Gamma}(a)} \abs[\Big]{ \int_{\RR} θ(ξ,y) \dl y}^2 \dl{ξ} } ^{p/2} \dl a }^{1/p} \\
    & \lesssim \sqrt{τ}α + \sqrt ε + α^{2/p} + ε^{1/4} + \sqrt ε + \sqrt ε +
    \brac[\Bigg]{\int_{5I_0} \paren[\Big]{ \int_{\tilde{\Gamma}(a)} \abs[\Big]{ \int_{\CC} θ_*(ξ,z) \dl{μ(z)}  }^2  \dl{ξ}}^{p/2}  \dl a }^{1/p} \\
    & \lesssim α^{2/p} + \brac[\Bigg]{ \int_{\substack{ζ\in 100B_0\setminus E                                     \\ \dist(ζ,E)\le 10}} \abs[\Big]{ \int_{\CC} \frac{1}{(ζ-z)^2} \dl{μ(z)} }^2  \dist(ζ,E) \dl m_2(ζ) }^{1/2} \\
    & \lesssim α^{2/p} + \sqrt{λ_2}.
  \end{align}
  Now, \crefnosort{thm:f2-measure-estimate-A-alpha,thm:lower-bound-conical} (with $p=3/2$) and the above estimate let us conclude
  \begin{align}
    μ(F_2) \lesssim \frac{\norm{A'}_2^2}{α^2}
     \lesssim \frac{1}{α} \norm{T^{∨} 1}_{3/2} + α^2 
     \lesssim α^{1/3} + \frac{\sqrt{λ_2}}{α} + α^2.
  \end{align}
\end{proof}

\begin{proof}[Proof of \cref{thm:main-lemma}]
  The graph $\gamma$ has been constructed in \cref{sec:graph-construction}.  We recall the final bookkeeping.  The sets $Z$, $F_1$, and $F_2$ form a partition of the set $F$ appearing in \cref{thm:main-lemma}.  Moreover, by \cref{thm:graph-properties},
  \begin{align}
    Z\subset Z_0\subset \gamma .
  \end{align}
  Hence
  \begin{align}
    \mu(B_0\cap F\cap \gamma)
      &\ge \mu(B_0\cap Z) \\
      &= \mu(B_0\cap F)-\mu(B_0\cap F_1)-\mu(B_0\cap F_2) \\
      &\ge \mu(B_0)-\mu(B_0\setminus F)-\mu(F_1)-\mu(F_2).
  \end{align}
  Since $B_0\subset 10B_0$, the first assumption of \cref{thm:main-lemma} gives
  \begin{align}
    \mu(B_0\setminus F)\le \mu(10B_0\setminus F)\le \eta .
  \end{align}
  Also $\delta_\mu(B_0)=1$, so $\mu(B_0)=1$.

  \Cref{thm:f1-small} gives $\mu(F_1)\le C\delta$, where $\delta$ is the lower density threshold used to define good balls.
  We choose $\delta$ small enough that this contribution is less than, say, $1/8$.
  \Cref{thm:f2-small} gives the estimate
  \begin{align}
    \mu(F_2)\lesssim \alpha^{1/3}+\frac{\sqrt{\lambda_2}}{\alpha}.
  \end{align}
  After $\delta$ is fixed, we choose $\alpha$ sufficiently small, and then $\lambda_2$ sufficiently small in terms of $\alpha$, so that this contribution is also less than $1/8$.  The parameters $\varepsilon$ and $\tau$ are chosen according to the hierarchy in \cref{sec:graph-construction}, and $\lambda_1$ is chosen small enough for the graph construction and comparison estimates above.  Finally, we choose $\eta$ small enough, in terms of the localized USFE and tail constants in \cref{thm:main-lemma}, so that $\eta<1/8$.

  With these choices,
  \begin{align}
    \mu(B_0\cap F\cap \gamma)
      \ge 1-\eta-\mu(F_1)-\mu(F_2)
      \ge \frac12 .
  \end{align}
  Since $\mu(B_0)=1$, this is the desired conclusion, with for instance $c_0=1/2$ after fixing the preceding parameters.  This proves \cref{thm:main-lemma}.
\end{proof}

\section{Blow-up}\label{sec:blow-up}
The strategy to prove \cref{thm:local-nta-rectifiable} is to use a localized blow-up analysis to show that the assumptions imply the weak-flatness property, which brings us into the setting of \cref{thm:main-theorem}.
We first record the convergence facts for sets and measures used in the compactness argument.
In \cref{sec:main-lemma-local-nta} we prove the localized blow-up lemma, reducing the obstruction to weak flatness to the Cauchy-flat description in \cref{sec:cauchy-flat}.
Finally, we apply the localized lemma in the proof of \cref{thm:local-nta-rectifiable} in \cref{sec:proof-theorem-local-nta}.

\subsection{Convergence of sets}\label{sec:sets}
For non-empty sets $A,B\subset\RR^2$ and $R>0$, we define
\begin{align}
    \excess(A,B) = \sup_{x\in A} \dist(x,B), \quad \excess(\emptyset, B) = 0,
\end{align}
and
\begin{align}
    d_R(A,B) = \max\{ \excess(A \cap \overline{B(0,R)}, B),  \excess(B \cap \overline{B(0,R)}, A) \}.
\end{align}
This leads us to the following definition of local convergence of sets.
\begin{definition}[Local convergence]
    A sequence of non-empty sets $A_j$ \emph{converges locally} to a non-empty set $A$ (we write $A_j\to A$ locally) if, for every $R>0$,
    \begin{align}
        \lim_{j\to ∞} d_R(A_j,A) = 0.
    \end{align}
\end{definition}
In variational analysis this convergence corresponds to the Attouch-Wets topology, which is usually defined by identifying non-empty closed sets with their distance functional.
\begin{lemma}[{\Cite[Theorem 3.1.7]{beerTopologiesClosedClosed1993}}]\label{thm:convergence-of-dist-local}
    Let $F_k,F\subset\RR^2$ be non-empty closed sets.
    Then $F_k\to F$ locally if and only if
    $\dist(\cdot,F_k) \to \dist (\cdot,F)$ uniformly on compact sets.
\end{lemma}
Together with the Arzelà-Ascoli theorem, this identification and characterization directly yields the following compactness result.
\begin{lemma}\label{thm:local-conv-sequential-compactness}
    Let $K$ be a compact set and $F_k\subset\RR^2$ be a sequence of closed sets with $F_k\cap K\neq\emptyset$.
    Then there exists a subsequence $F_{k_j}$ that converges locally to a closed set $F$ (with $F\cap K\neq\emptyset$).
\end{lemma}

Except for boundary effects, local convergence also carries over to the respective $β$-numbers.
\begin{lemma}\label{thm:beta-local-convergence}
    Let $F_k,F\subset\RR^2$ be non-empty closed sets such that $F_k\to F$ locally.
    Then for any $η>1$ we have $\limsup_{k\to \infty} β_{F_k}(x,r) \le η β_{F}(x,ηr)$.
\end{lemma}
\begin{proof}
    Let $η>1$, $ε>0$ and $L$ be a line.
    For all large $k$, the local convergence implies that for every $z\in F_k\cap B(x,3r)$ there exists $y_z\in F$ with $|z-y_z|<ε$.
    Therefore,
    \begin{align}
        \sup_{z\in F_k\cap B(x,3r)} \dist(z,L)
        \le ε + \sup_{y\in F\cap B(x,3ηr)} \dist(y,L),
    \end{align}
    if $ε$ is sufficiently small with respect to $η$.
    Taking the infimum over all lines $L$ and dividing by $r$ gives
    \begin{align}
        β_{F_k}(x,r)
        \le \frac{ε}{r} +\frac{1}{r}\inf_L\sup_{y\in F\cap B(x,3ηr)} \dist(y,L)
        = \frac{ε}{r} + η β_F(x,ηr).
    \end{align}
    Taking $\limsup_{k\to ∞}$ and then $ε\to 0$ concludes the proof.
\end{proof}

\begin{lemma}[Blow-up of a local two-sided NTA chart]\label{thm:local-nta-blowup-limit}
    Let $\Gamma\subset\CC$ be closed and locally two-sided NTA at $x\in\Gamma$.  Let $r_x>0$, $c_x>0$, $C_x<\infty$, and $\Omega_x^\pm$ be a local two-sided NTA chart at $x$ in the sense of \cref{def:local-two-sided-nta}.  Let $r_k\downarrow0$ and set
    \begin{align}
        R_k:=\frac{r_x}{r_k},\qquad
        \Gamma_k:=\frac{\Gamma-x}{r_k},\qquad
        \Omega_k^\pm:=\frac{\Omega_x^\pm-x}{r_k} .
    \end{align}
    Then, after passing to a subsequence, there are a closed set $\Gamma_\infty\subset\CC$ with $0\in\Gamma_\infty$ and disjoint open sets $\Omega_\infty^\pm$ such that
    \begin{enumerate}
        \item $\Gamma_k\to\Gamma_\infty$ locally;
        \item
              \begin{align}
                  \CC\setminus\Gamma_\infty=\Omega_\infty^+\,\dot\cup\,\Omega_\infty^-,
                  \qquad
                  \partial\Omega_\infty^+=\partial\Omega_\infty^-=\Gamma_\infty;
              \end{align}
        \item each $\Omega_\infty^\pm$ is connected and satisfies the corkscrew condition: for some $c_0=c_0(c_x)>0$, for every $y\in\Gamma_\infty$, $t>0$ and fixed sign there are balls
              \begin{align}
                  B(a^\pm,c_0t)\subset \Omega_\infty^\pm\cap B(y,t);
              \end{align}
              moreover, each $\Omega_\infty^\pm$ satisfies the Harnack-chain condition: for every $\eta>0$ there are constants $N_\eta,\rho_\eta>0$, depending only on $c_x$, $C_x$, and $\eta$, such that for every fixed sign, whenever       
               \begin{align}
                  p,q\in\Omega_\infty^\pm\cap B(b,t)
                  \quad\text{and}\quad
                  \dist(p,\Gamma_\infty),\ \dist(q,\Gamma_\infty)\ge \eta t,
              \end{align}
              the points $p$ and $q$ can be joined by a Harnack-chain $B_j = B(a_j,s_j)$, $j=1,...,N$ with
              \(B_j \subset \Omega_∞^\pm \cap B(b,C_x t)\), $s_j \approx \dist(B_j, Γ_∞)$, \(N\le N_\eta\), consecutive balls intersecting, $p\in B_1$ and $q\in B_N$.
        \item compact subsets of the limiting sides eventually remain on the corresponding side: if $K$ is a compact subset of $\Omega_\infty^\pm$, then $K\subset\Omega_k^\pm$ for all sufficiently large $k$.  Also, if $K$ is a compact subset of $\CC\setminus\Gamma_\infty$, then $K\cap\Gamma_k=\emptyset$ for all sufficiently large $k$.
    \end{enumerate}
\end{lemma}
\begin{proof}
    We work in the scaled coordinates above.  For every fixed $R>0$ and all large $k$, the local chart identities give
    \begin{align}\label{eq:scaled-local-chart-identity}
        B(0,R)\setminus\Gamma_k
        = (\Omega_k^+\cup\Omega_k^-)\cap B(0,R),
    \end{align}
    and the local corkscrew and Harnack-chain estimates hold in $B(0,R)$ with the same constants, since $R<R_k$ for all large $k$.

    By \cref{thm:local-conv-sequential-compactness}, and since $0\in\Gamma_k$ for every $k$, we may pass to a subsequence so that $\Gamma_k\to\Gamma_\infty$ locally for some closed set $\Gamma_\infty$ with $0\in\Gamma_\infty$.

    We next define the two limiting sides.  Let $\mathcal B$ be the countable family of balls with rational centre and rational radius whose closures are contained in $\CC\setminus\Gamma_\infty$.  If $B\in\mathcal B$, then $\overline B\cap\Gamma_k=\emptyset$ for all large $k$.  For such $k$, \cref{eq:scaled-local-chart-identity} and the connectedness of $B$ imply that either $B\subset\Omega_k^+$ or $B\subset\Omega_k^-$.  Passing to one more diagonal subsequence, we may assume that, for every $B\in\mathcal B$, this sign is eventually constant.  Let $\mathcal B^\pm$ be the subfamilies with eventual sign $\pm$, and set
    \begin{align}
        \Omega_\infty^\pm:=\bigcup_{B\in\mathcal B^\pm} B .
    \end{align}
    Then $\Omega_\infty^\pm$ are open.  They cover $\CC\setminus\Gamma_\infty$, because every point outside $\Gamma_\infty$ is contained in a ball from $\mathcal B$.  They are disjoint: if a ball from $\mathcal B^+$ met a ball from $\mathcal B^-$, a smaller ball from $\mathcal B$ would eventually be contained in both $\Omega_k^+$ and $\Omega_k^-$, which is impossible.  If $K$ is a compact subset of $\Omega_\infty^\pm$, then the finite subcover property gives $K\subset\Omega_k^\pm$ for all sufficiently large $k$.  The assertion for compact subsets of $\CC\setminus\Gamma_\infty$ follows from the local convergence of $\Gamma_k$.

    We record a consequence of the construction that will be used below.
    Suppose that $k_j$ is a subsequence, $a_j\to a$, $s_j\to s>0$, and for a fixed sign $B(a_j,s_j)\subset\Omega_{k_j}^\pm$ for all $j$.  Then, for every $0<\theta<1$, $B(a,\theta s)\subset\Omega_\infty^\pm$.  Indeed, every ball from $\mathcal B$ whose closure is contained in $B(a,\theta s)$ is contained in $B(a_j,s_j)$ for all large $j$, and hence cannot have the opposite eventual sign.

    Let $y\in\Gamma_\infty$ and $t>0$.  Choose $y_k\in\Gamma_k$ with $y_k\to y$.
    Since $B(y_k,t/2)\subset B(0,R)$, we can apply the scaled local corkscrew condition for all large $k$.
    Thus, for each sign, there are balls
    \begin{align}
        B(a_k^\pm,c_x t/2)\subset\Omega_k^\pm\cap B(y_k,t/2).
    \end{align}
    Passing to subsequences of the centres and applying the preceding paragraph gives,
    \begin{align}
        B(a^\pm,c_x t/4)\subset\Omega_\infty^\pm\cap B(y,t).
    \end{align}
    Hence every point of $\Gamma_\infty$ is a boundary point of both sides.  Since $\CC\setminus\Gamma_\infty$ is the disjoint union of the two open sides, no point outside $\Gamma_\infty$ lies in either boundary.  Therefore
    \begin{align}
        \partial\Omega_\infty^+=\partial\Omega_\infty^-=\Gamma_\infty .
    \end{align}

    Finally we pass the Harnack-chain condition to the limit.  Fix a sign and let $p,q\in\Omega_\infty^\pm$ lie in a ball $B(b,t)$ with
    \begin{align}
        \dist(p,\Gamma_\infty),\ \dist(q,\Gamma_\infty)\ge \eta t
    \end{align}
    for some $\eta>0$.  By the compact-subset assertion above and \cref{thm:convergence-of-dist-local}, for all large $k$ the points $p,q$ lie in $\Omega_k^\pm$ and are at distance at least $\eta t/2$ from $\Gamma_k$.  The local Harnack-chain condition gives constants $N_\eta,\rho_\eta,\tau_\eta>0$, depending only on $x$ and $\eta$, and, for all large $k$ such that $B(b,(1+τ_η)C_x t)\subset B(0,R_k)$, a chain of balls
    \begin{align}
        B_{k,j}=B(a_{k,j},s_{k,j}),\qquad j=1,\ldots,N_k,
    \end{align}
    joining $p$ to $q$, with $N_k\le N_\eta$, $s_{k,j}\ge\rho_\eta t$, consecutive balls intersecting, and
    \begin{align}
        B(a_{k,j},(1+\tau_\eta)s_{k,j})\subset\Omega_k^\pm\cap B(b,(1+τ_η)C_x t).
    \end{align}
    Passing to a further subsequence, we may suppose that $N_k=N$, $a_{k,j}\to a_j$, and $s_{k,j}\to s_j$ for each $j$.  Applying the observation above to the enlarged balls shows that
    \begin{align}
        B_j^*:=B(a_j,(1+\tau_\eta/2)s_j)\subset\Omega_\infty^\pm .
    \end{align}
    Since the original consecutive balls intersect, $|a_j-a_{j+1}|\le s_j+s_{j+1}$; hence the balls $B_j^*$ intersect consecutively.  Also $p\in B_1^*$ and $q\in B_N^*$.  Thus the balls $B_j^*$ form a Harnack-chain in $\Omega_\infty^\pm$ joining $p$ to $q$.  This proves the limiting Harnack-chain condition,
    and, in particular each $Ω_∞^{\pm}$ is connected.
\end{proof}

\subsection{Convergence of measures}\label{sec:measures}~
\begin{definition}
    Let $(μ_k)_k$ be a sequence of (Radon) measures.
    We say that \emph{$μ_k$ converges to $μ$ weakly} (we write $μ_k\weaklyto μ$) if
    \begin{align}
        \lim_{k\to \infty} \int f \dl {μ_k} = \int f \dl {μ} ,\qquad\forall f\in C_c(\RR^2),
    \end{align}
    where $C_c(\RR^2)$ denotes the continuous functions with compact support.
\end{definition}
\begin{lemma}\label{thm:weak-limit-squared-conv}
    Let $(ρ_k)_k$ be a sequence of positive Radon measures with $ρ_k\weaklyto ρ$, and let $z\in \CC\setminus\supp ρ$.
    Assume that there exists $r_z>0$ such that $B(z,r_z)\cap\supp ρ_k=\emptyset$ for all sufficiently large $k$.
    Assume moreover that
    \begin{align}\label{eq:kernel-tail-tightness}
        \lim_{N\to\infty}\limsup_{k\to\infty}
        \int_{\CC\setminus B(0,N)} \frac{1}{1+|ξ|^2}\dl{ρ_k(ξ)}=0,
    \end{align}
    and that the same tail integral is finite for $ρ$.
    Then
    \begin{align}
        \int \frac{1}{(z-ξ)^2} \dl{ρ_k(ξ)} \to \int \frac{1}{(z-ξ)^2} \dl{ρ(ξ)}.
    \end{align}
\end{lemma}
\begin{proof}
    Let $χ\in C(\CC)$ be equal to $0$ on $B(z,r_z/4)$ and equal to $1$ on $\CC\setminus B(z,r_z/2)$.
    For $N>2|z|+2$, choose $φ_N\in C_c(\CC)$ with $φ_N=1$ on $B(0,N)$ and $φ_N=0$ outside $B(0,2N)$.
    Then
    \begin{align}
        ξ\mapsto \frac{χ(ξ)φ_N(ξ)}{(z-ξ)^2}
    \end{align}
    is continuous and compactly supported, so the corresponding integrals against $ρ_k$ converge to the integral against $ρ$.
    The tails are uniform in $k$ by \cref{eq:kernel-tail-tightness}; indeed, for $|ξ|>N>2|z|+2$ one has $|z-ξ|\gtrsim 1 +|ξ|$.
    Letting first $k\to\infty$ and then $N\to\infty$ proves the claim.
\end{proof}

\begin{definition}
    Let $μ$ be a measure. We call $ν$ a \emph{tangent measure} of $μ$ (at $x$), if there exists a sequence of positive numbers $r_i\to 0$ such that
    \begin{align}
        \frac{μ(x+r_i\cdot)}{r_i} \weaklyto ν .
    \end{align}
\end{definition}

\subsection{Cauchy-flat measures with finitely many components}
\label{sec:cauchy-flat}
Cauchy-flat measures were introduced by David and Semmes in their study of the usual square function and uniform rectifiability (cf. \cite[Definition III.2.30]{davidAnalysisUniformlyRectifiable1993}).
\begin{definition}\label{def:cauchy-flat}
    We say that a positive Radon measure $μ$ with at most linear growth is \emph{Cauchy-flat} if
    \begin{align}
        \int \frac{1}{(ζ-ξ)^2} \dl{μ(ξ)} = 0 ,\qquad\forall ζ\in\CC\setminus \supp μ.
    \end{align}
\end{definition}
These measures will arise as \enquote{extremal measures} for USFE during the blow-up procedure in \cref{sec:main-lemma-local-nta}.
\begin{proposition}\label{thm:cauchy-flat-finite-components}
    Let $ν$ be a non-trivial positive Radon measure on $\CC$ with
    \begin{itemize}
        \item at most linear growth,
        \item $\supp ν$ having empty interior, and
        \item $\CC\setminus \supp ν$ consisting of finitely many connected components.
    \end{itemize}
    If $ν$ is Cauchy-flat, then $\CC\setminus \supp ν$ consists of at least two distinct connected components and each component is a convex, possibly unbounded, polygonal domain.
    Furthermore, $ν$ is a finite positive linear combination of one-dimensional Hausdorff measure restricted to the sides of these polygonal domains.
\end{proposition}
The proof is given in \cref{app:proof-cauchy-flat-finite-components}.

We remark that the converse of \cref{thm:cauchy-flat-finite-components} is not true.
Consider the measure
\begin{align}
    μ:= μ_1 + μ_2 := c_1\hd1\restriction_\RR + c_2 \hd 1\restriction_{\{\Re z = 0,\,\Im z \ge 0\}}
\end{align}
for $c_1,c_2>0$.
It is not difficult to see that $μ_1$ is Cauchy-flat while $μ_2$ is not.

\subsection{Localized blow-up lemma}\label{sec:main-lemma-local-nta}
Our goal is to derive weak flatness from a pointwise blow-up argument.  The next lemma is the form in which the blow-up is used in the proof of \cref{thm:local-nta-rectifiable}.

\begin{lemma}\label{thm:blow-up-lemma}
    Let $μ$ be a Radon measure, set $E:=\supp μ$, and assume that $E\subset Γ$.  Fix $C,ε,δ>0$.
    There does not exist a Borel set $H\subset E$ with $μ(H)<\infty$, a point $x\in H$ at which $Γ$ is locally two-sided NTA, and scales $r_k\downarrow0$ such that
    \begin{enumerate}
        \item there is $r_H>0$ such that $μ(B(y,r))\le Cr$ for all $y\in H$ and all $0<r<r_H$;
        \item for every $R>0$,
              \begin{align}\label{eq:localized-density-point}
                  μ(B(x,Rr_k)\setminus H)=o(r_k),
                  \quad\text{ i.e. } \frac{μ(B(x,Rr_k)\setminus H)}{r_k} \to 0;
              \end{align}
        \item it holds that
              \begin{align}\label{eq:localized-bad-scale}
                  β_E(x,50r_k)\ge ε
                  \quad\text{and}\quad
                  μ(B(x,50r_k))\ge δ r_k;
              \end{align}
        \item for every $R>0$,
              \begin{align}\label{eq:localized-energy-small}
                  \frac{1}{r_k}
                  \int_{z\in B(x,Rr_k)\setminus E}
                  \abs[\Big]{\int_H \frac{1}{(z-ξ)^2}\dl{μ(ξ)}}^2 \dist(z,E)\dl{m_2(z)}
                  \to 0 .
              \end{align}
    \end{enumerate}
\end{lemma}
\begin{proof}
    Assume that such $H,x$ and $r_k$ exist.  Let $r_x>0$ and $Ω_x^\pm$ be the local two-sided NTA chart for $Γ$ at $x$.
    Applying \cref{thm:local-nta-blowup-limit}, we obtain a (sub)sequence, a closed set $Γ_∞$ and open connected limiting sides $Ω_∞^\pm$ such that
    \begin{align}\label{eq:local-nta-blowup-limit-output}
        Γ_k\to Γ_∞ \text{ locally},\qquad
        \CC\setminus Γ_∞=Ω_∞^+\,\dot\cup\,Ω_∞^-,
        \qquad
        \partial Ω_∞^+=\partial Ω_∞^-=Γ_∞,
    \end{align}
    where\footnote{We put a tilde on $\widetilde E_k$ to not confuse it with $\supp μ_k$.}
    \begin{align}
        Γ_k:=\frac{Γ-x}{r_k},\qquad
        Ω_k^\pm:=\frac{Ω_x^\pm-x}{r_k},\qquad
        \widetilde E_k:=\frac{E-x}{r_k},\qquad
        μ_k:=\frac{(μ\restriction H)(x+r_k\cdot)}{r_k} .
    \end{align}
    The limiting sides satisfy the corkscrew condition at all boundary points and scales, with constants depending only on the local NTA constants at $x$, and compact subsets of $Ω_∞^\pm$ are eventually contained in the corresponding $Ω_k^\pm$.

    We first record compactness of the measures.  Let $B(a,s)$ be fixed.  If $H\cap B(x+r_ka,r_ks)$ is nonempty, choose $y_k$ in this intersection.
    We have
    \begin{align}
        μ(H\cap B(x+r_ka,r_ks))\le μ(B(y_k,2r_ks))\le 2C r_ks
    \end{align}
    for all large $k$.  If the intersection is empty, the same bound is trivial.  Thus
    \begin{align}\label{eq:localized-growth-limit}
        \limsup_{k\to\infty} μ_k(B(a,s))\le 2Cs .
    \end{align}
    In particular, the sequence $(μ_k)_k$ is locally bounded.  After passing to a subsequence, we have
    \begin{align}
        μ_k\weaklyto ν, \label{eq:limit-measure-blowup-main-lemma}
    \end{align}
    for a Radon measure $ν$.  By \cref{eq:localized-growth-limit}, $ν$ has at most linear growth.

    We have $\supp ν\subset Γ_∞$, since for any $z\in \CC\setminus Γ_∞$ a small ball around it is contained in $\CC\setminus Γ_k\subset \CC\setminus \supp μ_k$ for all large $k$.
    The lower density in \cref{eq:localized-bad-scale}, together with \cref{eq:localized-density-point}, also shows that $ν$ is non-zero:
    Indeed, if $φ\in C_c(B(0,51))$ satisfies $0\le φ\le1$ and $φ=1$ on $B(0,50)$, then
    \begin{align}\label{eq:limit-measure-nontrivial}
        ν(B(0,51))
         & \ge \int φ\dl{ν}
        =\lim_{k\to\infty}\int φ\dl{μ_k}                           \\
         & \ge \liminf_{k\to\infty}\frac{μ(H\cap B(x,50r_k))}{r_k}
        \ge δ .
    \end{align}
    Moreover, by \cref{eq:localized-bad-scale} and \cref{thm:beta-local-convergence},
    \begin{align}\label{eq:limit-boundary-nonflat}
        ε\le \limsup_{k\to\infty}β_{\widetilde E_k}(0,50)
        \le \limsup_{k\to\infty}β_{Γ_k}(0,50)
        \le 2β_{Γ_∞}(0,51).
    \end{align}

    We next prove that $ν$ is Cauchy-flat.  We verify the tail hypothesis in \cref{thm:weak-limit-squared-conv} for the measures $μ_k$.
    Since $x\in H$, the first assumption gives $μ(B(x,t))\le Ct$ for all $0<t<r_H$.
    A dyadic annulus decomposition in the region $N<|ξ|<r_H/r_k$ gives a contribution $\lesssim N^{-1}$.
    On the remaining tail $|ξ|>r_H/r_k$, the finiteness of $μ(H)$ gives a contribution $\lesssim_{H,r_H} r_k$.  Hence, for fixed $N$ and all large $k$,
    \begin{align}\label{eq:scaled-tail-tightness}
        \int_{\CC\setminus B(0,N)} \frac{1}{1+|ξ|^2}\dl{μ_k(ξ)}
        \lesssim \frac{1}{N}+ C_{H,r_H} r_k .
    \end{align}
    Thus the tails are tight uniformly in the sense of \cref{eq:kernel-tail-tightness}.
    The tail integral of $ν$ is finite since $ν$ has at most linear growth.

    Let $R>1$.  Changing variables in \cref{eq:localized-energy-small} yields
    \begin{align}\label{eq:scaled-energy-small}
        \int_{y\in B(0,R)\setminus \widetilde E_k}
        \abs[\Big]{\int \frac{1}{(y-ξ)^2}\dl{μ_k(ξ)}}^2 d(y,\widetilde E_k)\dl{m_2(y)}\to0 .
    \end{align}
    Let $z\in B(0,R)\setminus Γ_{∞}$.
    Then a ball around $z$ is eventually disjoint from $Γ_k$, hence also from $\supp μ_k$.
    By \cref{thm:weak-limit-squared-conv}, the Cauchy kernels for $μ_k$ converge pointwise at such $z$ to the Cauchy kernel of $ν$.
    Using \cref{thm:convergence-of-dist-local}, we get $\dist(z,\widetilde E_k)\ge \dist(z,Γ_k) \to \dist(z,Γ_{∞})$.
    Fatou's lemma applied to \cref{eq:scaled-energy-small} gives
    \begin{align}
        \int_{B(0,R)\setminus Γ_∞}
        \abs[\Big]{\int \frac{1}{(z-ξ)^2}\dl{ν(ξ)}}^2 \dist(z,Γ_∞)\dl{m_2(z)}=0 .
    \end{align}
    Since $R$ was arbitrary, the Cauchy transform of $ν$ vanishes on $Ω_∞^+\cup Ω_∞^-$.

    Let now $z\in\CC\setminus\suppν$.  If $z\in Ω_∞^+\cup Ω_∞^-$, we are done.  Otherwise $z\in Γ_∞\setminus\suppν$.  Choose $r>0$ with $B(z,r)\cap\suppν=\emptyset$.  The function
    \begin{align}
        w\mapsto \int \frac{1}{(w-ξ)^2}\dl{ν(ξ)}
    \end{align}
    is holomorphic on $B(z,r)$, and the two-sided corkscrew condition gives a nonempty open subset of $B(z,r)\cap(Ω_∞^+\cup Ω_∞^-)$ on which it vanishes.  Hence it also vanishes at $z$.  Thus $ν$ is Cauchy-flat in the sense of \cref{def:cauchy-flat}.

    We can apply \cref{thm:cauchy-flat-finite-components} to $ν$.  The measure is non-trivial by \cref{eq:limit-measure-nontrivial}, has at most linear growth by \cref{eq:localized-growth-limit}, and $\suppν$ has empty interior because it is contained in the two-sided corkscrew boundary $Γ_∞$.
    Moreover, $\CC\setminus\suppν$ has at most two connected components.  Indeed, $Ω_∞^+$ and $Ω_∞^-$ lie in components of $\CC\setminus\suppν$, and any point of $Γ_∞\setminus\suppν$ has a ball disjoint from $\suppν$ which, by the two-sided corkscrew condition, meets $Ω_∞^+\cup Ω_∞^-$.  Thus every component is one of the components containing $Ω_∞^+$ or $Ω_∞^-$, possibly with these two components merged.  Each component is corkscrew, since at every boundary point and scale it contains a corkscrew ball from one of the limiting domains.

    The proposition implies that $\CC\setminus\suppν$ has exactly two components and that both are convex polygonal domains.  Two complementary convex polygonal domains whose common complement has empty interior must be complementary half-planes; hence $\suppν$ is a line $L$.  Since $L\subset Γ_∞$ and $Ω_∞^+$, $Ω_∞^-$ are the two components of $\CC\setminus Γ_∞$ with common boundary $Γ_∞$, the inclusion is equality:
    Otherwise a point of $Γ_∞\setminus L$ would have a small ball around it lying in one of the two half-planes determined by $L$, contradicting that it is a boundary point of both limiting domains.  This gives $Γ_∞=L$, contradicting \cref{eq:limit-boundary-nonflat}.
\end{proof}

\subsection[Proof of the locally two-sided NTA theorem]{Proof of \texorpdfstring{\cref{thm:local-nta-rectifiable}}{theorem \ref{thm:local-nta-rectifiable}}}\label{sec:proof-theorem-local-nta}

\begin{proof}[Proof of \cref{thm:local-nta-rectifiable}]
    It is enough to prove that $μ$ is weakly flat; the conclusion then follows from \cref{thm:main-theorem}.

    For $N\ge1$, set $E:=\supp μ$ and
    \begin{align}
        H_N:=\cbrac{ x\in E:
            μ(B(x,r))\le Nr \text{ for }0<r<1/N,
            \text{ and } δ_μ^*(x)\ge 1/N } .
    \end{align}
    The density hypothesis gives $μ(E\setminus\bigcup_N H_N)=0$.
    Suppose now that $μ$ is not weakly flat.
    Then for some $N$ there is a subset $A\subset H_N$ of positive $μ$-measure on which weak flatness fails.
    We will use \cref{thm:blow-up-lemma} to arrive at a contradiction.

    Since $μ$ is finite, $μ(H_N)<\infty$.
    Applying \cref{thm:localized-square-function-differentiation} to $G:=H:=H_N$, gives a $μ$-nullset $O$ such that for every $x\in H_N\setminus O$ and every $n\in\NN$,
    \begin{align}\label{eq:local-nta-pointwise-energy-zero}
        \lim_{r\to0}\frac{1}{r}
        \int_{z\in B(x,nr)\setminus E }
        \abs[\Big]{\int_{H_N} \frac{1}{(z-ξ)^2}\dl{μ(ξ)}}^2 \dist(z,E)\dl{m_2(z)}=0 .
    \end{align}

    Choose a point $x\in A\setminus O$ which is a $μ$-density point of $H_N$ and at which $Γ$ is locally two-sided NTA.
    Such $x$ exists since the set of points with these properties has positive measure.
    There is a number $τ>0$ and a sequence $s_k\downarrow0$ such that
    \begin{align}
        β_E(x,s_k)\frac{μ(B(x,s_k))}{s_k}\ge τ .
    \end{align}
    Since $x\in H_N$, after discarding finitely many terms we have $μ(B(x,s_k))/s_k\le N$.
    Also $β_E(x,s_k)\le3$, because the line passing through $x$ is part of the infimum in the definition of $β$.
    Thus
    \begin{align}
        β_E(x,s_k)\ge τ/N,
        \qquad
        μ(B(x,s_k))\ge (τ/3)s_k .
    \end{align}
    Put $r_k=s_k/50$.  Then \cref{eq:localized-bad-scale} holds with $ε=τ/N$ and $δ=50τ/3$.

    The density point property gives \cref{eq:localized-density-point}: for every fixed $R>0$,
    \begin{align}
        μ(B(x,Rr_k)\setminus H)
        =μ(B(x,Rr_k)\setminus G_N)
        =o(μ(B(x,Rr_k)))=o(r_k),
    \end{align}
    using the upper bound $μ(B(x,Rr_k))\le NRr_k$ for large $k$.
    Finally, \cref{eq:local-nta-pointwise-energy-zero} implies \cref{eq:localized-energy-small}.

    Therefore, $H=H_N$, the point $x$, and the scales $r_k$ satisfy all hypotheses forbidden by \cref{thm:blow-up-lemma}, a contradiction.  Hence $μ$ is weakly flat.
\end{proof}

\appendix
\section{Properties of the graph}\label{sec:appendix}

For terminology in this \namecref{sec:appendix}, compare the construction in \cite[Section 7]{tolsaAnalyticCapacityCauchy2014}.
Before we prove \cref{thm:compare-t-dist} and the additional property claimed in \cref{thm:graph-properties}, we note the following lemma.
\begin{lemma}\label{thm:ax-aix-distance}
  For $x\in L_0\setminus π(Z_0)$ and $i\in I$ such that $φ_i(x)\neq 0$, we have
  \begin{align}
    |A(x) - A_i(x)| \lesssim εℓ(x).
  \end{align}
\end{lemma}
\begin{proof}
  Let $i,j\in I$ be such that $φ_i(x),φ_j(x)\neq 0$.
  The claim follows from summing up \cite[Lemma 7.23 (b)]{tolsaAnalyticCapacityCauchy2014} which says that $|A_j(x) - A_i(x)| \lesssim εℓ(x)$.
\end{proof}

\begin{proof}[Proof of \cref{thm:compare-t-dist}]
  The strategy of the proof is to find a good ball $B$ that contains $B(\grapha(x)+is,|s|)$.
  Then, we find a point $ξ$ in $U(L_B,εr(B))$, the $εr(B)$-strip around the line $L_{B}$ from the $β$-number, satisfying $π(ξ) = x$.
  This will allow us to calculate
  \begin{align}
    \dist(\grapha(x)+is,E)
    &\ge \min\cbrac[\big]{|s|,\dist(\grapha(x)+is,U(L_{B},εr(B)))} \\
    &\ge \min\cbrac[\big]{|s|,\cos(α)|\grapha(x)+is-ξ| - 2εr(B)}. \\
    &\gtrsim |s| - |\grapha(x)-ξ| - 2εr(B),
  \end{align}
  so that it suffices to show $|\grapha(x)-ξ| + εr(B) \le η |s|$, for some very small $η>0$.

  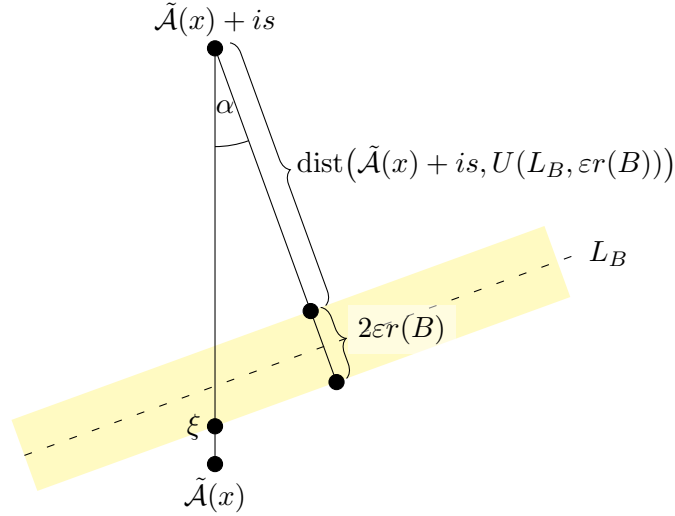
\begin{figure}
    \centering
    \begin{tikzpicture}
    \def\r{1.8}  %
    \def\w{0.5}  %
    \def\l{5}  %
    \coordinate (Ax) at (0,-0.5);
    \coordinate (xi) at (0,0);
    \coordinate (Axt) at (0,5);

    \begin{scope}[rotate=20]
        \begin{scope}[shift={(0,\w)},]
            \fill[fill=yellow!50, opacity=0.6]  (-\l/2,-\w) rectangle (\l,\w);
            \draw[loosely dashed] (-\l/2,0) -- (\l+0.3,0) node[right] {$L_B$};
            \coordinate (dir) at (\l,-\w);
        \end{scope}
    \end{scope}

    \coordinate (proj) at ($(xi)!(Axt)!(dir)$);
    \coordinate (proj2) at ($(proj)!1cm!(Axt)$);

    \draw (Ax) -- (Axt) node[midway,left]{}
    -- (proj2) node[midway,right]{};

    \draw[decorate,decoration={brace,amplitude=5pt,raise=6pt},yshift=0pt] (Axt) -- (proj2) node [midway,anchor=west,xshift=9pt,yshift=6pt]{$\dist\paren[\big]{\grapha(x)+is, U(L_B,εr(B))}$};

    \draw (proj2) -- (proj) {};
    \draw[decorate,decoration={brace,amplitude=5pt,raise=4pt},yshift=0pt] (proj2) -- (proj) node [midway,anchor=west,xshift=9pt,yshift=6.4pt,
        fill=white,        %
        fill opacity=0.5,  %
        text opacity=1,    %
        draw=none,         %
    ]{$2εr(B)$};
    \pic ["$α$", draw, angle radius=1.3cm] {angle = Ax--Axt--proj};

    \filldraw (Ax) circle (1mm) node[below=0.5mm] {$\grapha(x)$};
    \filldraw (xi) circle (1mm) node[left=0.5mm] {$ξ$};
    \filldraw (Axt) circle (1mm) node[right=0.5mm,above=0.2mm] {$\grapha(x)+is$};
    \filldraw (proj) circle (1mm) node[below=0.5mm] {};
    \filldraw (proj2) circle (1mm) node[below=0.5mm] {};

\end{tikzpicture}
    \caption{$E\cap B(\grapha(x)+is,|s|)$ is contained in $U(L_B, εr(B))$, the $2εr(B)$-wide strip around $L_B$.}\label{fig:triangle}
  \end{figure}
  
  If $x\in π(Z_0)$, then $B:=B(\grapha(x),2|s|)\in\VG$ and $ξ:=\grapha(x)\in U(L_{B},2ε|s|)$.
  If $x\not\in π(Z_0)$, consider $i\in I$ such that $φ_i(x)\neq 0$ in the definition of $A(x) = \sum_{i\in I} φ_i(x) A_i(x)$, and write $\grapha_i(t):=t+iA_i(t)$.
  Let $A_i$ be associated to the interval $R_i$ and the ball $B(z_i,r_i)\in \VG$.
  Denote $\tilde r := \max\{ r_i, |\grapha(x)+is-z_i| + |s| \}$, and $B := B(z_i,\tilde r)$.
  Applying \cite[Lemma 7.11 (b)]{tolsaAnalyticCapacityCauchy2014} to $B$ and $B(z_i,r_i)$ yields that $\angle(L_{B(z_i,r_i)}, L_B)\lesssim ε$.
  Therefore, the $ξ\in L_{B}$ with $π(ξ) = x$ satisfies
  \begin{align}
    |\grapha_i(x) - ξ| \lesssim (\tan(α+ε) - \tan(α)) |x-π(z_i)| + ε\tilde r \lesssim ε |x-π(z_i)| + ε\tilde r.
  \end{align}
  Since $φ_i(x)\neq 0$, we have $x\in 3R_i$.
  Together with \cite[Lemmas 7.22 (b) and 7.20 (a)]{tolsaAnalyticCapacityCauchy2014}, we get
  \begin{equation}
    |x - π(z_i)| \le \dist(x,R_i) + \dist(R_i,π(z_i))
    \lesssim ℓ(R_i)
    \lesssim D(x).
  \end{equation}
  Finally, this means
  \begin{align}
    |\grapha(x)-ξ| + εr(B)
    &\lesssim |\grapha(x) - \grapha_i(x)| + |\grapha_i(x) - ξ | + ε |\grapha_i(x) - z_i | + ε|s| \\
    &\lesssim εD(x)+ε|s|
    \lesssim ε\paren[\big]{ 1 + \frac{1}{τ} }|s|,
  \end{align}
  where we used \cref{thm:ax-aix-distance} and the fact that
  $\grapha_i(x)\in L_{B(z_i,r_i)}$ with $\angle(L_0, L_{B(z_i,r_i)})\le α$ implies that
  $|\grapha_i(x) - z_i| \lesssim (1+\tan α) |x-π(z_i)| + 2εr_i \lesssim D(x)$.
\end{proof}

\begin{lemma}\label{thm:dx-Dpix}
  For $x\in F\cap π^{-1}(8I_0)$, we have
  \begin{align}
    d(x) \lesssim D(π(x)).
  \end{align}
\end{lemma}
\begin{proof}
  Assume $d(x) > 0$, i.e. $x\not\in Z_0$.
  We claim that this implies $π(x)\not\in π(Z_0)$.
  Assume the contrary, i.e. there exists $y\in Z_0$ such that $π(y)=π(x)$.
  Consider the ball $B:=B(y,2|x-y|)\in\mathcal{VG}$.
  Since $\angle (L_B,L_0)\le α$ and $x,y\in E$, we have $|x-y|\lesssim \dist(x,L_B) + \dist(y,L_B) \lesssim ε|x-y|$ which is a contradiction for small $ε$.

  By construction $π(x)\in R_i$ for some $i\in I$ with corresponding ball $B:= B(z_i,r_i)$.
  By \cite[Lemma 7.22 (b)]{tolsaAnalyticCapacityCauchy2014}, $|π(x) - π(z_i)| \lesssim ℓ(x) \lesssim r_i$, say $|π(x)-π(z_i)|\le cr_i$.
  Consider $\tilde r := \max\cbrac{r_i,|x-z_i|}$.
  We claim that $\tilde r \le 2cr_i$.
  Assume to the contrary that $|x-z_i| \ge 2cr_i$, implying $|\Im x - \Im z_i|>cr_i$.
  Taking into account that $\angle(L_{B(z_i,\tilde r)}, L_0)\le α$, we get
  \begin{align}
    \tilde r = |x-z_i|
    \lesssim 2 |\Im x - \Im z_i|
    &\lesssim \dist(x,L_{B(z_i,\tilde r)}) + \tan(α) |π(x)-π(z_i)| +  \dist(z_i,L_{B(z_i,\tilde r)})\\
    &\lesssim ε \tilde r + \tan(α)c\tilde r,
  \end{align}
  which is a contradiction for small $α$ and $ε$.
  
  By the definition of $d$ we get $d(x)\lesssim r_i$.
  In addition, \cite[Lemma 7.20 (a)]{tolsaAnalyticCapacityCauchy2014} yields $r_i\lesssim D(π(x))$, which completes the proof.
\end{proof}

\section[Proof of the Cauchy-flat description]{Proof of the Cauchy-flat description\except{toc}{\texorpdfstring{ (\cref{thm:cauchy-flat-finite-components})}{}}}
\label{app:proof-cauchy-flat-finite-components}

Throughout this section, let $ν$ satisfy the hypotheses of \cref{thm:cauchy-flat-finite-components}, and set $E:=\supp ν$.
Choose $z_0\in\CC\setminus E$ and define, for $z\in\CC\setminus E$,
\begin{align}
  F_ν(z):=\int \paren[\Big]{\frac{1}{z-ξ}-\frac{1}{z_0-ξ}}\dl{ν(ξ)}.
\end{align}
Also define
\begin{align}
  u(z):=\int
  \paren[\Big]{
    \log|z-ξ|-
    \log|z_0-ξ|-
    \Re\frac{z-z_0}{z_0-ξ}
  }\dl{ν(ξ)},
  \qquad z\in\CC.
\end{align}

\begin{lemma}\label{thm:cf-potential-continuity}
  The function $u$ is well-defined and continuous on $\CC$.
  More precisely, for every compact set $K\subset\CC$ there exists $C_K>0$ such that
  \begin{align}
    |u(z)-u(w)|
    \le C_K |z-w|\paren[\Big]{1+\log_+\frac{1}{|z-w|}},
    \qquad z,w\in K,
  \end{align}
  where $\log_+ := \max(0,\log(\cdot))$
\end{lemma}
\begin{proof}
  Fix a compact set $K\subset\CC$.  Choose $M>1$ so large that $K\cup\{z_0\}\subset B(0,M/4)$.
  In this proof all constants can depend on $K$, $M$ and $z_0$ without further mention.
  We denote
  \begin{align}
    K_z(ξ):=\log|z-ξ|-\log|z_0-ξ|-\Re\frac{z-z_0}{z_0-ξ}.
  \end{align}
  
  For $|ξ|>M$, Taylor expansion gives $K_z(ξ)= O(|ξ|^{-2})$ uniformly for $z\in K$.
  Since $ν$ has at most linear growth, a dyadic annulus decomposition gives
  \begin{align}
    \int_{|ξ|>M}\frac{1}{|ξ|^2}\dl{ν(ξ)}<\infty.
  \end{align}
  On $B(0,M)$, the logarithmic singularity is integrable against $ν$, again by the linear growth bound.
  Hence $u$ is well-defined.

  Let $z,w\in K$ and set $δ:=|z-w|$.
  It suffices to prove the estimate when $δ\le 1$.
  Indeed, the case $δ>1$ follows by subdividing the segment $[z,w]$ into subsegments of length at least $1/2$ and applying the $δ\le1$ estimate on the compact convex hull of $K$.
  We decompose $\CC=A_{\mathrm{loc}}\cup A_{\mathrm{mid}}\cup A_\infty$, where
  \begin{align}
    A_{\mathrm{loc}}:=B(z,2δ)\cup B(w,2δ),\quad
    A_{\mathrm{mid}}:=B(0,M)\setminus A_{\mathrm{loc}},\quad
    A_\infty:=\CC\setminus B(0,M).
  \end{align}
  On the local part, the standard shell estimate
  \begin{align}
    \int_{B(a,r)} \paren[\big]{1+|\log|a-ξ||}\dl{ν(ξ)}
    \lesssim r\paren[\big]{1+|\log r|},
    \qquad 0<r\le 1,
  \end{align}
  gives
  \begin{align}
    \int_{A_{\mathrm{loc}}}|K_z(ξ)-K_w(ξ)|\dl{ν(ξ)}
    \lesssim δ\paren[\Big]{1+\log\frac{1}{δ}}.
  \end{align}
  On $A_{\mathrm{mid}}$, $\abs[\big]{\log|z-ξ| - \log|w-ξ|} \le δ/|z-ξ|$ by the mean value theorem, and the integral of the renormalizing linear term in $K_z(ξ)$ contributes $O(δ)$.
  So another dyadic decomposition around $z$ yields
  \begin{align}
    \int_{A_{\mathrm{mid}}}|K_z(ξ)-K_w(ξ)|\dl{ν(ξ)}
    \lesssim δ\int_{A_{\mathrm{mid}}}\frac{\dl{ν(ξ)}}{|z-ξ|}+C_Mδ
    \lesssim δ\paren[\Big]{1+\log\frac{1}{δ}}.
  \end{align}
  Finally, for $|ξ|>M$,
  \begin{align}
    |K_z(ξ)-K_w(ξ)|\lesssim \frac{δ}{|ξ|^2},
  \end{align}
  and therefore
  \begin{align}
    \int_{A_\infty}|K_z(ξ)-K_w(ξ)|\dl{ν(ξ)}
    \lesssim δ\int_{|ξ|>M}\frac{\dl{ν(ξ)}}{|ξ|^2}
    \lesssim δ.
  \end{align}
  Combining these estimates proves the asserted modulus of continuity.
\end{proof}

\begin{lemma}\label{thm:cf-potential-identities}
  In the sense of distributions on $\CC$,
  \begin{align}
    \Delta u=2πν.
  \end{align}
  Moreover, on $\CC\setminus E$,
  \begin{align}
    2\partial u=F_ν.
  \end{align}
  Consequently, $u$ is affine on each connected component of $\CC\setminus E$.
\end{lemma}
\begin{proof}
  Since $\Delta\log|z-ξ|=2πδ_ξ$ distributionally and the renormalizing terms of $K_z$ are harmonic in $z$, we have $\Delta u=2πν$.
  If $z\notin E$, differentiation under the integral sign is justified on compact subsets of $\CC\setminus E$, and gives
  \begin{align}
    2\partial_z u(z)
    =\int \paren[\Big]{\frac{1}{z-ξ}-\frac{1}{z_0-ξ}}\dl{ν(ξ)}
    =F_ν(z).
  \end{align}
  Cauchy-flatness says that the complex derivative of $F_ν$ is zero on $\CC\setminus E$, so $F_ν$ is constant on each connected component of $\CC\setminus E$.
  Hence the gradient of $u$ is constant on each such component, and $u$ is affine there.
\end{proof}

Let $Ω_1,\dots,Ω_m$ be the connected components of $\CC\setminus E$.
By \cref{thm:cf-potential-identities}, there exist affine functions
\begin{align}
  λ_j(z)=\Re(\overline{c_j}z)+b_j, \qquad j=1,...,m,
\end{align}
such that $u=λ_j$ on $Ω_j$.
Let $λ_1,\dots,λ_N$ denote the distinct affine functions appearing in this list.
For each pair $α\neq β\in\{1,...,N\}$, we define the \emph{equality line}
\begin{align}
  L_{α,β}:=\{z\in\CC:λ_α(z)=λ_β(z)\}.
\end{align}

\begin{lemma}\label{thm:cf-support-lines}
  We have
  \begin{align}
    E\subset \bigcup_{α\neq β}L_{α,β},
  \end{align}
  where the union is taken over $α,β\in\{1,...,N\}$.
\end{lemma}
\begin{proof}
  Let $x\in E$ and suppose that $x$ does not belong to the union of the equality lines.
  Choose $r>0$ so that $B(x,r)$ is disjoint from every equality line.
  Since $E$ has empty interior, $B(x,r)\setminus E$ is dense in $B(x,r)$.

  Let $\mathcal A := \{ λ_j : Ω_j \cap B(x,r)\neq \emptyset\}$.
  This set is finite and non-empty.
  Because no equality line meets $B(x,r)$, for $λ_α,λ_β\in\mathcal A$, each difference $λ_α-λ_β$ has a constant sign on $B(x,r)$.
  Thus the functions in $\mathcal A$ are totally ordered on $B(x,r)$; let $λ_*$ be the largest one.

  Set $v:=u-λ_*$.
  Then $\Delta v=2πν\ge0$ in $B(x,r)$, so $v$ is subharmonic there.
  On every complementary component meeting $B(x,r)$, the function $v$ is equal to $λ_α-λ_*$ for some $λ_α\in\mathcal A$, hence $v\le0$ on $B(x,r)\setminus E$.
  By density of $B(x,r)\setminus E$ and continuity of $v$ (cf. \cref{thm:cf-potential-continuity}), it follows that $v\le0$ on $B(x,r)$.
  Since $λ_*$ occurs on a component meeting $B(x,r)$, $v=0$ on a non-empty open subset of $B(x,r)$.
  The strong maximum principle for subharmonic functions gives $v\equiv0$ in $B(x,r)$.
  Therefore $\Delta u=0$ in $B(x,r)$, so by \cref{thm:cf-potential-identities}, we conclude $ν\restriction_{B(x,r)}=0$ which contradicts $x\in\supp ν$.
\end{proof}

Let $\mathcal L$ be the finite collection of equality lines $L_{α,β}$.
We call the connected components of $\CC\setminus\bigcup_{L\in\mathcal L}L$ \emph{cells}.
By \cref{thm:cf-support-lines}, every cell is contained in $\CC\setminus E$ and therefore carries one affine piece: if $C$ is a cell, then $u=λ_{α(C)}$ on $C$ for some $α(C)\in\{1,\dots,N\}$.
We denote by an \emph{open edge} a connected component of a line $L\in\mathcal L$ after removing the finitely many intersection points of lines in $\mathcal L$, which we call \emph{vertices} of $\mathcal L$.

\begin{lemma}\label{thm:cf-edge-model}
  Let $F$ be an open edge separating two neighboring cells $C^+$ and $C^-$, and suppose
  \begin{align}
    u=λ^+ \text{ on } C^+,
    \qquad
    u=λ^- \text{ on } C^-.
  \end{align}
  Let $d_F$ be the signed distance to the line containing $F$, chosen so that $d_F>0$ on $C^+$.
  Then there exists $κ_F\ge0$ such that
  \begin{align}
    λ^+-λ^-=κ_Fd_F,
    \qquad
    u=λ^-+κ_F(d_F)_+
  \end{align}
  on a neighborhood of $F$, where $(d_F)_+ := \max\{0, d_F\}$.
  Moreover,
  \begin{align}
    ν\restriction_F=\frac{κ_F}{2π}\hd 1\restriction_F.
  \end{align}
\end{lemma}
\begin{proof}
  Since $u$ is continuous and agrees with $λ^+$ and $λ^-$ on $C^+$ and $C^-$, respectively, $λ^+=λ^-$ on $F$.
  Furthermore $λ^+ - λ^-$ is affine, so we have $λ^+ - λ^- = κ_F d_F$ for some $κ_F\in\RR$.
  On a neighborhood of $F$, this implies $u=λ^-+κ_F(d_F)_+$.

  Choose a ball $B$ centered on $F$ which meets no other open edge.
  Without loss of generality, let $F\cap B=\{x_2=0\}\cap B$ and $d_F(x_1,x_2)=x_2$.
  We have
  \begin{align}
    u= λ^-+κ_F(x_2)_+
  \end{align}
  on $B$, where $(x_2)_+ := \max\{0,x_2\}$.
  Thus $\Delta u=κ_F\hd 1\restriction_F$ on $B$.
  Since $\Delta u=2πν$ and $ν$ is positive, $κ_F\ge0$ and the claimed identity follows.
\end{proof}

\begin{lemma}\label{thm:cf-convex}
  The function $u$ is convex on $\CC$.
\end{lemma}
\begin{proof}
  It is enough to prove convexity on every line segment.
  Fix distinct points $z,w\in\CC$.
  Choose perturbations $[z_ε,w_ε]$ converging to $[z,w]$ such that each perturbed segment avoids the vertices of $\mathcal L$ and is not contained in any equality line (recall that $\mathcal L$, the collection of equality lines, is finite).
  Parameterize one such segment by
  \begin{align}
    γ_ε(t)=z_ε+t(w_ε-z_ε),\qquad t\in[0,1].
  \end{align}
  Then $g_ε(t):=u(γ_ε(t))$ is continuous and piecewise affine, with breakpoints only at transverse intersections with open edges.
  Near such a breakpoint $t_0$, \cref{thm:cf-edge-model} gives
  \begin{align}
    g_ε(t)=at+b+κ_F\paren[\big]{a_F(t-t_0)}_+
  \end{align}
  for some $a,b,a_F\in\RR$ with $a_F\neq0$ and $κ_F\ge0$.
  The slope jump at $t_0$ is $κ_F|a_F|\ge0$.
  Therefore the slopes of $g_ε$ are nondecreasing, so $g_ε$ is convex on $[0,1]$.
  Letting $ε\to0$ and using continuity of $u$ gives convexity of the restriction of $u$ to $[z,w]$.
\end{proof}

\begin{lemma}\label{thm:cf-max-representation}
  On $\CC$,
  \begin{align}
    u=\max\{λ_1,\dots,λ_N\}.
  \end{align}
\end{lemma}
\begin{proof}
  Fix $α\in\{1,\dots,N\}$ and choose a cell on which $u=λ_α$.
  Let $z_α$ be a point in this cell.
  For any $w\in\CC$, define
  \begin{align}
    h(t):=u\paren[\big]{(1-t)z_α+tw}-λ_α\paren[\big]{(1-t)z_α+tw},
    \qquad t\in[0,1].
  \end{align}
  By \cref{thm:cf-convex}, $h$ is convex.
  Since $u=λ_α$ in a neighborhood of $z_α$, $h(t)=0$ for all sufficiently small $t\ge0$.
  Hence $h(1)\ge0$, and so $u(w)\ge λ_α(w)$.
  Taking the maximum over $α$ yields
  \begin{align}
    u\ge\max\{λ_1,\dots,λ_N\}.
  \end{align}
  The reverse inequality holds on the (dense) union of cells, and then everywhere by continuity.
\end{proof}

For each $α\in\{1,\dots,N\}$, define
\begin{align}
  P_α:=\{z\in\CC:λ_α(z)>λ_β(z)\text{ for every }β\neq α\}.
\end{align}

\begin{lemma}\label{thm:cf-regions}
  The connected components of $\CC\setminus E$ are exactly the sets $P_α$ which are non-empty.
  In particular, every component of $\CC\setminus E$ is a convex, possibly unbounded, polygonal domain.
\end{lemma}
\begin{proof}
  Each $P_α$ is an intersection of finitely many open half-planes, hence is convex and polygonal.
  If $z\in P_α$, then \cref{thm:cf-max-representation} gives $u(z)=λ_α(z)$ with strict inequality against the other affine pieces in a neighborhood of $z$.
  Thus $u$ is affine near $z$, and $z\notin E$.
  Hence $P_α\subset\CC\setminus E$.

  Conversely, let $Ω$ be a connected component of $\CC\setminus E$, and suppose that $u=λ_α$ on $Ω$.
  We claim that $Ω$ is disjoint from every equality line $L_{α,β}$ with $β\neq α$.
  If $z\in Ω\cap L_{α,β}$, then $λ_α-λ_β$ is nonnegative on $Ω$ by \cref{thm:cf-max-representation} and vanishes at the interior point $z$ of $Ω$.
  Since $λ_α-λ_β$ is affine, this is possible only if $λ_α=λ_β$ as affine functions, contradicting the choice of distinct affine pieces.

  Therefore the sign of $λ_α-λ_β$ is constant on $Ω$ for each $β\neq α$.
  Since $u=λ_α=\max\{λ_1,\dots,λ_N\}$ on $Ω$, we have $λ_α>λ_β$ on $Ω$ for every $β\neq α$.
  Thus $Ω\subset P_α$.
  Since $P_α$ is connected, contained in $\CC\setminus E$, and meets $Ω$, we also have $P_α\subset Ω$.
\end{proof}

\begin{proof}[Proof of \cref{thm:cauchy-flat-finite-components}]
  If $\CC\setminus E$ had only one connected component, then only one affine piece would occur and \cref{thm:cf-support-lines} would force $E=\emptyset$, which contradicts that $ν$ is non-trivial.
  Thus $\CC\setminus E$ has at least two connected components.
  By \cref{thm:cf-regions}, every component is a convex, possibly unbounded, polygonal domain.

  The support $E$ is contained in the finite union of equality lines by \cref{thm:cf-support-lines}.
  Away from the finitely many vertices of $\mathcal L$, a point of $E$ lies on a single open edge.
  If the same affine piece occurred on both sides of such an edge $F$, \cref{thm:cf-edge-model} would imply that $ν\restriction_F=0$.
  Hence $E$ is contained in the union of the open edges on which the adjacent affine pieces differ, together with the finitely many vertices of $\mathcal L$.
  Taking into account \cref{thm:cf-max-representation}, we conclude that $E = \bigcup_{α=1}^N \partial P_α$.
  The vertices carry no $ν$-mass by the upper linear growth bound.
  Applying \cref{thm:cf-edge-model} to the remaining edges gives constants $a_j>0$ and sides $e_j$ of the polygonal decomposition such that
  \begin{align}
    ν=\sum_j a_j\hd 1\restriction_{e_j}.
  \end{align}
  This proves the proposition.
\end{proof}

\section{A support-dust obstruction with vanishing \texorpdfstring{$\alpha$}{alpha}-numbers}\label{app:finite-support-dust-example}

The examples in this appendix show that the weak flatness hypothesis in \cref{thm:main-theorem} cannot be replaced merely by a measure-theoretic flatness condition.  The construction has two parts.  First we isolate a non-rectifiable measure, due to Preiss \cite{preissGeometryMeasures$mathbbR^n$1987}, which has linear growth, finite support length, positive finite upper density, and vanishing Tolsa's $\alpha$-numbers \cite{tolsaUniformRectifiabilityCalderon2009}.
We present the reduction to a compact density-controlled piece in detail so that the uniform growth condition is transparent.
We then add a very small amount of rectifiable dust to the support of the \enquote{core measure}.
The dust makes the support-distance weight in USFE small near the core while carrying negligible density at almost every point of the core.

If $B=B(x,r)$ and $\sigma,\tau$ are locally finite Radon measures, write
\begin{align}\label{eq:F-B}
  F_B(\sigma,\tau)
  :=
  \sup_{\substack{\varphi\in \operatorname{Lip}_1(\CC)\\ \supp\varphi\subset B}}
  \abs[\Big]{\int \varphi\,d\sigma-\int\varphi\,d\tau} .
\end{align}
Here $\operatorname{Lip}_1(\CC)$ denotes the class of $1$-Lipschitz functions on $\CC$.
The $\alpha$-number used below is
\begin{align}\label{eq:appendix-line-alpha}
  \alpha_{\mu}(x,r)
  :=
  \frac{1}{r^2}
  \inf_{c\ge0,\ L}
  F_{B(x,r)}\left(\mu,c\hd1\restriction_L\right),
\end{align}
where the infimum is over all lines $L\subset\CC$.  We also write
\begin{align}
  N_r(A):=\{z\in\CC:\dist(z,A)<r\}.
\end{align}

\begin{lemma}[Flat tangents imply vanishing $\alpha$-numbers]
  Let $\Phi$ be a finite Radon measure on $\CC$, and suppose that, at a point $x$,
  \begin{align}\label{eq:flat-tangent-alpha-density-bound}
    \limsup_{r\to0}\frac{\Phi(B(x,r))}{r}<\infty .
  \end{align}
  If every non-zero tangent measure of $\Phi$ at $x$ is of the form $c\hd1\restriction_L$, with $c>0$ and $L$ a line, then
  \begin{align}
    \alpha_{\Phi}(x,r)\to0
    \qquad\text{as }r\to0 .
  \end{align}
\end{lemma}
\begin{proof}
  Suppose not.  Then there are $\epsilon_0>0$ and $r_j\downarrow0$ such that
  \begin{align}\label{eq:bad-alpha-sequence}
    \alpha_{\Phi}(x,r_j)>\epsilon_0
    \qquad\text{for every }j.
  \end{align}
  Let
  \begin{align}
    T_{x,r}(y):=\frac{y-x}{r},
    \qquad
    \Phi_{x,r}:=\frac1r T_{x,r\#}\Phi .
  \end{align}
  By \cref{eq:flat-tangent-alpha-density-bound}, the sequence $(\Phi_{x,r_j})_j$ is locally uniformly bounded.  Passing to a subsequence, assume that $\Phi_{x,r_j}$ converges weakly to a tangent measure $\Psi$ of $\Phi$ at $x$.

  The limit $\Psi$ is non-zero.  Indeed, since the zero measure is allowed in the infimum defining \cref{eq:appendix-line-alpha}, \cref{eq:bad-alpha-sequence} gives
  \begin{align}
    F_{B(0,1)}(\Phi_{x,r_j},0)>\epsilon_0
  \end{align}
  for every $j$.  Hence we may choose $1$-Lipschitz functions $g_j$ supported in $B(0,1)$ with
  \begin{align}
    \abs[\Big]{\int g_j\,d\Phi_{x,r_j}}>\epsilon_0/2 .
  \end{align}
  After passing to a further subsequence, $g_j$ converges uniformly to a $1$-Lipschitz function $g$ supported in $\overline{B(0,1)}$.  The weak convergence then gives $|\int g\,d\Psi|\ge\epsilon_0/2$, so $\Psi\ne0$.

  By assumption,
  \begin{align}
    \Psi=c\hd1\restriction_L
  \end{align}
  for some $c>0$ and some line $L$.  If $g$ is $1$-Lipschitz and supported in $B(0,1)$, then $f_j(y):=r_j g(T_{x,r_j}(y))$ is $1$-Lipschitz and supported in $B(x,r_j)$.  Hence the weak convergence $\Phi_{x,r_j}\to\Psi$ gives
  \begin{align}
    \frac1{r_j^2}
    F_{B(x,r_j)} \paren[\Big]{\Phi,r_jT_{x,r_j}^{-1}{}_{\#}\Psi}
    =F_{B(0,1)}(\Phi_{x,r_j},\Psi)\to0 .
  \end{align}
  Since $r_jT_{x,r_j}^{-1}{}_{\#}\Psi$ is a constant multiple of $\hd1$ on a line, this contradicts \cref{eq:bad-alpha-sequence}.
\end{proof}

\begin{theorem}[A Preiss-type measure with vanishing $\alpha$-numbers]\label{thm:preiss-core-alpha}
  There is a compactly supported finite Radon measure $\nu$ on $\CC$, with $K:=\supp\nu$, such that
  \begin{enumerate}
    \item $\hd1(K)<\infty$;
    \item $\nu$ has linear growth, that is, $\nu(B(x,r))\le Cr$ for all $x\in\CC$ and $r>0$;
    \item $0<\delta_\nu^*(x)<\infty$ for $\nu$-almost every $x$;
    \item $\displaystyle \liminf_{r\to0}\nu(B(x,r))/r=0$ for $\nu$-almost every $x$;
    \item $\alpha_{\nu}(x,r)\to0$ for $\nu$-almost every $x$;
    \item $\nu$ is not $1$-rectifiable.
  \end{enumerate}
\end{theorem}
\begin{proof}
  The starting point is an example of Preiss \cite[Example~5.9~(2)]{preissGeometryMeasures$mathbbR^n$1987}.  It gives a finite non-zero Radon measure $\Phi$ on $\RR^2$ such that
  \begin{align}\label{eq:preiss-example-densities}
    \limsup_{r\to0}\frac{\Phi(B(x,r))}{r}=2,
    \qquad
    \liminf_{r\to0}\frac{\Phi(B(x,r))}{r}=0,
  \end{align}
  and every non-zero tangent measure of $\Phi$ at $x$ is a constant multiple of arclength measure on a line, for $\Phi$-almost every $x$.  By the preceding lemma,
  \begin{align}\label{eq:preiss-alpha-zero}
    \alpha_{\Phi}(x,r)\to0
    \qquad\text{for }\Phi\text{-almost every }x.
  \end{align}
  The measure $\Phi$ is not $1$-rectifiable (in fact purely unrectifiable).  Indeed, Preiss's rectifiability criterion \cite[Theorem~5.6]{preissGeometryMeasures$mathbbR^n$1987} would force a positive finite one-dimensional density at $\Phi$-almost every point, contradicting the zero lower density in \cref{eq:preiss-example-densities}.

  We now extract a compact piece on which the global growth and finite support-length bounds are transparent.  For integers $M,N,Q\ge1$, let $E_{M,N,Q}$ be the Borel set of all points $x\in B(0,Q)$ such that
  \begin{align}\label{eq:EMN-upper}
    \Phi(B(x,r))\le Mr
    \qquad\text{whenever }0<r<1/N,
  \end{align}
  and such that, for every $0<s<1/N$, there is a radius $0<r<s$ with
  \begin{align}\label{eq:EMN-lower}
    \Phi(B(x,r))\ge M^{-1}r .
  \end{align}
  The sets $E_{M,N,Q}$ are Borel, and by \cref{eq:preiss-example-densities}, their union has full $\Phi$-measure.
  Each $E_{M,N,Q}$ has finite $\hd1$ measure:
  Indeed, given $0<\delta<1/N$, choose for each $x\in E_{M,N,Q}$ a radius $r_x<\delta$ for which \cref{eq:EMN-lower} holds.  By the Vitali covering theorem, there is a disjoint subcollection $B(x_i,r_i)$ such that $E_{M,N,Q}$ is covered by the balls $B(x_i,5r_i)$.  Hence
  \begin{align}
    \hd1_{10\delta}(E_{M,N,Q})
    \lesssim \sum_i r_i
    \le M\sum_i \Phi(B(x_i,r_i))
    \le M\Phi(B(0,Q+2))<\infty .
  \end{align}
  Letting $\delta\downarrow0$ proves $\hd1(E_{M,N,Q})<\infty$.

  Since $\Phi$ is not rectifiable and is carried by the countable union of the sets $E_{M,N,Q}$, there are $M,N,Q$ such that $\Phi\restriction_{E_{M,N,Q}}$ is not rectifiable.  By inner regularity and an increasing compact exhaustion of $E_{M,N,Q}$, there is a compact set
  \begin{align}
    K_0\subset E_{M,N,Q}
  \end{align}
  such that
  \begin{align}
    \nu:=\Phi\restriction_{K_0}
  \end{align}
  is still not rectifiable; otherwise $\Phi\restriction_{E_{M,N,Q}}$ would be a countable sum of rectifiable restrictions.  Let $K:=\supp\nu\subset K_0$.  Then $\hd1(K)<\infty$.

  We next prove linear growth.  If $B(y,r)\cap K=\emptyset$, there is nothing to show.  Otherwise choose $x\in B(y,r)\cap K$.  For $0<r<1/(2N)$, \cref{eq:EMN-upper} gives
  \begin{align}
    \nu(B(y,r))
    \le \Phi(B(x,2r))
    \le 2Mr .
  \end{align}
  For $r\ge1/(2N)$, the same estimate holds after increasing the constant, since $\nu$ is finite.  Thus $\nu(B(y,r))\le Cr$ for all $y,r$.

  By the differentiation theorem, for $\nu$-almost every $x\in K_0$,
  \begin{align}\label{eq:K-density-point-absolute}
    \frac{\Phi(B(x,r)\setminus K_0)}{\Phi(B(x,r))}\to0 .
  \end{align}
  At such points, \cref{eq:EMN-upper} implies
  \begin{align}\label{eq:K-complement-over-r}
    \frac{\Phi(B(x,r)\setminus K_0)}{r}\to0 .
  \end{align}
  Combining \cref{eq:EMN-lower} with \cref{eq:K-density-point-absolute} shows that $\delta_\nu^*(x)>0$, while linear growth gives $\delta_\nu^*(x)<\infty$.  The lower-density identity follows from the second equality in \cref{eq:preiss-example-densities} and the fact that $\nu\le\Phi$.

  Finally, let $x$ be a point for which \cref{eq:preiss-alpha-zero} and \cref{eq:K-complement-over-r} both hold.  If $\varphi$ is supported in $B(x,r)$ and has Lipschitz constant at most $1$, then $\norm{φ}_{∞}\le r$.
  Since $\nu=\Phi\restriction_{K_0}$,
  \begin{align}
    F_{B(x,r)}(\Phi,\nu)
    \le r\,\Phi(B(x,r)\setminus K_0).
  \end{align}
  Dividing by $r^2$ and using \cref{eq:K-complement-over-r}, we get
  \begin{align}
    \alpha_{\nu}(x,r)
    \le \alpha_{\Phi}(x,r)
      +C\frac{\Phi(B(x,r)\setminus K_0)}{r}\to0 .
  \end{align}
  This completes the proof.
\end{proof}

\begin{lemma}[Segment square function estimate]
  There is an absolute constant $C$ such that, for every line segment $S\subset\CC$ and every $g\in L^2(\hd1\restriction_S)$,
  \begin{align}\label{eq:segment-usfe-appendix}
    \int_{\CC\setminus S}
    \abs[\Big]{\int_S \frac{g(\xi)}{(z-\xi)^2}\,d\hd1(\xi)}^2
    \dist(z,S)\,dm_2(z)
    \le C\int_S |g(\xi)|^2\,d\hd1(\xi) .
  \end{align}
\end{lemma}
\begin{proof}
  By translation, rotation, and dilation invariance, it is enough to consider $S=[0,L]\subset\RR$.  Extend $g$ by zero to all of $\RR$ and write
  \begin{align}
    Tg(z):=\int_0^L \frac{g(t)}{(z-t)^2}\,dt .
  \end{align}
  We split $\CC\setminus S$ into the vertical strip over the segment and the two endpoint regions.

  First consider the strip $\{z=x+iy:0<x<L,\ y\ne0\}$.  There $\dist(z,S)=|y|$.  For $y>0$, Plancherel's theorem and the standard Fourier transform of $(x+iy)^{-2}$ give
  \begin{align}
    \int_{\RR}|Tg(x+iy)|^2\,dx
    \lesssim
    \int_{\RR}|\xi|^2e^{-c y|\xi|}|\widehat g(\xi)|^2\,d\xi
  \end{align}
  with an absolute constant $c>0$.  Therefore
  \begin{align}
    \int_0^\infty\int_{\RR}|Tg(x+iy)|^2y\,dx\,dy
    &\lesssim
    \int_{\RR}|\widehat g(\xi)|^2|\xi|^2
      \int_0^\infty y e^{-c y|\xi|}\,dy\,d\xi  \\
    &\lesssim \int_{\RR}|\widehat g(\xi)|^2\,d\xi
    \lesssim \int_0^L |g(t)|^2\,dt .
  \end{align}
  The lower half-plane is analogous.
  Restricting the $x$-integral to $0<x<L$ gives the desired estimate in the strip.

  It remains to control the endpoint regions.  We treat $\Omega_-:=\{z:\Re z<0\}$; the region $\Omega_+:=\{z:\Re z>L\}$ is the same after the change of variables $t\mapsto L-t$ and $z\mapsto L-z$.  Write $z=\rho e^{i\theta}$ in $\Omega_-$, so that $\pi/2<\theta<3\pi/2$.  Then $\dist(z,S)=|z|=\rho$, and for $t\ge0$,
  \begin{align}
    |z-t|^2=|z|^2+t^2-2t\Re z\ge \rho^2+t^2 .
  \end{align}
  Hence
  \begin{align}
    \int_{\Omega_-}|Tg(z)|^2\dist(z,S)\,dm_2(z)
    &\lesssim
    \int_0^\infty
      \paren[\Big]{\int_0^L\frac{|g(t)|}{\rho^2+t^2}\,dt}^2\rho^2\,d\rho \\
    &=
    \int_0^\infty
      \paren[\Big]{\int_0^L\frac{\rho |g(t)|}{\rho^2+t^2}\,dt}^2\,d\rho .
  \end{align}
  The operator on $(0,\infty)$ with kernel
  \begin{align}
    K(\rho,t):=\frac{\rho}{\rho^2+t^2}
  \end{align}
  is bounded on $L^2(0,\infty)$.  Indeed, Schur's test with the weight $t^{-1/2}$ gives
  \begin{align}
    \int_0^\infty K(\rho,t)t^{-1/2}\,dt=C\rho^{-1/2},
    \qquad
    \int_0^\infty K(\rho,t)\rho^{-1/2}\,d\rho=Ct^{-1/2}.
  \end{align}
  Thus the contribution of $\Omega_-$ is bounded by $C\norm{g}_{L^2(S)}^2$, and the same estimate holds for $\Omega_+$.  Combining the strip and endpoint estimates proves \cref{eq:segment-usfe-appendix}.
\end{proof}

\begin{theorem}[USFE does not force rectifiability under $\alpha$-number flatness]\label{thm:alpha-usfe-counterexample}\label{thm:finite-support-counterexample}
  There is a compactly supported finite Radon measure $\mu$ on $\CC$, with support $E=\supp\mu$, such that
  \begin{enumerate}
    \item $\hd1(E)<\infty$;
    \item $\mu$ has linear growth;
    \item $0<\delta_\mu^*(x)<\infty$ for $\mu$-almost every $x$;
    \item $\alpha_{\mu}(x,r)\to0$ for $\mu$-almost every $x$;
    \item $\mu$ satisfies USFE;
    \item $\mu$ is not $1$-rectifiable;
    \item $\mu$ is not weakly flat.
  \end{enumerate}
  Consequently, the weak flatness assumption in \cref{thm:main-theorem} cannot be replaced by the necessary measure-theoretic condition $\alpha_{\mu}(x,r)\to0$.
\end{theorem}
\begin{proof}
  Let $\nu$ and $K=\supp\nu$ be the measure from \cref{thm:preiss-core-alpha} which will act as the core.
  We add a small amount of rectifiable dust near $K$.  Set
  \begin{align}
    r_k:=2^{-k}\quad\text{and}\quad h_k:=r_k^3 .
  \end{align}
  For each $k$, choose a finite $h_k/8$-net $\{q_{k,\ell}\}_{\ell=1}^{N_k}$ in the compact set $\overline{N_{r_k}(K)}$.  Since $\hd1(K)<\infty$, the set $K$ has empty interior.  Thus we may choose points $p_{k,\ell}\in\CC\setminus K$ with $|p_{k,\ell}-q_{k,\ell}|<h_k/8$.

  Successively choose closed line segments $S_{k,\ell}\subset\CC\setminus K$ ($K$ is closed with empty interior), centered at $p_{k,\ell}$, so small that
  \begin{align}\label{eq:alpha-dust-segment-lengths}
    S_{k,\ell}\subset B(p_{k,\ell},h_k/16),
    \qquad
    \sum_{k,\ell}\hd1(S_{k,\ell})<\infty,
  \end{align}
  and so that no two distinct segments share a non-trivial subsegment.  This is possible by choosing the lengths sufficiently small and avoiding, at each step, the countably many directions that would make the new segment overlap a previously chosen one.  Since $\nu$ is supported on $K$, each such segment has zero $\nu$-measure.  Put
  \begin{align}
    G_k:=\bigcup_{\ell=1}^{N_k}S_{k,\ell} .
  \end{align}
  Then
  \begin{align}\label{eq:alpha-dust-net-properties}
    G_k\subset N_{2r_k}(K)
    \quad\text{and}\quad
    \overline{N_{r_k}(K)}\subset N_{h_k}(G_k).
  \end{align}

  Enumerate the segments as $\{S_j\}_{j\ge1}$.  Choose positive weights $a_j$ so small that, if
  \begin{align}
    A_k:=\sum_{j:S_j\subset G_k} a_j,
  \end{align}
  then
  \begin{align}\label{eq:alpha-dust-weight-choice}
    A_k\le 2^{-k}
    \quad\text{and}\quad
    \sum_{j=1}^\infty a_j<\infty .
  \end{align}
  Define the measures
  \begin{align}
    \sigma:=\sum_{j=1}^\infty a_j\hd1\restriction{S_j}
    \quad\text{and}\quad
    \mu:=\nu+\sigma,
  \end{align}
  and let $E:=\supp\mu$.  Since $r_k\to0$ and $G_k\subset N_{2r_k}(K)$,
  \begin{align}\label{eq:alpha-dust-support}
    E
    = \overline{ K\cup\bigcup_{k=1}^\infty G_k  }
    =K\cup\bigcup_{k=1}^\infty G_k .
  \end{align}
  The set $E$ is compact, and by \cref{eq:alpha-dust-segment-lengths},
  \begin{align}
    \hd1(E)\le \hd1(K)+\sum_{k,\ell}\hd1(S_{k,\ell})<\infty .
  \end{align}

  The linear growth is immediate from the linear growth of $\nu$ and the bound $\hd1(S_j\cap B(x,r))\le2r$, which gives
  \begin{align}
    \sigma(B(x,r))\le2r\sum_{j=1}^{∞} a_j .
  \end{align}
  Hence $\mu(B(x,r))\lesssim r$ for all $x,r$.  The upper density is therefore finite everywhere.  It is positive at $\nu$-almost every point by \cref{thm:preiss-core-alpha}, and it is positive at $\sigma$-almost every point because $\sigma$ is a countable sum of weighted arclength measures on segments.  This proves $0<\delta_\mu^*<\infty$ at $\mu$-almost every point.

  We next prove that the $\alpha$-numbers vanish.  First let $x$ be a $\nu$-typical point.  Since $\nu(S_j)=0$ for every $j$, the point $x$ lies on none of the dust segments for $\nu$-almost every $x$.  For such an $x$ and fixed $N$, all sufficiently small balls centered at $x$ miss the finitely many generations $G_1,\dots,G_{N-1}$.  The contribution of all later generations satisfies
  \begin{align}
    \sigma(B(x,r))
    \le 2r\sum_{k\ge N}A_k
  \end{align}
  for all sufficiently small $r$.  Letting $N\to\infty$ and using \cref{eq:alpha-dust-weight-choice}, we get
  \begin{align}\label{eq:sigma-density-zero-nu-ae}
    \frac{\sigma(B(x,r))}{r}\to0
    \qquad\text{for }\nu\text{-a.e. }x.
  \end{align}
  Since $\alpha_{\nu}(x,r)\to0$ at $\nu$-almost every $x$, \cref{eq:sigma-density-zero-nu-ae} and the definition of $F_B$ (cf. \cref{eq:F-B}) imply
  \begin{align}
    \alpha_{\mu}(x,r)\to0
  \end{align}
  for $\nu$-almost every $x$.

  Now fix a segment $S_j$.  Since $S_j\subset\CC\setminus K$ is compact and $K$ is closed, $\dist(S_j,K)>0$. 
  $\hd1$-almost every point of $S_j$ lies in the relative interior of $S_j$ and on no other dust segment.
  Fix such a point $x$.  All sufficiently late generations are contained in an arbitrarily small neighborhood of $K$, hence miss a fixed neighborhood of $S_j$.  The remaining earlier generations form a finite family of segments, and, since $x$ lies on none of them except $S_j$, sufficiently small balls centered at $x$ avoid them.  These balls are also disjoint from $K$.  Thus, for sufficiently small r,
  \begin{align}
    \mu(B(x,r)) = a_j\hd1(S_j\cap B(x,r)).
  \end{align}
  Since, for all small enough $r$, the restriction $a_j\hd1\restriction_{S_j}$ agrees in $B(x,r)$ with a constant multiple of arclength measure on a line, this gives $\alpha_{\mu}(x,r)\to0$ for $\sigma$-almost every $x$.
  Thus the $\alpha$-number condition holds $\mu$-almost everywhere.

  The measure $\mu$ is not rectifiable because $\nu\le\mu$ and $\nu$ is not rectifiable.

  We now prove USFE.  Write
  \begin{align}
    T_\rho f(z):=\int \frac{f(\xi)}{(z-\xi)^2}\,d\rho(\xi)
    \quad\text{and}\quad
    w(z):=\dist(z,E).
  \end{align}
  For $\xi\in K$, define
  \begin{align}
    I(\xi):=\int_{\CC\setminus E}\frac{w(z)}{|z-\xi|^4}\,dm_2(z) .
  \end{align}
  We claim that $\sup_{\xi\in K}I(\xi)<\infty$.  If $r_{k+1}<|z-\xi|\le r_k$, then $z\in\overline{N_{r_k}(K)}$, and \cref{eq:alpha-dust-net-properties} gives $w(z)\le h_k$.  Consequently
  \begin{align}
    \int_{\{r_{k+1}<|z-\xi|\le r_k\}}
    \frac{w(z)}{|z-\xi|^4}\,dm_2(z)
    \lesssim \frac{h_k}{r_k^2}=r_k .
  \end{align}
  Summing in $k$ gives a finite contribution.  The region $|z-\xi|>r_1$ is harmless, because $\xi\in E$ and hence $w(z)\le |z-\xi|$.  Thus the claim follows.  By Cauchy--Schwarz and Fubini,
  \begin{align}\label{eq:alpha-dust-core-usfe}
    \int_{\CC\setminus E}|T_\nu f(z)|^2w(z)\,dm_2(z)
    \le \norm{f}_{L^2(\nu)}^2\int_K I(\xi)\,d\nu(\xi)
    \lesssim \norm{f}_{L^2(\nu)}^2 .
  \end{align}

  For the dust part, let $\sigma_j=a_j\hd1\restriction_{S_j}$.  Since $w(z)\le\dist(z,S_j)$ and $\CC\setminus E\subset\CC\setminus S_j$, the segment estimate \cref{eq:segment-usfe-appendix}, applied to the measure $\sigma_j$, gives
  \begin{align}
    \norm{T_{\sigma_j} f}_{L^2(w\,dm_2)}
    \lesssim a_j^{1/2}\norm{f}_{L^2(\sigma_j)} .
  \end{align}
  By Minkowski's inequality, first for finite partial sums and then by passing to the limit in $L^2(w\,dm_2)$,
  \begin{align}\label{eq:alpha-dust-sigma-usfe}
    \norm{T_\sigma f}_{L^2(w\,dm_2)}
    &\le \sum_j \norm{T_{\sigma_j} f}_{L^2(w\,dm_2)} \\
    &\lesssim \sum_j a_j^{1/2}\norm{f}_{L^2(\sigma_j)} \\
    &\le \paren[\Big]{ \sum_j a_j }^{1/2}
      \paren[\Big]{ \sum_j \norm{f}_{L^2(\sigma_j)}^2 }^{1/2}
    \lesssim \norm{f}_{L^2(\sigma)} .
  \end{align}
  Since $T_\mu f=T_\nu f+T_\sigma f$, combining \cref{eq:alpha-dust-core-usfe} and \cref{eq:alpha-dust-sigma-usfe} gives USFE for $\mu$.

  Finally, we verify that weak flatness fails.  Let $x\in K$ be a point with $\delta_\nu^*(x)>0$.  Choose radii $s_i\downarrow0$ such that $\nu(B(x,s_i))\ge c_xs_i$ for some $c_x>0$.  For each $i$, choose $k_i$ with $r_{k_i}\le s_i<2r_{k_i}$.  Then
  \begin{align}\label{eq:alpha-dust-density-lower-weak-flatness}
    \frac{\mu(B(x,2r_{k_i}))}{2r_{k_i}}
    \ge \frac{c_x}{2} .
  \end{align}
  On the other hand, for any line $L$ there is a point $y\in B(x,r_{k_i}/2)$ with $\dist(y,L)\ge c r_{k_i}$, where $c>0$ is absolute.  Since $y\in N_{r_{k_i}}(K)$, \cref{eq:alpha-dust-net-properties} gives a point $e\in G_{k_i}\subset E$ with $|e-y|<h_{k_i}$.  For large $i$, $h_{k_i}\ll r_{k_i}$, so $e\in B(x,r_{k_i})$ and $\dist(e,L)\ge c r_{k_i}/2$.  Taking the infimum over $L$ gives
  \begin{align}
    \beta_E(x,2r_{k_i})\ge c_0>0 .
  \end{align}
  Together with \cref{eq:alpha-dust-density-lower-weak-flatness}, this proves
  \begin{align}
    \limsup_{r\to0}\beta_E(x,r)\delta_\mu(x,r)>0
  \end{align}
  for $\nu$-almost every $x$.  Thus $\mu$ is not weakly flat.
\end{proof}

\begin{rem}\label{rem:jordan-curve-version}
  Containment in a Jordan curve alone does not replace the local two-sided NTA hypothesis in \cref{thm:local-nta-rectifiable}.
  Indeed, if one does not require the $\alpha$-number conclusion in \cref{thm:alpha-usfe-counterexample}, the same support-dust mechanism may be applied to a four-corner Cantor set core, and the auxiliary segments may be chosen so that the resulting support is contained in a Jordan curve.
  That variant satisfies the density assumption and USFE and has finite $\hd1$ support, but it is not rectifiable.
\end{rem}

\printbibliography

\end{document}